\documentclass[11pt,reqno]{amsart}    

\usepackage[margin=1in]{geometry}
\usepackage{amsfonts,amsmath,amssymb,enumerate}    
\usepackage{amsthm, bm}    
\usepackage{mathrsfs}
\usepackage{setspace}
\usepackage{comment}
\usepackage{tikz}
\usetikzlibrary{arrows,matrix}

\tikzset{cross/.style={cross out, draw=black, minimum size=2*(#1-\pgflinewidth), inner sep=0pt, outer sep=0pt}, cross/.default={1pt}}

\numberwithin{equation}{section}
\usepackage[normalem]{ulem}
\theoremstyle{plain}
\newtheorem{thm}{Theorem}[section]
\newtheorem{lem}[thm]{Lemma}
\newtheorem{prop}[thm]{Proposition}
\newtheorem{cor}[thm]{Corollary}

\newtheorem*{prop*}{Proposition}
\theoremstyle{definition}
\newtheorem{defn}[thm]{Definition}

\theoremstyle{remark}
\newtheorem{rem}[thm]{Remark}

\newtheorem{exmp}[thm]{Example}

\newtheorem*{rem*}{Remark}
\newtheorem*{ack}{Acknowledgements}

\newtheorem*{exmp*}{Example}

\newcommand{\Wedge}[1]{\bigwedge\nolimits^{\!#1}}

\newcommand{\be}{\begin{equation}}    
\newcommand{\ee}{\end{equation}}    
\newcommand{\beu}{\begin{equation*}}    
\newcommand{\eeu}{\end{equation*}}    
\newcommand{\bea}{\begin{eqnarray}}    
\newcommand{\eea}{\end{eqnarray}}    
\newcommand{\beaa}{\begin{eqnarray*}}    
\newcommand{\eeaa}{\end{eqnarray*}}    
\newcommand{\bmx}{\begin{pmatrix}}    
\newcommand{\emx}{\end{pmatrix}}

\newcommand{\del}{\partial}    
\newcommand{\g}{{\mathfrak g}}

\newcommand{\no}{\n}
\newcommand{\sib}{{\tau}}
\newcommand{\nos}{\n^\sio}

\newcommand{\sio}{\sigma}
\newcommand{\siod}{\sigma}

\newcommand{\n}{{\mathfrak n}}
\newcommand{\ns}{\n^\sigma}
\newcommand{\nb}{\mf a}
\newcommand{\nbs}{\bar\nb}

\newcommand{\Mb}{\mathcal M_{0,N}}
\newcommand{\Mbnz}{\mathcal M_N}
\newcommand{\cca}{\cc_{0,N;m}}
\newcommand{\Mo}{\mathsf M_{0,N}}
\newcommand{\ccI}[1]{\cc_{0,N;#1}}

\newcommand{\cco}[1]{\cc_{#1}}
\newcommand{\ccoo}[1]{\cc_{0,#1}}

\newcommand{\clc}[2]{#1_{#2}}

\newcommand{\f}[2]{\,{}^{#1\!} f_{#2}}
\newcommand{\e}[2]{\,{}^{#1\!} e_{#2}}
\newcommand{\hh}[2]{\,{}^{#1\!} h_{#2}}

\newcommand{\h}{{\mathfrak h}}

\newcommand{\mf}{\mathfrak}
\newcommand{\mc}{\mathcal}

\newcommand{\gs}{\g^{\sio}}    
\newcommand{\hs}{\h^{\si}}    
\newcommand{\gt}{{\widetilde \g}}
\newcommand{\gts}{\gt^\si}
\newcommand{\nt}{{\widetilde \n}}
\newcommand{\nts}{\nt^\si}

\newcommand{\half}{\frac{1}{2}}

\newcommand{\nn}{\nonumber}

\newcommand{\eps}{\epsilon}

\newcommand{\ad}{{\rm ad}}

\newcommand{\Stab}[1]{\mathrm{Stab}({#1})}

\newcommand{\Z}{{\mathbb Z}}

\newcommand{\C}{{\mathbb C}}

\renewcommand{\P}{{\mathcal P}}

\newcommand{\id}{{\mathrm{id}}}    
\newcommand{\bl}{{\bullet}}    
\newcommand{\wh}{{\circ}}

\newcommand{\on}{.}    

\newcommand{\Cx}{\mathbb C^\times}

\newcommand{\btp}{\begin{tikzpicture}[baseline=0pt,scale=0.9,line width=0.25pt]}    
\newcommand{\etp}{\end{tikzpicture}}

\newcommand{\atp}[1]{}

\newcommand{\at}[2]{\underset{\substack{$ $ \\ \tikz{\draw[->] (0,0) -- (0,.13); } \\[-.2mm] {#1}}}{\smash{#2}}}

\newcommand{\path}{\longrightarrow}

\DeclareMathOperator{\wt}{wt}

\DeclareMathOperator{\Span}{span}

\newcommand{\W}{\C^m}

\newcommand{\Ht}[3]{{}^{#3} \!H_{#1,#2}}
\newcommand{\Hz}[3]{{}^{#3} \!H_{#1}^{#2}}
\newcommand{\Ho}[1]{H^0_{#1}}

\newcommand{\kv}{{}^{\bm k}\!}
\newcommand{\kvp}{{}^{\bm k'}\!}

\newcommand{\K}{\mathsf k}
\newcommand{\F}{\mc F}

\newcommand{\lsigma}{\sigma}

\newcommand{\ie}{\textit{i.e. }}

\DeclareMathOperator{\codim}{codim}
\DeclareMathOperator{\Flag}{Flag}
\newcommand{\Fl}{\mathscr{F}}
\newcommand{\A}{\mathscr{A}}
\newcommand{\cc}{\mathcal C}

\newcommand{\cmfn}{\Phi}

\newcommand{\ZT}{\Z_T}

\let\eps\varepsilon

\let\phi\varphi
\let\si\sigma

\newcommand{\Geom}{G}
\newcommand{\Ghom}{\mathcal G}
\newcommand{\mones}{{[1^m]}}

\newcommand{\ru}{R}
\newcommand{\rf}{r}
\newcommand{\sym}{\mathsf s}
\newcommand{\cai}{\varphi}

\newcommand{\psis}{\psi}
\newcommand{\oo}{\mathcal O}
\newcommand{\into}{\hookrightarrow}

\newcommand{\td}{\check}

\newcommand{\Pin}{\overline}

\newcommand{\wtI}[1]{[#1]}
\newcommand{\gu}{\mf b}
\newcommand{\gus}{\bar{\mf b}}

\newcommand{\Sa}{S}

\newcommand{\nz}{{\neq 0}}
\newcommand{\Co}{\Omega}
\newcommand{\Gm}{\Geom_0}

\newcommand{\Se}{\Sa}
\newcommand{\Cou}{\Co_{\leq 1}}

\newcommand{\Ip}{{\{1,\dots,m\}\setminus I}}
\newcommand{\Kill}{K}
\newcommand{\gse}{\dot\g^\si}

\newcommand{\Sg}{S_\g}

\newcommand{\Jp}{{\{1,\dots,m\} \setminus J}}

\newcommand{\I}{{\{1,\dots,\ru\}}}
\newcommand{\Is}{{\{1,\dots,\rf\}}}

\newcommand{\Shom}{{\mathcal S}}
\newcommand{\Sw}{\Sigma\!\Z}
\newcommand{\Swl}{{\Sigma\!\Z_\lambda}}

\newcommand{\bs}{b^\si}
\newcommand{\Bs}{B^\si}

\newcommand{\Hess}{\mathrm{Hess}}
\newcommand{\HH}{\Ht {\clc in}{\clc jp}k}

\let\nc\newcommand
\nc{\bean}{\begin{eqnarray}}
\nc{\eean}{\end{eqnarray}}
\nc{\Ref}[1]{{\rm(\ref{#1})}}

\let\Hat\widehat

\let\ox\otimes
\let\wx\wedge

\let\bra\langle
\nc\ket{\rangle\,}

\author{Alexander Varchenko}
\address{
AV: Department of Mathematics, University of North Carolina at Chapel Hill, Chapel Hill, NC 27599-3250, USA.}
\email{anv@email.unc.edu}  
 
\author{Charles Young}
\address{
CY: School of Physics, Astronomy and Mathematics, University of Hertfordshire, College Lane, Hatfield AL10 9AB, UK.}

\email{charlesyoung@cantab.net}

\makeatletter
\def\Ddots{\mathinner{\mkern1mu\raise\p@
\vbox{\kern7\p@\hbox{.}}\mkern2mu
\raise4\p@\hbox{.}\mkern2mu\raise7\p@\hbox{.}\mkern1mu}}
\makeatother
\begin{document} 
\title[Cyclotomic discriminantal arrangements]{Cyclotomic discriminantal arrangements\\ and diagram automorphisms of Lie algebras}

\begin{abstract} 
Recently a new class of quantum integrable models, the \emph{cyclotomic Gaudin models}, were described in \cite{VY1,VY2}. Motivated by these, we identify a class of affine hyperplane arrangements  that we call \emph{cyclotomic discriminantal arrangements}. We establish correspondences between the flag and Aomoto complexes of such arrangements and chain complexes for nilpotent subalgebras of Kac-Moody type Lie algebras with diagram automorphisms.

As a byproduct, we show that the Bethe vectors of cyclotomic Gaudin models associated to diagram automorphisms are nonzero.

\end{abstract}
\maketitle
\setcounter{tocdepth}{1}
\tableofcontents

\section{Introduction and overview}
It is known that the theory of Knizhnik-Zamolodchikov (KZ) equations and Gaudin models  is closely related with the theory of arrangements of hyperplanes.  For example, KZ equations were solved in multidimensional hypergeometric integrals associated with a family of discriminantal arrangements \cite{SV};  the Kohno-Drinfeld theorem \cite{Kohno,D2,D1}, describing the monodromy of KZ equations in terms of quantum groups, was given a geometric proof in \cite{VarchenkoBook}, where the homology groups of the complements to the discriminantal arrangements were described in terms of quantum groups;  the Bethe vectors in the Gaudin models were constructed in terms of the combinatorics of discriminantal arrangements and were labeled by critical points of master functions associated with these arrangements \cite{BabujianFlume, RV, SV}.

The foundations for these relations were laid in  \cite{SV}. In that paper, discriminantal arrangements of hyperplanes were defined, and the geometric objects of those arrangements were related to homology of nilpotent subalgebras of Kac-Moody type Lie bialgebras. 

Recently a class of \emph{cyclotomic} Gaudin models was introduced in \cite{VY1,VY2}. We expect that the cyclotomic models will also have close relations with the theory of arrangements of hyperplanes. In this paper we take the first step in this direction. We identify a new class of arrangements of hyperplanes that we call \emph{cyclotomic discriminantal arrangements}, and establish correspondences between the flag and Aomoto complexes of such arrangements and chain complexes for nilpotent subalgebras of Kac-Moody type Lie algebras \emph{with diagram automorphisms}.
Let us describe our results in detail.

\subsection{Cyclotomic discriminantal arrangements}
Let $\omega\in\Cx$ be a primitive root of unity of order $T\in \Z_{\geq 1}$. 
Fix coordinates $t_1,\dots,t_m$ on $\C^m$ and let $\cca$ denote the arrangement (\ie collection of affine hyperplanes) in $\C^m$ consisting of:
\begin{alignat}{2}
\Ht ijk &: t_i - \omega^k t_j = 0 ,\quad &&1\leq i<j \leq m,\quad k\in \ZT, \nn\\
\Hz ijk &: t_i - \omega^k z_j = 0, \quad &&1\leq i\leq m,\quad 1\leq j \leq N,\quad k\in \ZT,\nn\\
 \Ho i &: t_i = 0 , \quad &&1\leq i \leq m, \nn\end{alignat}
where $z_1,\dots,z_N\in \Cx$ are $N\in \Z_{\geq 0}$ nonzero points whose orbits under the action of the cyclic group $\omega^\Z\cong \ZT := \Z/T\Z$ are pairwise disjoint. 
In the special case $T=1$ such hyperplane arrangements are called \emph{discriminantal arrangements}. We shall  call $\cca$ a \emph{cyclotomic discriminantal arrangement}. 

Given any hyperplane arrangement $\cc$ in $\C^m$, one has the \emph{Orlik-Solomon algebra} $\A^\bl(\cc)= \bigoplus_{p=0}^m \A^p(\cc)$: it can be defined as the $\C$-algebra of differential forms generated by $1$ and the one-forms $d\log l_H= d l_H /l_H$, where, for each hyperplane $H\in \cc$, $l_H=0$ is an affine equation for $H$. One has also the \emph{flag space} $\Fl^p(\cc)$, for $p=0,1,\dots,m$. It is a certain quotient of the $\C$-span of all those flags $L^0 \supset \dots \supset L^p$ in which each $L^i$ is an edge (\ie non-empty intersection of hyperplanes) of $\cc$ of codimension $i$. 
There are canonical isomorphisms  $\Fl^{p}(\cc)^* \cong \A^{p}(\cc)$ for each $p$. (See \S{\ref{sec: arrangements}.)

The Poincare polynomial which encodes the dimensions of the flag spaces/Orlik-Solomon spaces for the cyclotomic discriminantal arrangement $\cca$ is given by
\be P(x) :=  \sum_{p=0}^m  x^{p} \dim(\A^p(\cca)) 
= \prod_{p=0}^{m-1} \Big(1+\big(1+(p+N)T\big)x\Big).\nn\ee
See Theorem \ref{Pointhm}. When $N=0$ one recovers the arrangement of a complex reflection group, whose Poincare polynomial is known from  \cite[\S4]{OS}.

There is a canonical differential $d: \Fl^p(\cc) \to \Fl^{p+1}(\cc)$, given by extending flags in all possible ways. This gives the \emph{flag complex}, $(\Fl^\bl(\cc), d)$. 
Our first result gives an algebraic description of the flag complex $(\Fl^\bl(\cca),d)$ of the cyclotomic discriminantal arrangement. 
Let $\nb$ denote the free Lie algebra in generators $\f ki$, $i=1,\dots,m$, $k\in \ZT$, and let $\nbs$ denote the fixed-point subalgebra under the automorphism $\nb\to\nb$ defined by $\f ki \mapsto \f{(k+1\!\!\!\mod T)}i$. 
For any $\nbs$-module $M$ we have the standard chain complex $(C_\bl(\nbs,M),d)$ whose spaces are $C_p(\nbs,M) := \Wedge p \nbs \ox M$. Consider taking 
$M$ to be the module 
\be\Mb := U(\nbs) \ox U(\nb)^{\ox N},\nn\ee 
where $U$ denotes universal envelope. 
Let $(\Wedge p \nbs \ox \Mb)_{\mones}$ denote the subspace spanned by terms in which each $\f *i$ appears exactly once, for $i=1,\dots,m$. We show the following; see Theorem \ref{p1}. 
\begin{thm}
There are linear isomorphisms 
\be \psis_p : \left(\Wedge p \nbs \ox \Mb\right)_{\mones} \to \Fl^{m-p}(\cca), \qquad 0\leq p\leq m,\nn\ee 
and they define an isomorphism of complexes 
\be \psis_\bl : (C_\bl(\nbs ,\Mb )_{\mones},d)\to (\Fl^{m-\bl}(\cca),d).\nn\ee
\end{thm}
The way these isomorphisms work is best understood by inspecting Example \ref{p1ex}.

(To prove the dimension formulas above, we construct dual bases of $\Fl^p(\cca)$ and $\A^p(\cca)$. Pulling back, we get  bases of $\left(\Wedge p \nbs \ox \Mb\right)_{\mones}$, and in particular of $U(\nbs)_\mones$. These bases are labelled by ``decorated'' \emph{Lyndon words}. See \S\ref{lyndon}.)

\subsection{Symmetrization of flags}
Now let $\no$ be the free Lie algebra over $\C$ in generators $F_i$, $i=1,\dots,\ru$ for some $\ru\in \Z_{\geq 1}$. For any  permutation $\siod\in \Sigma_{\ru}$ there is an automorphism $\sio:\no\to \no$ defined by $\sio(F_i) = F_{\siod(i)}$. 
Let $\nos\subset \no$ denote the fixed-point subalgebra. We have the $\nos$-module 
\be \Mo := U(\nos) \ox U(\no)^{\ox N}.\nn\ee
There is a ``weight'' gradation 
of the spaces $\Wedge p \nos \ox \Mo$ 
of the chain complex $(C_\bl(\nos, \Mo),d)$, given by counting the number of $F_i$ from each $\sio$-orbit.
Fix any weight $\lambda = (\lambda_1,\dots,\lambda_\rf)\in \Z_{\geq 0}^{\rf}$, where $\rf$ is the number of orbits. Let $m=\lambda_1+\dots+\lambda_\rf$ and let $T$ be the order of $\sio$.
We construct a linear ``cyclotomic symmetrization'' map, $\sym: \left(\Wedge p \nos \ox \Mo\right)_\lambda\to \left(\Wedge p \nbs \ox \Mb\right)_{\mones}$. (Example \ref{ex2} illustrates how this map works.) Combining this with the theorem above we get an identification between vectors in  $\left(\Wedge p \nos \ox \Mo\right)_\lambda$ and certain ``cyclotomically symmetrized'' linear combinations of flags in the flag space $\Fl^{m-p}(\cca)$ of the cyclotomic discriminantal arrangement. More precisely, there is a finite group $\Swl$ -- a certain semi-direct product of symmetric and cyclic groups -- depending on the weight $\lambda$ and also the automorphism $\siod$. It acts naturally on the flag spaces, and among the isotypical components for its action is one we label $\Fl^{m-p}(\cca)^\Swl$. We establish the following in Theorem \ref{symthm}.
\begin{thm}\label{t1i}
There are linear isomorphisms
\be (\psis_p \circ \sym) : \left(\Wedge p \nos \ox \Mo\right)_\lambda \to \Fl^{m-p}(\cca)^\Swl, \quad p=0,1,\dots,m .\nn\ee
They define an isomorphism of complexes
\be (\psis_{\bl}\circ \sym) : C_\bl(\nos,\Mo)_\lambda \to \Fl^{m-\bl}(\cca)^\Swl .\nn\ee
\end{thm}

\subsection{Main result: Shapovalov and geometric forms}
Let $\ru\in \Z_{\geq 0}$. Suppose we fix the following data:
\begin{enumerate}[(1)]
\item A finite-dimensional complex vector space $\h$;
\item A non-degenerate symmetric bilinear form $(\cdot,\cdot): \h\times \h\to \C$;
\item A collection $\alpha_1,\dots,\alpha_\ru\in \h^*$ of linearly independent elements, called the \emph{simple roots}.
\end{enumerate}
This defines a symmetric $\ru\times\ru$ matrix $B= \left( (\alpha_i,\alpha_j)\right)_{i,j\in \I}$.

To the data (1--3) is associated a Lie algebra $\g = \g(B)$, which is roughly-speaking a ``Kac-Moody algebra without Serre relations''. See \S\ref{sec: cd}. One has $\g = \n \oplus \h \oplus \n_+$, and $\g$ is generated by $\h$ together with generators $E_i\in \n_+$ and $F_i\in \n$, $i=1,\dots,\ru$. Now we suppose we have, in addition to (1--3), 
\begin{enumerate}[(4)]
\item a \emph{diagram automorphism} of $\g$;
\end{enumerate} 
namely, an automorphism $\sio : \g \to \g$ such that $\sio(E_i) = E_{\siod(i)}$ and $\sio(F_i)= F_{\siod(i)}$, where $\siod\in \Sigma_\ru$ is a permutation such that $(\alpha_i,\alpha_j) = (\alpha_{\siod(i)},\alpha_{\siod(j)})$ for each $i,j$.\footnote{Cf. \cite{Enriquez,Brochier} where the automorphism is inner, \ie fixes $\h$ pointwise.}

Let us also fix weights $\Lambda_0\in (\h^\si)^*$ and $\Lambda_1,\dots,\Lambda_N\in \h^*$ and let $M(\bm \Lambda)$ denote the tensor product of Verma modules with these highest weights:
\be M(\bm \Lambda) := M^\si(\Lambda_0) \ox \bigotimes_{i=1}^N M(\Lambda_i),\nn\ee
the first factor over $\gs$, the rest over $\g$. As $\ns$-modules, $M(\bm \Lambda)\cong \Mo$. So we are in the setting of Theorem \ref{t1i}. 

Now, the Cartan and weight data, \ie the numbers  
\be  \left( (\alpha_i,\alpha_j)\right)_{i,j\in \I}\quad\text{and}\quad \left((\alpha_i,\Lambda_j)\right)_{i\in \I, j\in\{0,1,\dots,N\}} \label{weightdata}\ee 
define a bilinear form, the \emph{Shapovalov form} $S$, on the spaces $\Wedge p \ns \ox M(\bm \Lambda)$; see \S\ref{sec: cd}. This form respects the weight decomposition, so it can be regarded as a linear map
\be \Shom^p: \left(\Wedge p \ns \ox M(\bm \Lambda)\right)_\lambda \to \left(\Wedge p \ns \ox M(\bm \Lambda)\right)_\lambda^*\nn\ee
sending each weight subspace to its dual.

On the other hand, the numbers \eqref{weightdata} also naturally define a \emph{weighting} of the cyclotomic discriminantal arrangement $\cca$; that is, an assignment to every hyperplane 
of a number, its weight. 
For example the weights of the hyperplanes with equations $t_{i}-\omega^kt_j=0$ are given by the numbers $(\alpha_i,\sigma^k \alpha_j)$; see \eqref{HdC} and \eqref{rhw} for precise details. 
The weighting of the arrangement defines a bilinear form on the flag spaces, which we will call the \emph{geometric form}, $\Geom$; see \S\ref{sec: arrangements}. It can be regarded as a linear map
\be \Ghom^{m-p} : \Fl^{m-p}(\cca) \to \A^{m-p}(\cca) \nn\ee
from each flag space to its dual. 

One of the main results of \cite{SV} was that, in the non-cyclotomic case ($\sio = \id$, $\omega =1$, $T=1$) the Shapovalov form $S$ essentially coincides with the pull-back $(\psis\circ \sym)^*(G)$ of the geometric form by the isomorphisms of Theorem \ref{t1i}. This statement does \emph{not} hold in the cyclotomic setting in general (though it does turn out to hold in some important special cases related to finite-dimensional semisimple Lie algebras, as we discuss in \S\ref{cis}). However, we do have the following natural generalization, which is the main result of the present paper.
\begin{thm}\label{mti}
There exists 
a central extension $\gse$ of the fixed-point subalgebra $\gs$, by a one-dimensional centre $\C\K$, such that if we let 
\be M^\si(\Lambda_0) := U(\gse) \ox_{U(\h^\si \oplus \n_+^\si) \oplus \C\K} \C v_0 \nn\ee
be the Verma module over this central extension $\gse$ (with $\K v_0 := v_0$) \emph{then} the Shapovalov form agrees with the pull-pack of the geometric form. More precisely, then the following diagram commutes for each $p$:
\be\begin{tikzpicture}    
\matrix (m) [matrix of math nodes, row sep=4em,    
column sep=10em, text height=1ex, text depth=1ex]    
{     
C_p(\ns,M(\bm \Lambda))_{\lambda} &  C_p(\ns, M(\bm \Lambda))_{\lambda}^* \\    
\Fl^{m-p}(\cca)^\Swl  & \A^{m-p}(\cca)^\Swl,  \\    
};    
\path[->,shorten <= 2mm,shorten >= 2mm]    
(m-1-1) edge node [above] {$\Shom^p$} (m-1-2)    
(m-2-1) edge node [above] {$(-1)^{m-p} T^p \Ghom^{m-p}$} (m-2-2);    
\path[->,shorten <= 0mm,shorten >= 2mm]    
(m-1-1) edge node [right] {$\sim$} node [left] {$\psis_p\circ\sym$} (m-2-1)    
(m-1-2) edge node [right] {$\sim$} node [left] {$(\pi\circ\psis_p^{-1})^*$} (m-2-2)   ;
\end{tikzpicture}\nn\ee  
where $\Shom^p$ is defined with respect to $\gse$. 
\end{thm}
(See Theorem \ref{mt}, and \S\ref{sfd}--\ref{waf} for the definition of $\Shom^p$. The map $\pi$ is the inverse of the symmetrization-of-flags map $\sym$, up to a nonzero constant of proportionality; see Lemma \ref{isolem}.)

\subsection{The definition of the extension $\gse$ and the cocycle $\Co$}\label{coi}
Let us explain the key difference between the usual and cyclotomic discriminantal arrangements which gives rise to this central extension of the fixed-point subalgebra $\gs$.
It is enough to consider the case $p=0$ and $N=0$, in which case the isomorphism of Theorem \ref{t1i} is a map
\be M^\si(\Lambda_0)_\lambda \xrightarrow\sim \F^m(\cc_{0;m})^\Swl \nn\ee
from a weight subspace of the Verma module $M^\si(\Lambda_0)\cong_{\ns} U(\ns)$ to the space of (suitably symmetrized) full flags in the arrangement $\cc_{0;m}$ in $\C^m$ consisting of only the hyperplanes $\Ht ijk: t_i - \omega^k t_j=0$ and $\Ho i: t_i=0$. 

Suppose we set $\Lambda_0=0$.  Then (it is easy to show that) the Shapovalov form vanishes on every weight subspace except the highest weight space $M^\si(\Lambda_0)_0 = \C v_0$. So it is zero whenever $m>0$. (Recall $m=\lambda_1+\dots+\lambda_\rf$.) 

On the geometric side, setting $\Lambda_0=0$ means assigning weight zero to the hyperplanes $\Ho i$. 
Nonetheless, the geometric form $G$ is not always identically zero. Consider for example $m=2$ and the full flag 
\be F =\big(\C^2 \supset (t_1=t_2) \supset (t_1=t_2=0)\big)
.\nn\ee 
By definition, \S\ref{sg}, we have
$G(F,F) = \sum_{(H_1,H_2)} a(H_1) a(H_2)$
where the sum is over all unordered pairs of hyperplanes of the arrangement such that $F = (\C^2 \supset H_1 \supset H_1 \cap H_2)$ in $\Fl^2(\cc_{0;2})$, and where $a(H_1)$, $a(H_2)$ are the weights of these hyperplanes.
Clearly we must set $H_1 = \Ht 12 0$. But then $H_2$ can be any of $\Ho 1$, $\Ho 2$, \emph{or} $\Ht 12 k$ for any $k\in \ZT\setminus \{0\}$. 
Thus, when $a(\Ho 1)=a(\Ho 2)=0$ we still have 
\be G(F,F) =  a(\Ht 120)\sum_{k\in \ZT\setminus \{0\}} a(\Ht 12 k)\nn\ee
which is trivially zero if $T=1$ (the non-cyclotomic situation) but  not zero in general.

This observation suggests (what turns out to be) the correct definition of the central extension $\gse$. 
Namely, we define a certain skew-symmetric bilinear map $\Co : \gs \times \gs \to \C$ on the Lie algebra $\gs$ in terms of the geometric bilinear form for the flag space $\Fl^m(\cc_m)$ of full flags for the arrangement consisting of \emph{only} the hyperplanes $\Ht ijk: t_i-\omega^k t_j = 0$. See \S\ref{gmsec}-\ref{coudef} and \S\ref{cocyclesec}. We prove that $\Co$ is in fact a cocycle. It is this cocycle which defines the extension $\gse$ of $\gs$. We go on to show that the ``extra'' terms in the geometric form always organize themselves in such a way that Theorem \ref{mti} holds.

\subsection{Special cases: finite-type, Kac-Moody, and Borcherds Lie algebras}\label{cis}
Let $\gt:= \g/\ker \Sg$ denote the quotient of $\g$ by the kernel of the (usual) Shapovalov form. When the matrix $B$ is the symmetrization of a symmetrizable generalized Cartan matrix then the quotient $\gt$ is the Kac-Moody Lie algebra associated to $B$ \cite{KacBook}.

With Theorem \ref{mti} in hand, we get a number of corollaries.

First, whenever the fixed-point subalgebra $\gts$ of $\gt$ is finite-dimensional and semisimple then the usual Shapovalov and geometric forms do coincide just as in the non-cyclotomic situation. (See Corollary \ref{sscor}.)
Indeed, our cocycle $\Co$ defining the central extension $\gse$ vanishes on $\ker\Sg \cap\gs\subset \gs$  (Proposition \ref{kercoprop}). Therefore it defines a central extension of the quotient $\gts$ too. But recall that finite-dimensional semisimple Lie algebras do not admit non-trivial central extensions, because Whitehead's lemma states that every cocycle is coboundary. In fact  (Proposition \ref{cbprop}) whenever our cocycle $\Co$ is coboundary it is actually zero. 

Thus, for example, our cocycle $\Omega$ vanishes for all foldings of simply-laced finite-type Dynkin diagrams by diagram automorphisms.

Actually we prove a statement which shows that $\Omega=0$ whenever $\gs$ ``is a folding'' in a looser sense. Namely, there is always a subalgebra $\g(\Bs)\subset \gs$ generated by the projections of the generators $E_i,F_i,\h$ of $\g$. We show (Corollary \ref{eqcor})  that whenever $\g(\Bs) = \gs$ then $\Omega =0$. 
(One can ask how, in practice, the ``extra'' terms in the geometric form, as in \S\ref{coi}, can fail to contribute in such cases. In examples, one finds seemingly rather subtle cancellations coming from the symmetrization of flags.)

The ``typical'' situation, though, is that the fixed-point subalgebra $\gs$ is not even finitely generated, and (therefore) is not a Kac-Moody algebra. However it is always a \emph{generalized Kac-Moody} or \emph{Borcherds} algebra \cite{Borcherds88}: see Remark \ref{borrem}.  So one can think of our construction as singling out a certain preferred one-dimensional central extension of any Borcherds Lie algebra obtained as a fixed-point subalgebra under a diagram automorphism.

\subsection{Canonical element and weight function}
In \S\ref{cosec} we apply Theorem \ref{mti} to prove results about the \emph{weight function} and \emph{Bethe vectors} for weighted cyclotomic discriminantal arrangements. Let us describe these objects briefly. 

We construct, as mentioned above, a dual pair of bases for the flag spaces $\Fl^p(\cca)$ and Orlik-Solomon spaces $\A^p(\cca)$, $p=0,1,\dots,m$. In particular we get a dual pair of bases of $\Fl^m(\cca)$ and $\A^m(\cca)$, and hence an expression for the canonical element $\Theta\in \Fl^m(\cca)\ox \A^m(\cca)$. 
By the isomorphisms of Theorem \ref{mti} that gives an element of $M(\bm\Lambda)_\lambda \ox \A^m(\cca)^\Swl$. From it, one constructs a rational map $\Psi: \C^m \to M(\bm\Lambda)_{\lambda} \to L(\bm\Lambda)_{\lambda}$, where $L(\bm\Lambda)_{\lambda} := M(\bm\Lambda)_\lambda/\ker \Shom^p \cong \Ghom( \Fl^m(\cca)^\Swl)$, again by Theorem \ref{mti}. This map $\Psi : \C^m \to L(\bm\Lambda)_\lambda$ is the \emph{(cyclotomic) weight function}.

The weighting of the hyperplane arrangement $\cca$ defines the \emph{master function} $\Phi$, \eqref{MF}. The evaluation of the map $\Psi$ at any critical point $p$ of $\Phi$ is called the \emph{Bethe vector} associated to that critical point. Making use of results from \cite{V11} we show that the Bethe vectors corresponding to isolated critical points are nonzero, under certain conditions. See Theorem \ref{bvthm}. 

In \cite{VY1,VY2} a \emph{cyclotomic Gaudin model} was constructed and solved by Bethe ansatz.  
Our construction of the canonical element $\Theta$ here is chosen in such a way that the weight function $\Psi$ coincides with the weight function $\psi_\Gamma$ of \cite{VY1}. So Theorem \ref{bvthm} establishes that the Bethe vectors of \cite{VY1} corresponding to isolated solutions of the cyclotomic Bethe equations are nonzero (in the case of diagram automorphisms, and if the conditions of Theorem \ref{bvthm} hold).

\subsection{Structure of the paper} After recalling in \S\ref{sec: arrangements} some facts from \cite{SV} about general hyperplane arrangements, in \S\ref{sec: cda} we define the cyclotomic discriminantal arrangements and give their flag relations explicitly. Then \S\ref{fla} and \S\ref{bf} are devoted to the situation ``before symmetrization'': that means working with a free Lie algebra in which each orbit, of the automorphism acting on the generators, has the same length, $T$, and working only in the subspace of weight $(1,\dots,1)$.
Finally in \S\ref{symsec} and \S\ref{sec: cd} we can consider the general situation: in \S\ref{symsec} we prove Theorem \ref{t1i} and then in \S\ref{sec: cd} we prove Theorem \ref{mti}. In \S\ref{sec: cop} we prove properties (in particular vanishing properties) of the cocycle $\Co$. Finally in 
\S\ref{cosec} we apply the results of the paper to prove results about the weight function of cyclotomic Gaudin models.

\begin{ack} 
The research of AV is supported in part by NSF grant DMS-1362924 and the Simons Foundation. 
AV thanks the Max Planck Institute, Bonn for hospitality during the preparation of this paper, and the School of Physics, Astronomy and Mathematics at the University of Hertfordshire for hospitality during a visit in June 2015. 
CY thanks the Department of Mathematics at UNC Chapel Hill for hospitality during his visit in August 2015.
\end{ack}
\section{Hyperplane arrangements}\label{sec: arrangements}
We recall some facts about hyperplane arrangements, Orlik-Solomon algebras and flag complexes from \cite{SV}. 
Let $m$ be a positive integer and let $\cc$ be a finite collection of affine hyperplanes in $\C^m$. We call $\cc$ an \emph{arrangement} in $\C^m$. 

\subsection{Edges, flags, and the flag complex}
An \emph{edge} of the arrangement $\cc$ is a non-empty intersection of its hyperplanes. 
For $k=0,\dots,m$, let $\Flag^k(\cc)$ denote the set of all flags 
\be\W=L^0 \supset L^1 \supset\dots \supset L^k\nn\ee 
with each $L^j$ an edge of $\cc$ of codimension $j$. Let 
$\Fl^k(\cc,\Z)$ denote the quotient of 
the free abelian group on $\Flag^k(\C)$
by the following relations. For every flag with a gap 
\be \widehat F = (L^0 \supset L^1 \supset L^{i-1} \supset L^{i+1} \supset \dots \supset L^k),\qquad i<k,\nn\ee we impose 
\be \sum_{F\supset \widehat F} F =0 \label{flagrel}\ee
in $\Fl^k(\cc,\Z)$, where the sum is over all flags $F = (\widetilde L^0 \supset \widetilde L^1 \supset \dots \widetilde \supset L^k) \in \Flag^k(\cc)$ such that $\widetilde L^j = L^j$ for all $j\neq i$.  

There is an ``extension of flags'' differential $d: \Fl^k(\cc, \Z) \to \Fl^{k+1}(\cc,\Z)$ defined by 
\be d(L^0 \supset L^1 \supset \dots \supset L^k) = 
\sum_{L^{k+1}} (L^0 \supset L^1 \supset \dots \supset L^k \supset L^{k+1}), \nn\ee
where the sum is over all edges $L^{k+1}$ of $\cc$ of codimension $k+1$ contained in $L^k$. It follows from \eqref{flagrel} that $d^2=0$. Thus we have a complex, the \emph{flag complex}, $(\Fl^\bl(\cc,\Z), d)$. 

\subsection{Orlik-Solomon algebras}\label{os}
Define Abelian groups $\A^k(\cc,\Z)$, $k=0,1,\dots,m$ as follows. For $k=0$, set $\A^0(\cc,\Z) = \Z$. For $k\geq 1$, $\A^k(\cc,\Z)$ is generated by $k$-tuples $(H_1,\dots,H_k)$ of hyperplanes $H_i\in \cc$, subject to the relations:
\begin{enumerate}[-]
\item $(H_1,\dots,H_k) =0$ if $H_1,\dots, H_k$ are not in general position (\ie if $\codim H_1 \cap \dots \cap H_k\neq k$);
\item $(H_{\sigma(1)},\dots, H_{\sigma(k)}) = (-1)^{|\sigma|} (H_1,\dots, H_k)$ for every permutation $\sigma \in \Sigma_k$;
\item for any $k+1$ hyperplanes $H_1,\dots,H_{k+1}$ that have non-empty intersection, $H_1\cap\dots\cap H_{k+1}\neq \emptyset$, and that are not in general position,
\be\sum_{p=1}^{k+1} (-1)^p (H_1,\dots,\widehat H_{p},\dots, H_{k+1}) =0,\nn\ee where $\widehat H_p$ denotes omission. 
\end{enumerate}
The \emph{Orlik-Solomon algebra} of the arrangement $\cc$ is the direct sum $\A^\bl(\cc, \Z)=\bigoplus_{k=0}^m \A^k(\cc, \Z)$ endowed with the product given by $(H_1,\dots, H_k) \wx (H'_1,\dots,H'_p) = (H_1,\dots,H_k,H'_1,\dots, H'_p)$. It is a graded skew-commutative algebra over $\Z$.

\subsection{Orlik-Solomon algebra as an algebra of differential forms}\label{osdf}
For each hyperplane $H\in \cc$, pick a polynomial $l_H$ of degree one on $\W$ whose zero set is $H$, \ie let $l_H=0$ be an affine equation for $H$. Consider the logarithmic differential form 
\be \iota(H):= d\log l_H = \frac{dl_H}{l_H}\nn\ee 
on $\W$. Note that $\iota(H)$ does not depend on the choice of $l_H$ but only on $H$. Let $\overline\A^\bl(\cc,\Z)$ be the $\Z$-algebra of differential forms generated by $1$ and $\iota(H)$, $H\in \cc$. The assignment $H\mapsto \iota(H)$ defines an isomorphism $\A(\cc,\Z)\xrightarrow\sim \overline\A(\cc,\Z)$ of graded algebras. 
\emph{Henceforth we shall not distinguish between $\A$ and $\overline\A$.}

\subsection{The pairing of flags with forms}\label{pf}
We say a $k$-tuple $\bar H = (H_1,\dots,H_k)$, $H_i\in \cc$, of hyperplanes is \emph{adjacent (with sign $(-1)^{|s|}$)} to a flag $F = (L^0\supset L^1 \supset \dots L^k) \in \Flag^k(\cc,\Z)$ if there exists a permutation $s \in \Sigma_k$ such that $L^i = H_{s(1)} \cap H_{s(2)} \cap \dots \cap H_{s(i)}$ for each $i=1,\dots, k$. Such a permutation is unique if it exists. 

For each $k=1,\dots,m$, there is an isomorphism  $\varphi^k:\A^k(\cc, \Z)\xrightarrow\sim\Fl^k(\cc, \Z)^*$ defined as follows. If $\bar H =(H_1,\dots,H_k)$ is adjacent with sign $(-1)^{|s|}$ to a flag $F$ then the pairing $\varphi^k(H_1,\dots,H_k)(F)$ is defined to be $(-1)^{|s|}$. Otherwise $\varphi^k(H_1,\dots,H_k)(F)$ is defined to be zero. We shall use the notation
\be \bra F, H_1\wx \dots \wx H_k \ket := \varphi^k(H_1,\dots,H_k)(F) \nn\ee
for this pairing. 

\subsection{Framings and bases}\label{fb}
A \emph{framing} $\oo$ of an arrangement $\cc$ is a choice, for every edge $L$ of $\cc$, of a hyperplane $H(L)$ containing $L$. Given a framing $\oo$, define $\Flag^p(\cc,\oo)$ to be the set of those flags $(L^0\supset L^1\supset \dots \supset L^p)\in \Flag^p(\cc)$, such that $L^k = H(L^1) \cap \dots \cap H(L^k)$ for $k=1,\dots,p$. Equivalently it is the set of flags such that $L^k \not\subset H(L^{k+1})$, $k=1,\dots,p-1$.
\begin{lem}\label{flem2} Given a flag $(L^0\supset L^1\supset \dots \supset L^p)\in \Flag^p(\cc,\oo)$ and a permutation $s\in \Sigma_p$, consider the flag $\tilde F = (\tilde L^0\supset \tilde L^1\supset \dots \supset \tilde L^p)$ defined by $\tilde L^k := H(L^{s(1)}) \cap  \dots \cap H(L^{s(k)})$, $k=1,\dots,p$. If $s$ is not the identity then $\tilde F\notin \Flag^p(\cc,\oo)$. 
\end{lem}
\begin{proof}
If $s\neq\id$ then we can let $k\in\{1,\dots,p\}$ be the largest such that $s(k)\neq k$. Then $k= s(l)$ for some $l\in \{1,\dots,k-1\}$, so  $\tilde L^{k-1} \subset H(L^{k})$. And $s(\{1,\dots,k\}) = \{1,\dots,k\}$, so $\tilde L^{k} = L^{k}$. Thus $\tilde L^{k-1} \subset H(\tilde L^{k})$ and hence $\tilde F \notin  \Flag^p(\cc,\oo)$ as required.\end{proof}

The next proposition is proved in \cite{SV}, Theorems 1.6.5 and 2.9.2.
\footnote{Let us remark that the definitions of $\mathscr O_i$ and hence $Fl_j(\mathscr O)$ in \cite[\S1.6]{SV} have misprints. Nevertheless the proofs there go through for $\Flag^p(\cc,\oo)$ as defined above. Namely, part (i) of the Proposition (\ref{basethm}) follows from \cite[Corollary 2.9.1]{SV} and part (ii) is a corollary of \cite[Lemma 1.5.2]{SV}.}

\begin{prop}\label{basethm} For $p=1,\dots,m$:
\begin{enumerate}[(i)]
\item The group $\Fl^p(\cc,\Z)$ is free over $\Z$, and admits $\Flag^p(\cc,\oo)$ as a base.
\item The group $\A^p(\cc,\Z)$ is free over $\Z$, and admits as a base the set
\be \left\{\left(H(L^1) , H(L^{2}) , \dots , H(L^p)\right) \right\}_{(L^0 \supset L^1 \supset \dots \supset L^p) \in \Flag^p(\cc,\oo)}.\nn\ee
\end{enumerate}
\qed\end{prop}
We have $\Fl^p(\cc,\Z)^*\cong \A^p(\cc,\Z)$ as in \S\ref{pf}. 
\begin{prop} The bases of (i) and (ii) are dual. 
\end{prop}
\begin{proof} Let $F=(L^0 \supset L^1 \supset \dots \supset L^p) \in \Flag^p(\cc,\oo)$. By definition, the pairing this flag $F$ with  $\left(H(L^1) , H(L^{2}) , \dots , H(L^p)\right)$ is 1.  It remains to show that if  $\tilde F = (\tilde L^0\supset \tilde L^1\supset \dots \supset \tilde L^p)\in \Flag^p(\cc)$ is any other flag with which $\left(H(L^1) , H(L^{2}) , \dots , H(L^p)\right)$ has non-zero pairing, then $\tilde F$ does not belong to the set $\Flag^p(\cc,\oo)$ of basis flags. This is the content of Lemma \ref{flem2}.
\end{proof}
\begin{cor}\label{cancor} For $p=1,\dots,m$, the canonical element of $\Fl^p(\cc,\Z)\otimes_\Z \Fl^p(\cc,\Z)^*$ can be expressed as
\be \sum_{(L^0 \supset L^1 \supset \dots \supset L^p) \in \Flag^p(\cc,\oo)} \left(L^0 \supset L^1 \supset \dots \supset L^p\right) \otimes \left(H(L^1) , H(L^{2}) , \dots , H(L^p)\right).  \nn\ee\qed
\end{cor}
\subsection{Weighted arrangements, the geometric form,  and the master function} \label{sg}
Let $\A^k(\cc) := \A^k(\cc, \Z) \otimes_\Z \C$ and $\Fl^k(\cc) := \Fl^k(\cc,\Z) \otimes_\Z \C$ for each $k$. 

A \emph{weighted} arrangement of hyperplanes is an arrangement $\cc$ together with an assignment, to each hyperplane $H\in \cc$, of a number $a(H)\in \C$, its \emph{weight}. 
The weighting defines a symmetric bilinear form $\Geom^k$ on $\Fl^k(\cc)$, \cite{SV}, given by 
\begin{align} \Geom^k(F,F')
  &:= \sum \bra F, H_1\wx\dots\wx H_k\ket \bra F',H_1\wx\dots\wx H_k\ket a(H_1) \dots  a(H_k), \nn\end{align}
where the sum is over the set of unordered $k$-tuples $(H_1,\dots, H_k)$ of hyperplanes.
The form $\Geom^k$ is sometimes called the \emph{quasi-classical contravariant form} of the arrangement $\cc$. 
We shall refer to it simply as the \emph{geometric} form. 
It defines a homomorphism, 
\be \Ghom^k :\Fl^k(\cc) \to \Fl^k(\cc)^* \simeq \A^k(\cc) \nn\ee
by $\Ghom^k(F) := \Geom^k(F,\cdot)$. Explicitly,
\be \Ghom^k( (L^0\supset \dots \supset L^k) ) = \sum a(H_1) \dots a(H_k) H_1 \wx \dots \wx H_k \ee
where the sum is over all $k$-tuples $(H_1,\dots,H_k)$ such that $H_i \supset L^i$ for all $i$. 

The \emph{master function} corresponding to this weighted arrangement is 
\be \Phi = \Phi_{\cc,a}  := \sum_{H\in \cc} a(H) \log l_H \nn\ee
where each $\l_H$ is an affine equation for the hyperplane $H$, as in \S\ref{osdf}. It is defined up to an additive constant.
Define a differential $d=d(a): \A^k(\cc)  \to \A^{k+1}(\cc)$ by the rule 
\be dx = x \wx d\Phi = x\wx \left(\sum_{H\in \C} a(H) H\right).\nn\ee 
It is clear that $d^2=0$, so this makes $(\A^\bl,d)$ into a complex, called the \emph{Aomoto complex}. 
\begin{thm} \label{cmap}$\Ghom^\bl$ is a map of complexes 
\be \Ghom^\bl : (\Fl^\bl(\cc)   , d) \to (\A^\bl(\cc), d) .\nn\ee
\end{thm}
\begin{proof} See \cite[Lemma 5.1]{FMTV} and \cite[Lemma 3.2.5]{SV}. \end{proof}

\subsection{Functoriality of the geometric form on subarrangements}\label{funcsec}
Suppose $\mc B \subset \cc$ is a subarrangement in $\W$. 
There are obvious inclusions $\A^k(\mc B) \into \A^k(\cc)$  and hence $\Fl^k(\cc) \simeq \A^k(\cc)^* \to \A^k(\mc B)^* \simeq \Fl^k(\mc B)$, for each $k$. 
 
\begin{lem}\label{funglem}
Given a weighting $a:\mc B \to \cc$, suppose we extend it to a weighting of the arrangement $\cc$ by setting $a(H) = 0$ for every remaining hyperplane $H\in \cc\setminus \mc B$. Then the following diagram commutes:
\be\begin{tikzpicture}    
\matrix (m) [matrix of math nodes, row sep=3em,    
column sep=5em, text height=1ex, text depth=1ex]    
{     
\Fl^k(\mc B) &  \A^k(\mc B), \\    
\Fl^k(\cc)  & \A^k(\cc).  \\    
};    
\path[->,font=\scriptsize,shorten <= 2mm,shorten >= 2mm]    
(m-1-1) edge node [above] {$\Ghom$} (m-1-2)    
(m-2-1) edge node [above] {$\Ghom$} (m-2-2);    
\path[->,shorten <= 2mm,shorten >= 0mm]  (m-2-1) edge  (m-1-1);   
\path[->,shorten <= 0mm,shorten >= 2mm]  (m-1-2) edge  (m-2-2);   
\end{tikzpicture}\nn\ee    
\qed
\end{lem}

\section{Cyclotomic discriminantal arrangements}\label{sec: cda}
Let $\omega$ be a primitive root of unity of order $T\in \Z_{\geq 1}$. We let the cyclic group $\ZT:= \Z/T\Z$ act on $\C$ by multiplication by powers of $\omega$, \ie $k\on z := \omega^k z$, $k\in \ZT$, $z\in \C$. 
Fix a tuple $z=(z_1,\dots, z_N)$, $N\in \Z_{\geq 0}$, of non-zero complex numbers whose $\ZT$-orbits are pairwise disjoint. 

Fix coordinates $t_1,\dots,t_m$ on $\C^m$.
We define the following hyperplanes in $\C^m$:
\begin{alignat}{2}
\Ht ijk &: t_i - \omega^k t_j = 0 ,\quad &&1\leq i<j \leq m,\quad k\in \ZT, \nn\\
\Hz ijk &: t_i - \omega^k z_j = 0, \quad &&1\leq i\leq m,\quad 1\leq j \leq N,\quad k\in \ZT,\nn\\
 \Ho i &: t_i = 0 , \quad &&1\leq i \leq m. \nn\end{alignat}
Note that $\Ht ijk = \Ht ji{-k}$, and $\Ht ijk \cap \Ht ijl = \Ho i \cap \Ho j$ whenever $k\neq l$.

Let $\cc_m$ denote the arrangement in $\C^m$ consisting of all the hyperplanes $\Ht ijk$.

Let $\cc_{0;m}$ denote the arrangement in $\C^m$ consisting of all the hyperplanes $\Ho i$ and $\Ht ijk$.


Let $\cca
$ denote the arrangement in $\C^m$ consisting of all the $\Ho i$, $\Hz ijk$, and  $\Ht ijk$. 

We have 
 $\cc_m\subset \cc_{0;m} \subset  \cca$.

In the special case $T=1$ we recover the \emph{discriminantal arrangements}, \cite{SV}. More generally we call such hyperplane arrangements \emph{cyclotomic discriminantal arrangements}.

\begin{rem} The arrangements $\cc_m$ and $\cc_{0;m}$ are those of the \emph{complex reflection groups} $G(T,T,m)$ and $G(T,1,m)$ respectively. See e.g. \cite{LT09}.  
\end{rem}

\subsection{Edges and flags of $\cca$} \label{efs}
Given a a $p$-element subset $I=\{i_1,\dots, i_p\}\subset\{1,\dots,m\}$ with $1\leq p\leq m$, and a map $\bm k:\{1,\dots,m\} \to \ZT$, let $\kv L_I$ denote the edge 
given by
\be \kv L_I : \omega^{k(i_1)}t_{i_1} = \omega^{k(i_2)} t_{i_2} = \dots = \omega^{k(i_p)} t_{i_p}.\label{LIdef}\ee
Note in particular the degenerate case $\kv L_{\{i\}} = \C^m$.
For each $i\in \{1,\dots,N\}$, let $\kv L^i_I$ denote the edge 
given by 
\be \kv L^i_I : \omega^{k(i_1)}t_{i_1} = \omega^{k(i_2)} t_{i_2} = \dots = \omega^{k(i_p)} t_{i_p}=z_i.\nn\ee
Let $\kv L^0_I= L^0_I$ denote the edge 
given by 
\be  L^0_I : t_{i_1} = t_{i_2} = \dots = t_{i_p} = 0. \nn\ee
We have $\codim \kv L_I = p-1$,  $\codim \kv L^i_I = p$, and $\codim \kv L^0_I = p$.

Following \cite[\S5.6.4]{SV}, let us think of edges of the form $\kv L_I$ as \emph{swimming islands}, and edges of the form $\kv L_I^i$, $i\in \{0,1,\dots,N\}$, as \emph{fixed islands}. Every edge of $\cca$ is of the form
\be \kv L_{I_1}^{i_1} \cap \dots \cap \kv L_{I_r}^{i_r} \cap \kv L_{J_1} \cap \dots \cap \kv L_{J_p}.\label{arp}\ee
for some function $\bm k: \{1,\dots,m\}\to \ZT$, some collection of $r\in \Z_{\geq 0}$ distinct numbers $i_1,\dots,i_r\in \{0,1,\dots,N\}$ and some pairwise disjoint subsets $I_1,\dots,I_r$, $J_1,\dots,J_p$ ($p\in \Z_{\geq 0}$) of $\{1,\dots,m\}$. We shall refer to the edge \eqref{arp} as the \emph{archipelago} consisting of the swimming islands $\kv L_{J_1}, \dots ,\kv L_{J_p}$ and fixed islands $\kv L_{I_1}^{i_1},\dots ,\kv L_{I_r}^{i_r}$. 

We will use $\kv L_I^\bl$ to denote any edge which is either the swimming island $\kv L_I$ or any one of the fixed islands $\kv L_I^i$.

We say $\kv L_I^\bl$ \emph{involves $j\in \{1,\dots,m\}$} if $j\in I$.  Note that, in particular, the hyperplanes $\Ht jik$, $\Hz jik$, $\Ho j$ are islands which involve $j$.

Given any flag  $F = ( L^0 \supset  L^1 \supset \dots  \supset L^p) \in \Flag^p(\cca)$, each edge $L^t$ is some archipelago. For $t=0,1,\dots,p-1$, each successive archipelago $L^{t+1}$ is obtained from its predecessor $L^t$ in one of the following ways:
\begin{enumerate}
\item A swimming island becomes fixed: that is, $\kv L_I$ is replaced by $\kv L_I^j$ for some $j=0,1,\dots,N$.
\item Two swimming islands become joined: that is, $\kv L_J \cap \kv L_I$ is replaced by $\kv L_{J\cup I}$. 
\end{enumerate}

\subsection{The flag relations of $\F^p(\cca)$}\label{sec: flagrels}
The defining relations, \eqref{flagrel}, in the flag space $\F^p(\cca)$ are 
\be 0=\sum_{w=1}^W (L^0\dots \supset L^t \supset L_w^{t+1}  \supset L^{t+2}\supset \dots \supset L^p ) \nn\ee
for $t\in\{0,\dots,p-2\}$ and for archipelagos $L^t$, $L_w^{t+1}$, $w=1,\dots, W$, and $L^{t+2}$ of the following forms, where in each case $L$ denotes the remaining islands, if any:
\begin{enumerate}[(i)]
\item $W=2$ and 
\begin{alignat}{3} & & L_1^{t+1} &= \kv L_{J_1\cup J_2}\cap \kv L_{J_3} \cap \kv L_{J_4}\cap L, & & \nn\\
L^t &= \kv L_{J_1} \cap \kv L_{J_2}\cap \kv L_{J_3}\cap \kv L_{J_4} \cap L, & & & 
 L^{t+2} & = \kv L_{J_1\cup J_2}\cap \kv L_{J_3\cup J_4}\cap L ,\nn\\
& & L_2^{t+1} &= \kv L_{J_1} \cap \kv L_{J_2} \cap \kv L_{J_3\cup J_4} \cap L.\nn
\end{alignat}
\item $W=2$ and
\begin{alignat}{3} & & L_1^{t+1} &= \kv L_{J_1\cup J_2}\cap \kv L_{I} \cap L, & & \nn\\
L^t &= \kv L_{J_1} \cap \kv L_{J_2}\cap \kv L_{I} \cap L, & & & 
 L^{t+2} & = \kv L_{J_1\cup J_2}\cap \kv L_{I}^j\cap L ,\nn\\
& & L_2^{t+1} &= \kv L_{J_1} \cap \kv L_{J_2} \cap \kv L_{I}^j \cap L,\nn
\end{alignat}
for some $j\in\{0,1,\dots,N\}$. 
\item\label{anchortwo} $W=2$ and
\begin{alignat}{3} & & L_1^{t+1} &= \kv L_{J_1}^{j_1}\cap \kv L_{J_2} \cap L, & & \nn\\
L^t &= \kv L_{J_1} \cap \kv L_{J_2}\cap L, & & & 
 L^{t+2} & = \kv L_{J_1}^{j_1} \cap \kv L_{J_2}^{j_2} \cap L ,\nn\\
& & L_2^{t+1} &= \kv L_{J_1} \cap \kv L_{J_2}^{j_2} \cap L,\nn
\end{alignat}
for some $j_1,j_2\in\{0,1,\dots,N\}$ with $j_1\neq j_2$. 
\item $W=3$ and 
\begin{alignat}{3} & & L_1^{t+1} &= \kv L_{J_1\cup J_2}\cap  \kv L_{J_3} \cap L, & & \nn\\
L^t &= \kv L_{J_1} \cap \kv L_{J_2} \cap \kv L_{J_3} \cap L,  & \qquad L_2^{t+1} &= \kv L_{J_2\cup J_3} \cap \kv L_{J_1} \cap L, \qquad & 
 L^{t+2} &= \kv L_{J_1\cup J_2\cup J_3}\cap L  ,\nn\\
& & L_3^{t+1} &= \kv L_{J_3\cup J_1} \cap \kv L_{J_2}  \cap L.\nn
\end{alignat} 
\item 
$W=3$ and 
\begin{alignat}{3} & & L_1^{t+1} &= \kv L_{J_1\cup J_2} \cap L, & & \nn\\
L^t &= \kv L_{J_1} \cap \kv L_{J_2}  \cap L,  & \qquad L_2^{t+1} &= \kv L_{J_1} \cap \kv L_{J_2}^j \cap L, \qquad & 
 L^{t+2} &= \kv L_{J_1\cup J_2}^j \cap L  ,\nn\\
& & L_3^{t+1} &= \kv L_{J_1}^j \cap \kv L_{J_2}  \cap L,\nn
\end{alignat}
for some $j\in\{1,\dots,N\}$. 
\item \label{nerel}
$W=T+2$ and, for some (any) fixed choice of $j_1\in J_1$, $j_2\in J_2$ and for $\ell=1,2,\dots,T$,
\begin{alignat}{3} 
& & L_1^{t+1} &=   \left(\kv L_{J_1} \cap \kv L_{J_2} \cap \Ht{j_1}{j_2}{1}\right) \cap L, & & \nn\\
& & L_2^{t+1} &=   \left(\kv L_{J_1} \cap \kv L_{J_2} \cap \Ht{j_1}{j_2}{2}\right) \cap L, & & \nn\\
& &        & \vdots & & \nn\\
& & L_T^{t+1} &=   \left(\kv L_{J_1} \cap \kv L_{J_2} \cap \Ht{j_1}{j_2}{T}\right) \cap L, & & \nn\\
L^t &= \kv L_{J_1} \cap \kv L_{J_2}  \cap L,  & \qquad L_{T+1}^{t+1} &= \kv L_{J_1} \cap L_{J_2}^0 \cap L, \qquad & 
 L^{t+2} &= L_{J_1\cup J_2}^0 \cap L  ,\nn\\
& & L_{T+2}^{t+1} &=  L_{J_1}^0 \cap \kv L_{J_2}  \cap L.\nn
\end{alignat} 
\end{enumerate}

\begin{rem}\label{virem} Note especially the final relation.
As a simple example, we have 
\be (\C^2 \supset \Ho 1 \supset \Ho 1 \cap \Ho 2) + (\C^2 \supset \Ho 2 \supset \Ho 1 \cap \Ho2) + \sum_k(\C^2 \supset \Ht 12k \supset \Ho 1 \cap \Ho 2) = 0.\nn\ee
The other relations are just as in the non-cyclotomic case of \cite{SV}.
\end{rem}

\subsection{On flags with no zero ends}\label{sec:nozero}
Let us say a flag $F = (L^0\supset\dots\supset L^p) \in \Flag^{p}(\cca)$ has \emph{no zero ends} if $L^p \not\subset \Ho i$ for all $i\in \{1,\dots,m\}$. Otherwise we say $F$ \emph{has zero ends}. Thus, informally, a flag has zero ends ``if some $t_i$ is set equal to zero''.

Fix a flag $F\in \Flag^p(\cca)$ with no zero ends, for some $p\in \{1,\dots,m\}$.

The last edge $L^p$ of $F$ is of the form
\be L^p = \kv L_{I_1}^{i_1} \cap \dots \cap \kv L_{I_r}^{i_r} \cap \kv L_{J_1} \cap \dots \cap \kv L_{J_p}.\nn\ee
for some map\footnote{unique if and only if $p=m$} $\bm k: \{1,\dots,m\}\to \ZT$, some collection of $r\in \Z_{\geq 0}$ distinct non-zero numbers $i_1,\dots,i_r\in \{1,\dots,N\}$ and some pairwise disjoint subsets $I_1,\dots,I_r$, $J_1,\dots,J_p$ ($p\in \Z_{\geq 0}$) of $\{1,\dots,m\}$. 
Define new coordinates 
\be \td t_i := \omega^{k(i)} t_i, \qquad i=1,\dots,m. \nn\ee
Let $\kv\cc_{\bm z;m}\subset \cca$ be the arrangement consisting of the hyperplanes
\begin{alignat}{2}
\td H_{ij} &: \td t_i -  \td t_j = 0 ,\quad &&1\leq i<j \leq m, \nn\\
\td H_{i}^j &: \td t_i - z_j = 0, \quad &&1\leq i\leq m,\quad 1\leq j \leq N.\label{tdH}\end{alignat}
By construction, $\kv\cc_{\bm z;m}$ is a discriminantal arrangement, in the sense of \cite{SV}, and $F\in \Flag^p(\kv\cc_{\bm z;m})$. 

\begin{lem} If a hyperplane $H\in \cca$ contains the edge $L^p$, then $H$ belongs to $\kv\cc_{\bm z;m}$. \qed
\end{lem}
\begin{cor}\label{adcor} If a tuple  $(H_1,\dots,H_p)$ of hyperplanes of $\cca$ is adjacent to $F$, in the sense of \S\ref{pf},  then this tuple consists of hyperplanes belonging to the discriminantal arrangement $\kv\cc_{\bm z;m}$. \qed
\end{cor}

Note that the condition that $F$ had no zero ends was necessary, as the following example shows. 

\begin{exmp} Consider the full flag $\C^2 \supset (t_1=t_2) \supset (t_1=t_2=0)$. This flag does have zero ends. It is adjacent to the tuples of hyperplanes $\{(t_1=t_2), (t_1=0)\}$ and $\{(t_1=t_2),(t_2=0)\}$ but \emph{also} to the tuples $\{ (t_1=t_2), (t_1=\omega^kt_2) \}$ for each $k\in \ZT\setminus\{0\}$. Thus there is no single discriminantal arrangement such that all tuples adjacent to this flag consist of hyperplanes from that arrangement.
\end{exmp}

Suppose now that we have a weighting $a: \cca \to \C$ of the arrangement $\cca$. It defines a bilinear form $\Geom^p(\cdot,\cdot)$ on $\Fl^p(\cca)$ as in \S\ref{sg}. 
Let $\td a: \kv\cc_{\bm z;m}$ be the restriction of this weighting $a$ to the arrangement $\kv\cc_{\bm z;m}$. It defines a bilinear form $\td\Geom^p(\cdot,\cdot)$ on $\Fl^p(\kv\cc_{\bm z,m})$.

The next lemma is a consequence of Corollary \ref{adcor}. It will be an important source of simplifications in what follows. 

\begin{lem}\label{Glem} Let $F$ be a flag with no zero ends, and let $\kv C_{\bm z;m}$ be the corresponding discriminantal arrangement, as above. For all $F'\in \Flag^p(\cca)$,  
\be G^p(F,F') = \begin{cases} \td G^p(F,F')& \text{if } F'\in \Flag^p(\kv\cc_{\bm z;m}) \\ 0 & \text{otherwise.} \end{cases}\nn\ee\qed
\end{lem}

\section{Free Lie algebra with automorphism}\label{fla}
Let $\omega\in \Cx$ be a primitive $T$th root of unity and $\ZT$ the cyclic group of order $T$, as above. 

Let $\nb$ denote the free Lie algebra over $\C$ with generators $\f k i$, $i= 1,\dots,m$, $k\in \ZT$.  There is an automorphism $\sib:\nb\to\nb$ defined by
\be \sib: \f k i \to \f {k+1} i, \nn\ee
where we understand that $k+1$ means addition in $\ZT$. 

Let $\nbs$ denote the Lie subalgebra of $\sib$-invariant elements of $\nb$ and let
$\Pin{\phantom x}: \nb \to \nbs$ denote the surjective linear map given by 
\be \Pin g := \sum_{j\in \ZT} \sib^j g\label{Pindef}, \qquad g\in \nb.\ee 

We sometimes write $\f 0i$ as $f_i$, and hence $\f ki$ as $\sib^k f_i$. 

There is a unique $\Z_{\geq 0}^m$-grading of $\nb$ as a Lie algebra,  $\nb = \bigoplus_{(r_1,\dots,r_m)\in \Z_{\geq 0}^m} \nb_{[r_1,\dots,r_m]}$,
such that \be\nb_{[0,\dots,0,\at i 1,0,\dots,0]} := \Span_\C(\f k i)_{k\in \ZT}\nn\ee for each $i=1,\dots,m$.
We call this the \emph{weight decomposition} of $\nb$, and say $\nb[r_1,\dots,r_m]$ is the subspace of \emph{weight} $(r_1,\dots,r_m)$. We write $\wt( v) = (r_1,\dots,r_n)$ if $v\in \nb[r_1,\dots,r_m]$.

In particular we have the weight decomposition of $\nbs$. 
 
We extend the weight decomposition additively over tensor products, \ie  $\wt(a\otimes b) := \wt a + \wt b$. 
So we have the weight decompositions of the universal envelopes $U(\nb)$ and $U(\nbs)$, and of 
\be \Mb := U(\nbs) \ox \bigotimes_{i=1}^N U(\nb) .\label{Mbdef}\ee
Note that $\Mb$ is an $\nbs$-module, with the action given by 
\begin{multline} x\on (m_0 \ox m_1 \ox \dots \ox m_N) \\
:= (x m_0) \ox m_1 \ox \dots \ox m_N + m_0 \ox (xm_1) \ox m_2 \ox \dots \ox m_N + \dots + m_0 \ox m_1 \ox \dots \ox (xm_N) .\nn\end{multline}

Let $(C_\bl(\nbs,\Mb),d)$ denote the standard chain complex of $\nbs$ with coefficients in $\Mb$. Namely, for $k\in \Z_{\geq 0}$, define 
\be C_k(\nbs,\Mb) := \Wedge k \nbs \otimes \Mb \ee
(where $\Wedge k$ denotes the $k$-th exterior power), and let 
\be d:  C_k(\nbs,\Mb) \to C_{k+1}(\nbs,\Mb) \ee
be the linear map defined by 
\begin{multline} d(\Pin g_k \wx \Pin g_{k-1} \wx \dots \wx \Pin g_1 \ox x) 
 = \sum_{i=1}^k (-1)^{i-1} \Pin g_k \wx \dots \Hat{\Pin g}_i \wx\dots \wx  \Pin g_1 \ox \Pin g_i\on x \\
 +
 \sum_{1\leq i<j\leq k} (-1)^{i+j} \Pin g_k \wx \dots \wx \Hat{\Pin g}_i \wx \dots \wx \Hat{\Pin g}_j \wx \dots \wx \Pin g_1 \wx [\Pin g_j,\Pin g_i] \otimes x \label{ddef}\end{multline} 
for $x\in \Mb$, $\Pin g_1,\dots,\Pin g_k\in \nbs$, where $\Hat{\phantom x}$ denotes omission. 

Recall the cyclotomic discriminantal arrangement $\cca$ from \S\ref{sec: cda}. We have the flag complex $(\Fl^\bl(\cca),d)$ as in \S\ref{sec: arrangements}. 
Let us write  
\be [1^m] := \underbrace{[1,\dots,1]}_m.\nn\ee
In this section we define (Theorem \ref{p1}) a family of linear isomorphisms 
\be \psis_k : (\Wedge k \nbs \otimes \Mb)_{\mones} \to \Fl^{m-k}(\cca),\nn\ee
such that $ d\circ \psis_k = \psis_{k-1} \circ d$.

\subsection{Commutators and projected commutators}\label{cpc}
Let us say that an element $g\in \nb$ is a \emph{commutator} if either
\begin{enumerate}\item $g$ is one of the generators $\f ki$ of $\nb$, or 
\item $g=[g_1,g_2]$ for commutators $g_1,g_2\in \nb$.
\end{enumerate} 
For a given commutator $g$ and generator $\f ki$ there is a well-defined notion of the number of times $\f ki$ occurs in $g$. Let the \emph{length} $l(g)$ of a commutator $g$ be the total number of generators that occur in $g$, and the \emph{content} of $g$ the set $\{(k_1,i_1),\dots (k_{l(g)},i_{l(g)})\}$ of the labels of these generators.\footnote{So each commutator lies in some weight subspace $\nb[r_1,\dots,r_m]$ of $\nb$, determined by its content. In general the content is a set with multiplicities. When working with $C_\bl(\nbs ,\Mb )_{\mones}$ these multiplicities will be at most one.}

For example, $[[\f {k_1}{i_1}, \f {k_2}{i_2}], \f {k_3}{i_3}]$ has length 3 and content $\{ (k_1,i_1), (k_2,i_2), (k_3,i_3)\}$.


Whenever $g$ is a commutator in $\nb$, we call the element $\Pin g\in \nbs$ a \emph{projected commutator}, cf. \eqref{Pindef}. 
The Lie algebra $\nbs$ is spanned by the set of all projected commutators  $\Pin g$ as $g$ runs over the   set of commutators in $\nb$. (Indeed, the latter span $\nb$ and $\Pin{\phantom a}:\nb \to \nbs$ is a surjective linear map.) 

\begin{rem} Suppose we are not in the degenerate cases $m= 1$ or $T= 1$. Then $\nbs$ is \emph{not} spanned by commutators of the projected generators $\Pin f_i$. For example $\Pin{[ \f p 1,\f q 2]}\in \nbs$, $p\neq q$, is not in the span of $[\,\Pin f_1,\Pin f_2]$. In fact $\nbs$ is not finitely generated \cite{Bryant,BryantPapistas}.
\end{rem}


We say an element $x\in U(\nb)$ is a \emph{monomial} if it is equal to a product of commutators, \ie if $x=g_p\dots g_1$ for commutators $g_1,\dots, g_p \in \nb$. 

We say an element $x\in U(\nbs)$ is a \emph{monomial} if it is equal to a product of projected commutators, \ie if $x=\Pin g_p\dots \Pin g_1$ for commutators $g_1,\dots, g_p \in \nb$.  

We say an element $x\in \Wedge p \nbs \ox \Mb$ is a \emph{monomial} if 
\be x=\Pin g_p\wx \dots \wx \Pin g_1 \ox x_0 \ox x_1 \ox \dots \ox x_N\nn\ee 
for commutators $g_1,\dots,g_p$ and monomials $x_0,x_1,\dots,x_N$.


\subsection{The maps $\psis_p$}\label{psisdef}
Given any commutator $g$ in $\nb$, with content say $\{(k_1,i_1),\dots, (k_{l},i_{l})\}$, we shall write $L_g$, $L_g^j$, $j=1,\dots,N$, and $L_g^0$ for the edges of the arrangement $\cca$ given by 
\begin{align} L_g &: \omega^{k_1}t_{i_1} = \dots = \omega^{k_{l}}t_{i_l},\nn\\
L_g^j&: \omega^{k_1}t_{i_1} = \dots = \omega^{k_{l}}t_{i_l} = z_j,\nn\\ 
L_g^0&: t_{i_1} = \dots = t_{i_l} = 0. \label{Lgdef}\end{align}
(When $g$ has length 1, $L_g= \C^m$.) 

Observe that $L_g= L_{\sib g}$.

Now we define linear maps $\psis_p : (\Wedge p \nbs \ox \Mb)_{\mones} \to \Fl^{m-p}(\cca)$ by induction on $m-p$ as follows. 

For the base case $p=m$, we set $\psis_m( \Pin f_1 \wx \dots \wx\Pin f_m  \ox 1 \ox 1 \ox \dots \ox 1)$ to be the trivial flag, $L^0=\C^m$.

For the inductive step, assume that $\psis_{p+1}$ has been defined in such a way that for any monomial 
$x=\Pin g_{p+1} \wx \dots \wx \Pin g_1 \ox x_0 \ox x_1 \ox \dots \ox x_N\in (\Wedge {p+1} \nbs \ox \Mb)_{\mones}$,
the image $\psis_{p+1}(z)$ is of the form $(-1)^s F$ for some sign $(-1)^s$ and some flag $F = (L^0 \supset \dots \supset L^{m-p})\in \Flag^{m-p}(\cca)$. 

Then we define $\psis_p(x') := (-1)^{s} F'$ for $x',F'$ as follows:
\begin{enumerate}[(A)]
\item\label{r1}
If $x'= \Pin g_{p+1} \wx \dots \wx  \Pin g_2 \ox x_0 \ox x_1\ox \dots \ox g_1\on x_j \ox \dots \ox x_N$ then  $F' = (L^0 \supset \dots \supset L^{m-p} \supset L^{m-p} \cap L_{g_1}^j)$.
\item\label{r2}

If $x'=  \Pin g_{p+1} \wx \dots \wx  \Pin g_2 \ox \Pin g_1 \on x_0 \ox x_1\ox \dots \ox x_N$ then $F'= (L^0 \supset \dots \supset L^{m-p} \supset L^{m-p} \cap L_{g_1}^0)$.
\item\label{r3}

If $x' =  \Pin g_{p+1} \wx \dots \wx  \Pin g_3 \wx  \Pin{[g_1,g_2]} \ox x_0 \ox x_1\ox \dots \ox x_N$ then $F'= (L^0 \supset \dots \supset L^{m-p} \supset L^{m-p} \cap L_{[g_1,g_2]})$. 
\end{enumerate}

\begin{thm}\label{p1}
These rules correctly define a family of linear isomorphisms 
\be \psis_p : \left(\Wedge p \nbs \ox \Mb\right)_{\mones} \to \Fl^{m-p}(\cca), \qquad 0\leq p\leq m.\nn\ee 
Moreover this gives an isomorphism of complexes 
\be \psis_\bl : (C_\bl(\nbs ,\Mb )_{\mones},d)\to (\Fl^{m-\bl}(\cca),d).\nn\ee
\end{thm}
\begin{exmp}\label{p1ex}
Consider $m=2$ and $N=1$. The monomials in $\Wedge \bl(\nbs)\ox\Mb=\Wedge \bl(\nbs)\ox U(\nbs) \ox U(\nb)$, and their images under $\psis_\bl$, include:
\begin{alignat*}{2}
\Pin f_1 \wx \Pin f_2 \ox 1 & \ox 1 &    (\C^2&),  \\
\Pin f_1 \ox 1 & \ox \sib^k f_2 &   (\C^2 & \supset (\omega^kt_2=z)), \\
\Pin f_2 \ox 1 & \ox \sib^k f_1 &   -(\C^2 & \supset (\omega^kt_1= z)), \\
             1& \ox (\sib^p f_1) (\sib^qf_2) & (\C^2 & \supset (\omega^qt_2= z) \supset (\omega^q t_2 = \omega^p t_1 = z),\\
\Pin f_1 \ox \Pin f_2 & \ox 1 &   (\C^2 & \supset (t_2=0)) ,\\
       \Pin f_2 & \ox \sib^k f_1 &   (\C^2 & \supset (t_2=0)\supset (t_2=0,\omega^k t_1 = z)) \\
 & &  &=  -(\C^2 \supset (\omega^kt_1=z)\supset (t_2=0,\omega^k t_1 = z)) ,\\
\Pin{[\sib^pf_2, f_1]} \ox 1 & \ox 1 &   (\C^2 & \supset (\omega^pt_2=t_1)) ,\\
    1& \ox \sib^q[\sib^{p}f_2,f_1] & \qquad\qquad (\C^2 & \supset (\omega^pt_2=t_1) \supset (\omega^{p+q} t_2 = \omega^q t_1 = z)),\\
\Pin{[\sib^pf_2, f_1]}& \ox 1 & (\C^2 & \supset (\omega^pt_2=t_1)\supset (t_2=t_1=0)).
\end{alignat*}
\end{exmp}
\begin{proof}[Proof of Theorem \ref{p1}]
By recursively applying the rules and using the skew-symmetry in the factors of $\Wedge p \nbs$, one can compute $\psis_p$ on any monomial in $(\Wedge p \nbs \ox \Mb)_{\mones}$. 
The first thing to check is that the result well-defined, \ie that it is independent of the order in which we choose to apply the rules. 

Consider $x''=\Pin g_{p} \wx \dots \wx \Pin g_3 \ox x_0 \ox x_1 \ox \dots \ox g_1 \on x_i \ox \dots \ox g_2 \on x_j \ox \dots \ox x_N$. Starting from $\psis_p(x) = (-1)^s (L^0 \supset \dots \supset L^{m-p})$, we can compute $\psis_{p-2}(x'')$ in two ways. On the one hand we have 
\be \psis_{p-1}(\Pin g_{p} \wx \dots \wx \Pin g_3 \wx \Pin g_2 \ox x_0 \ox x_1 \ox \dots \ox g_1 \on x_i \ox \dots \ox x_N) = (-1)^s (L^0\supset \dots \supset L^{m-p} \supset L^{m-p}\cap L_{g_1}^i) \nn\ee
and hence
\be \psis_{p-2}( x'') = (-1)^s (L^0\supset \dots \supset L^{m-p} \supset L^{m-p}\cap L_{g_1}^i
\supset L^{m-p} \cap L_{g_1}^i \cap L_{g_2}^j). \label{a1}\ee
On the other hand we have
\be \psis_{p-1}(\Pin g_{p} \wx \dots \wx \Pin g_3 \wx \Pin g_1 \ox x_0 \ox x_1 \ox \dots \ox g_2 \on x_j \ox \dots \ox x_N) = - (-1)^s (L^0\supset \dots \supset L^{m-p} \supset L^{m-p}\cap L_{g_2}^j) \nn\ee
(note the extra sign, which comes from $\Pin g_2\wx \Pin g_1 = -\Pin g_1\wx \Pin g_2$) and hence
\be \psis_{p-2}( x'') = -(-1)^s(L^0\supset \dots \supset L^{m-p} \supset L^{m-p}\cap L_{g_2}^j
\supset L^{m-p} \cap L_{g_1}^i \cap L_{g_2}^j). \label{a2}\ee
And indeed \eqref{a1} and \eqref{a2} are equal in the flag space $\Fl^{m-p+2}(\cca)$, by virtue of the flag relations \eqref{flagrel}. (Specifically, a relation of type (\ref{anchortwo}) in \S\ref{sec: flagrels}.)

Similarly, one checks that, starting from $\psis_p(z)=(-1)^sF$, $\psis_{p-2}$ is well-defined on $\Pin g_p \wx \dots \wx \Pin g_3  \ox \Pin g_1 \on x_0 \ox x_1\ox\dots g_2\on x_i \ox \dots \ox x_N$,  on  $\Pin g_p \wx \dots \wx \Pin g_4 \wx \Pin{[g_3,g_2]} \ox x_0 \ox x_1\ox\dots g_1 \on x_i \ox \dots \ox x_N$, on $\Pin g_p \wx \dots \wx \Pin g_4 \wx \Pin{[g_3,g_2]} \ox \Pin g_1 \on x_0 \ox x_1\ox\dots \ox x_N$ and finally on $\Pin g_p \wx \dots \wx \Pin{[g_4,g_3]} \wx \Pin{[g_2,g_1]} \ox x_0 \ox x_1\ox\dots\ox x_N$. 

This establishes that $\psis_p$ is well-defined as a map from the set of monomials in $(\Wedge p \nbs \ox \Mb)_{\mones}$ into the flag space $\Fl^{m-p}(\cca)$. 

Now we must check that this map can be extended by linearity to a linear map 
$(\Wedge p \nbs \ox \Mb)_{\mones}\to \Fl^{m-p}(\cca)$. For this it is enough to check that 
\be \psis_{p-2}( \Pin g_p\wx \dots \wx \Pin g_3 \ox\big( (\Pin g_1) \on (\Pin g_2) \on x_0 - (\Pin g_2)\on (\Pin g_1)\on x_0 
                - [\Pin g_1,\Pin g_2] \on x_0 \big) \ox x_1 \ox \dots \ox x_N) =0,\label{rel1}\ee
\be \psis_{p-2}(\Pin g_p\wx  \dots \wx \Pin g_3 \ox x_0 \ox x_1 \ox \dots \ox \big( (\sib^k g_1)\on (\sib^l g_2) \on x_j -  (\sib^l g_2)\on(\sib^k g_1)\on x_j  
           - [\sib^k g_1,\sib^l g_2] \on x_j \big)  \ox \dots \ox x_N) =0,\nn\ee
for each $j=1,\dots,N$, and
\be \psis_{p-2}(\Pin g_p \wx \dots \wx  ( \Pin{[g_3,[g_2,g_1]]}+\Pin{[g_2,[g_1,g_3]]}+ \Pin{[g_1,[g_3,g_2]]})  \ox x_0\ox x_1 \ox \dots \ox x_N )=0.\nn\ee
Again these follow from the relations in the flag spaces. Let us consider the first in detail. We have
\begin{multline} \psis_{p-2}(\Pin g_p \wx \dots \wx \Pin g_3 \ox (\Pin g_2) \on (\Pin g_1) \on x_0\ox x_1\ox \dots \ox x_N) \\ = (-1)^s (L^0 \supset \dots \supset L^{m-p} \supset L^{m-p} \cap L_{g_1}^0 \supset L^{m-p} \cap L_{g_1}^0  \cap L_{g_2}^0) \label{b1}\end{multline}
and
\begin{multline} \psis_{p-2}(\Pin g_p \wx \dots \wx \Pin g_3 \ox (\Pin g_1) \on (\Pin g_2) \on x_0\ox x_1\ox \dots \ox x_N) \\ = -(-1)^s (L^0 \supset \dots \supset L^{m-p} \supset L^{m-p} \cap L_{g_2}^0 \supset L^{m-p} \cap L_{g_1}^0  \cap L_{g_2}^0) \label{b2}\end{multline}
(the sign coming from $\Pin g_2\wx \Pin g_1 = -\Pin g_1\wx \Pin g_2$). Now, in the flag space $\Fl^{m-p+2}(\cca)$ we have, by \eqref{flagrel}, the relation 
\begin{multline} 0= (L^0 \supset \dots \supset L^{m-p} \supset L^{m-p} \cap L_{g_1}^0 \supset L^{m-p} \cap L_{g_1}^0  \cap L_{g_2}^0) \\
+ (L^0 \supset \dots \supset L^{m-p} \supset L^{m-p} \cap L_{g_2}^0 \supset L^{m-p} \cap L_{g_1}^0  \cap L_{g_2}^0) \\
+ \sum_{k\in \ZT} (L^0 \supset \dots \supset L^{m-p} \supset L^{m-p} \cap L_{[g_1,\sib^k g_2]} \supset L^{m-p} \cap L_{g_1}^0  \cap L_{g_2}^0)\label{b3}\end{multline}
(which is of type \eqref{nerel} in \S\ref{sec: flagrels}).
The third line here is equal to 
\begin{multline} \sum_{k\in \ZT} \psi_{p-2}(\Pin g_p\wx \dots \wx \Pin g_3 \ox (\Pin{[ g_1, \sib^k g_2]})\on x_0 \ox x_1\ox \dots \ox x_N)\\
 = (-1)^s\psi_{p-2}(\Pin g_p \wx \dots \wx \Pin g_3 \ox (\Pin{[ g_1, \Pin g_2]})\on x_0 \ox x_1\ox \dots \ox x_N) \\
= (-1)^s\psi_{p-2}(\Pin g_p \wx \dots \wx \Pin g_3 \ox [ \Pin g_1, \Pin g_2] \on x_0 \ox x_1\ox \dots \ox x_N)\nn\end{multline}
Therefore, in view of \eqref{b1} and \eqref{b2}, the relation \eqref{b3} yields the required identity,  \eqref{rel1}.
 
Thus $\psis_p$ is a linear map $(\Wedge p \nbs \ox \Mb)_{\mones}\to \Fl^{m-p}(\cca)$. Now we show it is a linear isomorphism. To do so we define the inverse linear map $\psis_p^{-1}$, by induction on $m-p$. 

For the base case $p=m$ we set  $\psis_m^{-1}(L^0=\C^m) = \Pin f_1\wx\dots\wx \Pin f_m \ox 1 \ox 1 \ox \dots \ox 1$.

For the inductive step, pick any flag  $F = (L^0 \supset \dots \supset L^{m-p-1} \supset L^{m-p}) \in \Flag^{m-p}(\cca)$. Inductively we may assume we have the monomial 
\be \Pin g_{p+1} \wx \dots \wx \Pin g_1 \ox x_0 \ox x_1 \ox \dots \ox x_N := \psis_{p+1}^{-1}( L^0 \supset \dots \supset L^{m-p-1}).\nn\ee
Now consider the last step, $L^{m-p-1} \supset L^{m-p}$, of the flag $F$: 

First, suppose ``$t_i$ and $z_j$ became linked''. Namely, suppose that for some $i,j,k$, we have $\Hz ijk \supset L^{m-p}$ and $\Hz ijh \not \supset L^{m-p-1}$ for every $h\in \ZT$. 
By re-ordering the factors as necessary, using skew-symmetry in $\Wedge {p+1} \nbs$, we may assume that $\f l i$ appears in $g_1$, for some $l\in \ZT$. By the invariance $\Pin g_1 = \Pin{\sib g_1}$ we may assume $l=0$. That is, we may assume $f_i$ appears in $g_1$. Then we set 
\be \psis_p^{-1}(F) 
 := \Pin g_{p+1} \wx \dots \wx \Pin g_2 \ox x_0 \ox x_1 \ox \dots \ox (\sib^{-k} g_1)\on x_j \ox \dots \ox x_N.\nn\ee

Next, suppose that  ``$t_i$ and 0 became linked''. Namely, suppose that for some $i$, $\Ho i\supset L^{m-p}$ and $\Ho i \not\supset L^{m-p-1}$.  Again, we may assume that $f_i$ appears in $g_1$. We set
\be 
\psis_p^{-1}(F) := \Pin g_{p+1} \wx \dots \wx \Pin g_2 \ox (\Pin g_1)\on x_0 \ox x_1 \ox \dots \ox x_N.\nn\ee

Finally, suppose we are in neither of the above cases. Then it must be that  ``$t_i$ and $t_j$ became linked''. That is, for some $i,j,k$, we have $\Ht ijk \supset L^{m-p}$ and $\Ht ijn \not \supset L^{m-p-1}$ for any $n\in \ZT$. We may suppose that $f_i$ appears in $g_2$ and $f_j$ appears in $g_1$. We define
\be \psis_p^{-1} := \Pin g_{p+1} \wx \dots \wx \Pin g_3 \wx \Pin{[\sib^{k} g_1,g_2]} \otimes x_0 \ox x_1 \ox \dots \ox x_N.\nn\ee

These definitions respect the relations in $\Fl^\bl(\cca)$. 

This completes the inductive step, and we have the inverse map $\psis_p^{-1}$ for each $p=0,\dots,m$.

For the moreover part, one checks directly that $ d\circ \psis_p = \psis_{p-1} \circ d$ for each $p=1,\dots,m$.
\end{proof}

Given a subset $I\subseteq\{1,\dots,m\}$, let 
$\ccI I\subset \cca$
denote the arrangement consisting of hyperplanes $\Ht {i_1}{i_2}k$, $\Hz {i_1}jk$, 
and $\Ho {i_1}$
with $i_1,i_2\in I$, $j\in \{1,\dots,N\}$, $k\in \ZT$.  
Let $\cco I$ denote the arrangement consisting of just the ``diagonal'' hyperplanes $\Ht {i_1}{i_2}k $, $i_1,i_2\in I$, $k\in \ZT$.

Let $\wtI I$ denote the tuple $(r_1,\dots,r_m)$ with $r_i=1$  if $i\in I$ and $r_i=0$ if $i\notin I$. 
By an obvious generalization of Theorem \ref{p1}, we have linear isomorphisms
\be \psis = \psis^I_p : (\Wedge p \nbs \ox \Mb)_{\wtI I} \to \Fl^{|I|-p}(\ccI I), \qquad 0\leq p\leq |I|.\label{genpsis}\ee 
(We shall sometimes suppress the indices $p$ and $I$ from $\psis$.) 

For any commutator $c$ in $\nb$,  define the flag
\be \psis(c) := \psis(\Pin c) := \psis( \Pin c \ox 1 \ox 1^N).\label{compsis}\ee
Note that $\psis(c) = \psis(\tau^k c)$ for all $k\in \ZT$.

\section{Bilinear form}\label{bf}
We keep the conventions of section \S\ref{fla}. Let us now fix a weighting $a:\cca\to \C$ of the arrangement $\cca$, in the sense of \S\ref{sg}. Let $\gu$ be the Lie algebra with generators $\f ki$, $\e ki$ and $\hh ki$, $i=1,\dots m$, $k\in \ZT$, subject to the following relations:
\be [\e ki,\f lj] = \delta_{ij} \delta_{kl} \hh ki,\nn\ee
\be [\hh ki,\e lj] =  \begin{cases} a(\Ht ij {l-k}) \e lj & i\neq j \\ 0 & i=j \end{cases}, \qquad 
    [\hh ki,\f lj] = \begin{cases} -a(\Ht ij {l-k}) \f lj & i\neq j \\ 0 & i=j \end{cases}, \nn\ee
\be [\hh ki,\hh lj] = 0. \nn\ee

We sometimes write $e_i:= \e 0i$, $f_i := \f 0i$, $h_i := \hh 0i$. 

There is an automorphism $\sib: \gu\to \gu$ given by $\sib({}^kx_i) = {}^{k+1}x_i$ where $x\in \{e,f,h\}$ and $k+1$ is understood to mean addition modulo $T$. Let $\gus\subset \gu$ denote the subalgebra fixed by $\sib$ and $\Pin{\phantom x} : \gu \to \gus$ the surjective linear map $x\mapsto \Pin x := \sum_{j\in \ZT} \sib^j x$. 

We have the obvious embedding of Lie algebras $\nb \into \gu$.
Let $\gu_0$ (resp. $\nb_+$) denote the subalgebra  of $\gu$ generated by the $\hh ki$ (resp. $\e ki$).
Then $\gu = \nb \oplus \gu_0 \oplus \nb_+$ and hence $U(\gu) \cong_\C U(\nb) \ox U(\gu_0) \ox U(\nb_+)$. 
Likewise  
$\gus = \nbs \oplus \gus_0 \oplus \nbs_+$ and hence $U(\gus) \cong_\C U(\nbs) \ox U(\gus_0) \ox U(\nbs_+)$.

Let $M_i:= U(\gu) \ox_{U(\gu_0\oplus \nb_+)} \C v_i$, $i=1,\dots, N$, denote the $\gu$-module generated by a vector $v_i$ obeying the relations
\be \nb_+\on v_i = 0, \qquad \hh kj \on v_i = a(\Hz jik) v_i.\nn\ee

Let $M_0:= U(\gus) \ox_{U(\gus_0\oplus \nbs_+)} \C v_0$, denote the $\gus$-module generated by a vector $v_0$ obeying the relations
\be \nbs_+ \on v_0 = 0, \qquad \frac 1T \Pin{h}_j \on v_0 = a(\Ho j) v_0.\nn\ee

There are isomorphisms $M_i \cong_{\nb} U(\nb)$ and $M_0 \cong_{\nbs} U(\nbs)$. Hence, with $\Mb$ as in \eqref{Mbdef},
\be \Mb  \cong_{\nbs} M_0 \ox \bigotimes_{i=1}^N M_i.\ee

Our choice of weighting $a: \cca \to \C$ defines a symmetric bilinear form $\Geom^\bl$ on $\Fl^\bl(\cca)$ and corresponding linear maps $\Ghom^\bl: \Fl^\bl(\cca) \to \Fl^\bl(\cca)^*$. (See \S\ref{sg}.) Pulling back by the linear isomorphism $\psis_p$ of Theorem \ref{p1}, we obtain a bilinear form $\psis_p^*(\Geom^{m-p})$ on $(\Wedge p \nbs \ox \Mb)_{\mones}$, and a linear map 
\be (\psis_p)^{-1} \circ\Ghom^{m-p}\circ \psis_p:  \left(\Wedge p \nbs \ox \Mb\right)_{\mones} \to \left( \Wedge p \nbs \ox \Mb\right)_{\mones}^*.\nn\ee
In fact, for any subset $I\subseteq \{1,\dots,m\}$, we have the linear isomorphism $\psis=\psis^I_p$ of \eqref{genpsis}, and therefore a bilinear form $\psis_p^*(\Geom^{|I|-p})$ on $(\Wedge p \nbs \ox \Mb)_{\wtI I}$ and linear map
\be (\psis^I_p)^{-1} \circ\Ghom^{|I|-p}\circ \psis^I_p:  \left(\Wedge p \nbs \ox \Mb\right)_{\wtI I} \to \left( \Wedge p \nbs \ox \Mb\right)_{\wtI I}^*.\nn\ee
When there is no ambiguity, we shall sometimes denote these linear maps (for any $p$ and any $I$) simply by $\Ghom$. Similarly we will sometimes use $\Geom$ to refer to the corresponding symmetric bilinear forms. 

In particular when $p=1$ and we restrict to the subspace $\nbs_{\wtI I} \ox v_0 \ox v_1 \ox \dots \ox v_N\cong \nbs_{\wtI I}$ we get a linear map 
\be \Ghom := \psis^{-1} \circ\Ghom^{|I|-p}\circ \psis \quad:\quad \nbs_{\wtI I} \to \nbs^*_{\wtI I},\nn\ee 
cf. \eqref{compsis}. And when $p=0$ and we restrict to the subspace $v_0 \ox \dots \ox v_{i-1} \ox (M_i)_{\wtI I} \ox v_{i+1} \ox \dots v_N \cong (M_i)_{\wtI I}$ we get a linear map $\Ghom : (M_i)_{\wtI I}\to (M_i^*)_{\wtI I}$ for each $i=0,\dots,N$.

\subsection{Shapovalov form}
Let $\cai: \gu \to \gu$ be the anti-automorphism defined by 
\be \cai(\e ki) = \f ki,\quad \cai(\f ki) = \e ki,\quad \cai( \hh ki) = \hh ki.\nn\ee 
It restricts to an anti-automorphism of $\gus$.
There is a unique bilinear form $\Sa_i$ on $M_i$, $i=1,\dots,N$, such that
\be \Sa_i(v_i,v_i) = 1,\qquad \Sa_i( X \on v, w) = \Sa_i( v,\cai(X)\on w).\nn\ee
There is a unique bilinear form $\Sa_\nb$ on $\nb$ such that 
\be \Sa_\nb(\f ki , \f lj) = \delta_{kl} \delta_{ij}, \qquad \Sa_\nb( [ X,Y], Z ) = -\Sa(Y,[\cai(X),Z]) .\nn\ee
It restricts to a bilinear form on $\nbs$.

Define 
\be \Mbnz  := \bigotimes_{i=1}^N M_i \hookrightarrow \Mb; \quad x_1\ox\dots \ox x_N \mapsto v_0 \ox x_1 \ox \dots \ox x_N .\nn\ee
Let $\Sa^p$ be the bilinear form on the space $\Wedge p \nbs \ox \Mbnz$, $p = 0,1,\dots,m$, defined by
\begin{multline} \Sa^p(\Pin a_p \wx \dots \wx \Pin a_1 \ox x_1 \ox \dots \ox x_N,
\Pin b_p \wx \dots \wx \Pin b_1\ox y_1 \ox \dots \ox y_N)\\
= \det(\Sa_{\nb}(\Pin a_i,\Pin b_j)_{1\leq i,j\leq p}) \prod_{i=1}^N \Sa_i(x_i,y_i). \label{Spdef}\end{multline}
By restriction, we get a bilinear form $\Sa^p$ on the weight subspace $(\Wedge p \nbs \ox \Mbnz)_{\mones}$.

\subsection{The forms $G$ and $\Sa$ coincide on flags with no zero ends}
Recall from \S\ref{sec:nozero} the notion of flags with no zero ends. By definition of the map $\psis_p$, a monomial in $\bm x\in\left(\Wedge p \nbs \ox \Mb \right)_{\mones}$ belongs to the subspace $\left(\Wedge p \nbs \ox \Mbnz\right)_{\mones}$ if and only if its image $\psis_p(\bm x)\in \Fl^{m-p}(\cca)$ is a flag with no zero ends.

\begin{thm}\label{SGthm} On $\left(\Wedge {m-p} \nbs \ox \Mbnz\right)_{\mones}$, we have the equality $\Sa^{m-p}=  (-1)^m T^{m-p} \psis^*_{m-p}(\Geom^{p})$ of symmetric bilinear forms.
\end{thm}
\begin{proof} Suppose $\bm x,\bm y$ are monomials in $ \left(\Wedge {m-p} \nbs \ox \Mbnz\right)_{\mones}$, with corresponding flags 
\be \psis_{m-p}(\bm x) = F = (L^0 \supset \dots \supset L^p), \qquad \psis(\bm y) = F'. \nn\ee
Let $\kv\cc_{\bm z;m}$ be a discriminantal arrangement adapted to the last edge $L^p$ of $F$ as in \S\ref{sec:nozero}. 
Define
\be \td f_i := \f {k(i)} i, \quad \td h_i := \hh {k(i)} i, \quad \td e_i := \e {k(i)} i,  \qquad i=1,\dots,m.\label{tgs}\ee
These generate a subalgebra $\td \gu = \td\nb \oplus \td\gu_0 \oplus \td\nb^+$ of $\gu = \nb \oplus \gu_0 \oplus \nb^+$. Let $\td M_i:= U(\td\gu) \ox_{U(\td\gu_0\oplus \td\nb^+)} \C v_i$, $i=1,\dots, N$, where we note that 
\be \td\nb_+\on v_i = 0, \qquad \td h_j \on v_i = a(\td H_{ji}) v_i,\nn\ee
with $\td H_{ji}$ as in \eqref{tdH}. Set $\td M_{\nz} := \bigotimes_{i=1}^N \td M_i$. 

It follows from the definition of the map $\psis_{m-p}$ that the monomial $\bm x$ is of the form 
\be \bm x = \Pin g_{m-p} \wx \dots \wx \Pin g_1 \ox 1 \ox x_1 \ox \dots x_N\nn\ee
for some commutators $g_1,\dots,g_{m-p} \in \td\nb$ (\ie commutators in the generators $\td f_i$), and some monomial $x_1 \ox \dots \ox x_N\in \td M_\nz$. 

We have subalgebras $\tau^k \td\nb\subset \nb$, $k\in \ZT$. They are mutually orthogonal with respect to the Shapovalov form $\Sa_\nb$ on $\nb$. Therefore for any commutators $g,g' \in \td\nb$ we have
\be \Sa_{\nb}(\Pin g, \Pin g') = T \Sa_{\td\nb}( g,g' ), \nn\ee
where $\Sa_{\td\nb}$ denotes the Shapovalov form on $\td\nb$. 
On the other hand if $g$ is a commutator in $\td\nb$ and $g''$ is a commutator in $\nb \setminus \tau^\Z\td\nb$ then $\Sa_{\nb}(\Pin g, \Pin g'')=0$.
Similar statements hold for the factors $M_i$, $i=1,\dots,N$. 

Hence $\Sa^{m-p}(\bm x, \bm y)$ is zero unless the monomial $\bm y$ can also be expressed in the form
\be \bm y = \Pin g'_{m-p} \wx \dots \wx \Pin g'_1 \ox 1 \ox y_1 \ox \dots y_N\label{yfr}\ee
for some commutators $g'_1,\dots,g'_{m-p} \in \td\nb$ and some monomial $y_1 \ox \dots \ox y_N\in \td M_\nz$, and if $\bm y$ is of this form then
\be \Sa^p(\bm x, \bm y) = T^{m-p} \det(\Sa_{\td\nb}(g_i, g'_j)_{1\leq i,j\leq m-p}) \prod_{i=1}^N \Sa_i(x_i,y_i).\ee 
Now, up to the factor of $T^{m-p}$, the right-hand side is the usual Shapovalov form for $\Wedge p\td \nb \ox \td M_\nz$. Therefore we may apply \cite[Theorem 6.6]{SV} and conclude that
\be \Sa^p(\bm x, \bm y) = T^{m-p} (-1)^m \bar G^p(F,F'),\nn\ee 
where $\bar G^p(F,F')$ is the bilinear form of the discriminantal arrangement $\kv\cc_{\bm z;m}$. 

At the same time,  Lemma \ref{Glem} states that $G^p(F,F')$ is zero unless $F'\in \Flag^p(\kv\cc_{\bm z;m})$, in which case $G^p(F,F')=\bar G^p(F,F')$. And $F'\in \Flag^p(\kv\cc_{\bm z;m})$ if and only if the monomial $\bm y$ can be expressed in the form \eqref{yfr}. Thus we have the result.
\end{proof}

It will be useful to note the following special case of this argument.
\begin{cor}\label{kercor} On $\nbs$, we have the equality $\Sa_\nb=G$ of symmetric bilinear forms. In particular, for all $x\in \nbs$, if $x\in \ker \Sa_\nb$ then $x\in \ker \Ghom$. \qed
\end{cor}

\subsection{On the kernel of the geometric form}
\begin{lem}\label{anyzerolem}
Given an element 
$\bm x := y_1\wx\dots \wx y_p\otimes  w_0 \ox w_1\ox\dots\ox w_N\in \Wedge p \nbs \ox \Mb$, if any one of $y_1$, \dots $y_p$, $w_0$, $w_1$, \dots $w_N$ lies in the kernel of $\Ghom$, then $\bm x$ lies in the kernel of $\Ghom$.
\end{lem}
\begin{proof} By inspection of the definition of $\psis$ and $\Ghom$. 
\end{proof}

\begin{lem}\label{tozerolem}
If $y \ox w_0\in \nbs \ox M_0$ lies in the kernel of $\Ghom$, then $yw_0\in M_0$ lies in the kernel of $\Ghom$.
\end{lem}
\begin{proof}
This follows from Theorem \ref{cmap}. Indeed, we have $\psis(y_1 w_0) = d \psis(y_1 \ox w_0)$ and hence, by that theorem, $\Ghom(\psis(y_1 w_0)) = \Ghom(\psis(y_1 \ox w_0 )) \wx \left(\sum_{H\in \C} a(H) H\right)$. By the previous lemma, this is zero if $y_1$ lies in the kernel of $\Ghom$.
\end{proof}

\begin{thm} Suppose $\bm x\in \left(\Wedge p \nbs \ox \Mb\right)_{\mones}$ is of the form 
\be \bm x := y_1\wx\dots \wx  y_p\otimes  \left(x_1\dots  x_i\dots  x_l v_0\right)\ox w_1\ox\dots\ox w_N,\nn\ee
where   $x_1,\dots, x_i,\dots  x_l\in \nbs $ and $w_i\in M_i, i=1,\dots,N.$  
Assume that $x_i$ lies in the kernel of the Shapovalov form $\Sa_{\nb}$. 
Then $\bm x$ lies in the kernel of $\Ghom$.
\end{thm}
\begin{proof} By Corollary \ref{kercor}, $x_i\in \ker \Ghom$. Lemma \ref{anyzerolem} (with $N=0$) then implies that  $x_i \otimes  \left(x_{i+1} \dots  x_l v_0\right)\in \ker \Ghom$. Hence by Lemma \ref{tozerolem}, $\left(x_ix_{i+1} \dots  x_l v_0\right)\in \ker \Ghom$. Therefore by Lemma \ref{anyzerolem}, $x_{i-1} \ox \left(x_ix_{i+1} \dots  x_l v_0\right)\in \ker \Ghom$. By using Lemmas \ref{tozerolem} and \ref{anyzerolem} repeatedly in this way, we conclude that $\left(x_1 \dots x_i \dots  x_l v_0\right)\in \ker \Ghom$. The result follows, by one final invocation of Lemma \ref{anyzerolem}.
\end{proof}

\subsection{Shortened flags and linking hyperplanes}\label{sandl}
Given a subset $J\subset \{1,\dots,m\}$ and an edge $L$ of the arrangement $\cca$, define the edge $L|_J$ as follows.
Recall that we can write
$L= \kv L_{I_1}^{i_1} \cap \dots \cap \kv L_{I_r}^{i_r} \cap \kv L_{J_1} \cap \dots \cap \kv L_{J_p}$, as in \eqref{arp}. Let $L|_J := \kv L_{I_1\cap J}^{i_1} \cap \dots \cap \kv L_{I_r\cap J}^{i_r} \cap \kv L_{J_1\cap J} \cap \dots \cap \kv L_{J_p\cap J}$.
Now given a flag $F = (L^0 \supset L^1\supset \dots \supset L^p)$ define the flag $F|_J := (L^0|_J \supseteq L^1|_J\supseteq \dots \supseteq L^p|_J)$. Note that $F|_J$ has at most $p$ steps. Call $F|_J$ a \emph{shortened flag}. 

\begin{exmp} If $J = \{1,2,4\} \subset \{1,2,3,4\}$ and $F$ is the flag $(t_1=t_2) \supset (t_1=t_2 =t_3) \supset (t_1=t_2=t_3=t_4)$ then $F|_J$ is the flag $(t_1=t_2) \supset (t_1=t_2=t_4)$. 
\end{exmp}

Let $I,J\subset \{1,\dots,m\}$ be disjoint subsets. 
We say a hyperplane $\tilde H$ \emph{links} $I$ and $J$ if $\tilde H$  is $(t_i=\omega^k t_j)$ for some $i\in I$, $j\in J$, $k\in \ZT$. (That is, if $\tilde H$ is one of the $\Ht ijk$ with $i\in I$, $j\in J$.)

\subsection{Commutator lemmas in non-cyclotomic case}
In this subsection we let $T=1$, \ie we consider the non-cyclotomic case.  

Let $I_1,I_2\subset \{1,\dots,m\}$ be disjoint non-empty subsets. Let $b\in \nb$ be a commutator of weight 
$\wt b = \wtI {I_1 \sqcup I_2}$. The element 
$\psis(b) \in \Fl^{|I_1|+|I_2|-1}(\cco{I_1\sqcup I_2})$, cf. \eqref{compsis}, is a flag multiplied by a certain sign $\pm 1$,
\be\psis(b) = \pm F, \quad F=(L^0\supset L^1\supset\dots\supset L^{|I_1|+|I_2|-1})\in \Flag^{|I_1|+|I_2|-1}(\cco {I_1\sqcup I_2}).\nn\ee
For convenience we will refer such elements $\pm F$ simply as flags. We extend the notions of adjacency-with-sign (\S\ref{pf}) and shortening (\S\ref{sandl}) to such elements by linearity. In particular we have the shortened flags
\be \psis(b)|_{I_1} := \pm F|_{I_1}, \qquad
\psis(b)|_{I_2} := \pm F|_{I_2}. \nn\ee

Recall from \S\ref{efs} that we can regard each edge $L^j$ of the flag $\psis(b)$ as an archipelago of islands. These islands are all swimming since $b$ is a commutator. At each step $L^{j-1} \supset L^j$ two islands are joined.  

Let $k$ be the smallest index such that the edge $L^k$ is contained in some hyperplane which links $I_1$ and $I_2$. 
That is, let $k$ be smallest such that for some nonempty subsets $J_1\subset I_1$ and $J_2\subset I_2$, the edge $L^{k}$ has an island $L_{J_1\sqcup J_2}$.
Denote by \be\mc H = \{ (t_i=t_j): i\in J_1, j\in J_2\} \label{hcdef}\ee the set of all hyperplanes linking $J_1$ and $J_2$.

\begin{lem}
\label{lu}
Let $(H_1,\dots,H_{|I_1|-1})$ be a tuple of hyperplanes in $\cco {I_1}$.
Let $(H'_1,\dots,H'_{|I_2|-1})$ be a tuple of hyperplanes in $\cco {I_2}$. Let $\tilde H\in \cco {I_1\sqcup I_2}$ be a hyperplane linking $I_1$ and $I_2$.
Then the tuple
$(H_1,\dots,H_{|I_1|-1},\, H'_1,\dots,H'_{|I_2|-1},\,\tilde H)$ is adjacent to $\psis(b)$ if and only if
$(H_1,\dots,H_{|I_1|-1})$ is adjacent to $\psis(b)|_{I_1}$,  $(H'_1,\dots,H'_{|I_2|-1})$ is adjacent to $\psis(b)|_{I_2}$, and $\tilde H$ links $J_1$ and $J_2$.
\end{lem}
\begin{proof}
The ``only if'' direction is immediate from the definition of the shortened flags. 
For the ``if'' direction, note that the step $L^{k-1} \supset L^k$ of the flag $\psis(b)$ is the only one in which an island (see \S\ref{efs}) $L_{J_1}$ with $J_1\subset I_1$ is joined to an island $L_{J_2}$ with $J_2\subset I_2$. At this step we have $L^k = L^{k-1} \cap \tilde H$. At every other step $L^{j-1} \supset L^j$, $j\neq k$, we have $L^j = L^{j-1} \cap H$ for some $H \in \cco {I_1}$ or $H\in \cco {I_2}$. By definition of the shortened flags, a suitable hyperplane $H$ can always be drawn from the set $\{H_1,\dots,H_{|I_1|-1},H'_1,\dots,H'_{|I_2|-1}\}$.
\end{proof}

Now let $(H_1,\dots,H_{|I_1|-1})$ be a tuple of hyperplanes in $\cco {I_1}$ adjacent to $\psis(b)|_{I_1}$. We associate to any such hyperplanes a sign, $c(H_1,\dots,H_{|I_1|-1})$, defined as follows.
Suppose $\tilde H\in \mc H$. Suppose $(H'_1,\dots,H'_{|I_2|-1})$ is adjacent to the shortened flag $\psis(b)|_{I_2}$, with sign say $(-1)^t$. By Lemma \ref{lu}, the tuple $(H_1,\dots,H_{|I_1|-1},\, H'_1,\dots,H'_{|I_2|-1},\,\tilde H)$ is adjacent to $\psis(b)$ with some sign, say $(-1)^p$. 
\begin{lem}\label{lu2} The sign $(-1)^{p+t}$ is independent of the choice of $(H'_1,\dots,H'_{|I_2|-1})$ and $\tilde H$. \qed
\end{lem}
\begin{proof}
Call a step $L^{j-1} \supset L^j$ of the flag $\psis(b)$ black, $\bl$, if $L^j = L^{j-1} \cap H$ for some $H\in \cco{I_2}$ and white, $\wh$, otherwise. Reading the steps of the flag $\psis(b)$ in order we obtain a tuple $\mathsf T\in \{\bl,\wh\}^{|I_1|+ |I_2|-1}$ in which each entry is $\bl$ or $\wh$. 
There is a unique shuffle (\ie a permutation which does not alter the ordering of $\bl$'s or the ordering of $\wh$'s) which sends 
\be (\underbrace{\wh,\dots,\wh}_{|I_1|-1}, \underbrace{\bl,\dots,\bl}_{|I_2|-1},\wh) \to \mathsf T.\nn\ee
Let $(-1)^v$ be the sign of this shuffle. Let $(-1)^u$ be the sign with which $(H_1,\dots,H_{|I_1|-1})$ is adjacent to $\psis(b)|_{I_1}$. Then we have $(-1)^p = (-1)^{u+v+t}$. Hence $(-1)^{p+t}= (-1)^{u+v}$, and written in this form it is manifestly independent of the choice of  $(H'_1,\dots,H'_{|I_2|-1})$ and $\tilde H$. 
\end{proof}

Thus we may define the sign $c(H_1,\dots,H_{|I_1|-1}) = (-1)^{p+t}$.

\begin{prop}\label{comprop}
Let $a,b\in \nb$ be commutators such that $\wt a = \wtI {I_1}$ and $\wt b = \wtI {I_1\sqcup I_2}$.
Then
\be \psis([\cai(a),b]) + \psis(b)|_{I_2} \left(\sum_{(H_1,\dots,H_{|I_1|-1})}  c(H_1,\dots,H_{|I_1|-1}) a(H_1) \dots a(H_{|I_1|-1}) \right) \sum_{\tilde H\in \mc H} a(\tilde H)\label{ppe}\ee
lies in the kernel of $\Ghom$. Here $\mc H$ is as in \eqref{hcdef}, and the first sum is over all tuples $(H_1,\dots,H_{|I_1|-1})$, modulo reordering, of hyperplanes in $\cco {I_1}$ adjacent to $\psis(a)$ with sign  $+1$  and adjacent to $\psis(b)|_{I_1}$. 
\end{prop}
\begin{proof}
Let $c$ be any commutator of weight $\wt c = \wtI {I_2}$. 
It is shown in  \cite{SV} that $\Geom(\psis(\cdot), \psis(\cdot))  = \Sa(\cdot,\cdot)$ as an equality of bilinear forms on $\nb$, where $\Sa$ is the Shapovalov form. By definition of the latter, $\Sa([a,c],b) = -\Sa(c,[\cai(a),b])$. Hence, we have
\begin{align} &{} -\Geom( \psis(c), \psis( [\cai(a),b]))\nn\\  
&=  \Geom(\psis([a,c]), \psis(b))\nn\\
&= \sum_{(H'_1,\dots,H'_{|I_2|-1})}\sum_{(H_1,\dots,H_{|I_1|-1})} \sum_{\tilde H\in \mc H} (-1)^p a(H_1) \dots a(H_{|I_1|-1}) a(H'_1) \dots a(H'_{|I_2|-1})  a(\tilde H) ,\nn\end{align}
where:
\begin{enumerate}[(1)]
\item the outer sum is over all tuples $(H'_1,\dots,H'_{|I_2|-1})$, modulo reordering, of hyperplanes in $\cco {I_2}$ adjacent to $\psis(c)$ with sign $+1$ and adjacent to $\psis(b)|_{I_2}$; 
\item the middle sum is over all tuples $(H_1,\dots,H_{|I_1|-1})$, modulo reordering, of hyperplanes in $\cco {I_1}$ adjacent to $\psis(a)$ with sign $+1$ and adjacent to $\psis(b)|_{I_1}$; and 
\item $(-1)^p$ is the sign with which the tuple $(H_1,\dots,H_{|I_1|-1}, H'_1,\dots,H'_{|I_2|-1}, \tilde H)$ is adjacent to the flag $\psis(b)$. 
\end{enumerate}
Indeed, by definition of $\Geom$, \S\ref{sg}, we are to sum over all tuples, modulo reordering, of hyperplanes adjacent to both flags $\psis(b)$ and $\psis([a,c])$. By definition of $\psis([a,c])$, every tuple adjacent to to $\psis([a,c])$ is of the form $(H_1,\dots,H_{|I_1|-1}, H'_1,\dots,H'_{|I_2|-1}, \tilde H)$ for some $\tilde H$ linking $I_1$ and $I_2$, some $H_1,\dots,H_{|I_1|-1} \in \cco {I_1}$, and some $H'_1,\dots,H'_{|I_2|-1} \in \cco {I_2}$. Lemma \ref{lu} then tells us which of these tuples are also adjacent to $\psis(b)$. 

Now, in view of Lemma \ref{lu2}, we may factor the expression above as
\begin{multline}  \left(\sum_{(H'_1,\dots,H'_{|I_2|-1})} (-1)^t a(H'_1) \dots a(H'_{|I_2|-1}) \right)\\
   \times   \left( \sum_{(H_1,\dots,H_{|I_1|-1})}c(H_1,\dots,H_{|I_1|-1}) a(H_1) \dots a(H_{|I_1|-1})  \sum_{\tilde H\in \mc H}  a(\tilde H)\right)
  \nn \end{multline}
where $(-1)^t$ is the sign with which $(H'_1,\dots,H'_{|I_2|-1})$ is adjacent to $\psis(b)|_{I_2}$. But now observe that the left factor above is $\Geom(\psis(c), \psis(b)|_{I_2})$.
So we have shown that
\be -\Geom( \psis(c), \psis( [\cai(a),b])) = \Geom(\psis(c), \psis(b)|_{I_2}) \sum_{(H_1,\dots,H_{|I_1|-1})}  c(H_1,\dots,H_{|I_1|-1}) a(H_1) \dots a(H_{|I_1|-1})\sum_{\tilde H\in \mc H}  a(\tilde H).\nn\ee
Thus \eqref{ppe} lies in the kernel of $\Geom(\psis(c),\cdot)$, for any commutator $c\in \nb_{\wtI {I_2}}$. This suffices to show that it lies in the kernel of $\Ghom$. 
\end{proof}

\subsection{Definition of $\Gm$}\label{gmsec}
We return to the case of general $T\in \Z_{\geq 1}$, \ie to the general cyclotomic case. 

Let $I\subset \{1,\dots m\}$. Recall that $\cco I$ denotes the arrangement consisting of just the ``diagonal'' hyperplanes $\Ht {i_1}{i_2}k$, $i_1,i_2\in I$, $k\in \ZT$.
Define a bilinear form on $\Gm$ on $\nbs_{\wtI I}$ by
\be \Gm(\Pin a, \Pin b):= G_{\cco I}(\psis(\Pin a v_0), \psis(\Pin b v_0))\label{G0def}\ee 
where $G_{\cco I}$ is the bilinear form for the arrangement $\cco I$ of ``diagonal'' hyperplanes. 

Equivalently, but more explicitly, 
for two projected commutators $\Pin a, \Pin b \in \nbs$ of equal weight $\wtI I$, define  
\be \Gm(\Pin a, \Pin b) := \sum_{(H_1,\dots,H_{p+1})} (-1)^s a(H_1) \dots a(H_{p+1}) \nn\ee
where the sum is over all tuples $(H_1,\dots,H_{p+1})$, modulo reordering, of hyperplanes belonging to the arrangement $\cco I$, such that 
\begin{enumerate}
\item $(H_1,\dots,H_{p+1})$ is adjacent to the flag $\psis(\Pin a v_0)$, with sign +1,
\item $(H_1,\dots,H_{p+1})$ is adjacent to the flag $\psis(\Pin b v_0)$, with sign say $(-1)^s$.
\end{enumerate} 
Note that $\psis(\Pin a v_0)$ is the flag $\psis(\Pin a)$ completed by the last edge $L_I^0$. 

\begin{exmp} For $k,l\in \ZT$, 
\be \Gm(\Pin{[\f 0 1, \f k 2]}, \Pin{[\f 0 1, \f l 2]}) 
 =  a(\Ht 12 k) \left( - a(\Ht 12 l) + \delta_{kl} \sum_{p\in \ZT} a(\Ht 12 p)\right)\nn \ee
\end{exmp}

\begin{prop}\label{Gmprop} Let $J\subsetneq I$. For all $\Pin a \in \nbs_{\wtI J}, \Pin b\in \nbs_{\wtI I}$ and  $\Pin c \in \nbs_{\wtI {I\setminus J}}$,
\be  \Gm( \Pin b, [ \Pin c, \Pin a]) + \Gm( \Pin c, [\cai(\Pin a), \Pin b])
                            + \Gm(\Pin a, [ \Pin b, \cai(\Pin c)])=0.\nn\ee
\end{prop}
\begin{proof}
It is enough to consider the case that $\Pin a, \Pin b, \Pin c$ are projected commutators. 
Then
 $[\Pin c, \Pin a] = \sum_{k\in \ZT} \Pin{[c , \tau^k a]}$ and we have
\be \Gm( \Pin b, [\Pin c, \Pin a]) = \sum_{k\in \ZT} \Gm(\Pin b,\Pin{[c , \tau^k a]}) 
= \sum_{\substack{(H_1,\dots,H_{|J|})\\ (H'_1,\dots,H'_{|I\setminus J|}) \\ H_\bl, H_\wh }} (-1)^p a(H_1) \dots a(H_{|J|}) a(H'_1) \dots a(H'_{|I\setminus J|})  a(H_\bl) a(H_\wh).\label{bca1}\ee
Here the sum is over all tuples $(H_1,\dots,H_{|J|})$, modulo reordering, of hyperplanes in $\cco J$, all tuples $(H'_1,\dots, H'_{|I\setminus J|})$, modulo reordering, of hyperplanes in $\cco {I\setminus J}$, and all hyperplanes $H_\bl ,H_\wh\in \cco I$, such that:
\begin{enumerate}
\item $(H_1,\dots,H_{|J|}, H'_1,\dots,H'_{|I\setminus J|}, H_\bl)$ is adjacent to $\psis([c, \tau^k a])$ for some $k\in \ZT$, with sign $+1$,
\item  $H_1 \cap \dots \cap H_{|J|} \cap  H'_1 \cap \dots \cap H'_{|I\setminus J|} \cap H_\bl \cap H_\wh = \{0\}$
\item $(H_1,\dots,H_{|J|}, H'_1,\dots,H'_{|I\setminus J|}, H_\bl,H_\wh)$ is adjacent to the flag $\psis(\Pin bv_0)$, with sign say $(-1)^p$.
\end{enumerate}  
Note that (1) implies in particular that the hyperplane $H_\bl$ links $I\setminus J$ and $J$. The hyperplane $H_\wh$ is of one of three types:
\begin{enumerate}[(i)]
\item $H_\wh$ (also) links $I\setminus J$ and $J$,
\item $H_\wh$ belongs to the arrangement $\cco {I\setminus J}$,
\item $H_\wh$ belongs to the arrangement $\cco J$.
\end{enumerate}

Consider case (i). Then the sum also includes the term with $H_\wh$ and $H_\bl$ interchanged, and it comes with the opposite sign. Indeed, $(H_1,\dots,H_{|J|}, H'_1,\dots,H'_{|I\setminus J|}, H_\wh)$ is adjacent to $\psis([c, \tau^k a])$ for some (different) value of $k\in \ZT$, again with sign 1. Condition (2) still holds with $H_\wh$ and $H_\bl$ interchanged. And so does condition (3), but with an additional factor of $(-1)$. Therefore such terms cancel pairwise in the sum.

In the remainder of the proof we shall show that the terms in the sum of the types  (ii) and (iii) cancel against $\Gm(\Pin c, [\cai(\Pin a), \Pin b] )$ and  $\Gm(\Pin a, [ \Pin b, \cai(\Pin c)])$ respectively. 

Consider $\Gm( \Pin c, [\cai(\Pin a), \Pin b)])$. We have $[\cai(\Pin a),\Pin b] = \sum_{k\in \ZT} \Pin{[\cai(a), \tau^k b]}$. Consider any one term $[\cai(a),\tau^k b]$ in this sum. If the content, in the sense of \S\ref{cpc},  of the commutator $a$  is not a subset of the content of $\tau^k b$, then this term is evidently zero, since $[\e lj, \f ki]$ is zero unless $k=l$ and $i=j$. So suppose that the content of $a$ is a subset of the content of $\tau^k b$. (This happens for at most one value of $k$). As in \S\ref{sec:nozero}, we may consider the (non-cyclotomic) discriminantal arrangement adapted to the last edge of the flag $\psis(b)= \psis(\tau^k b) := \psis(\Pin b)$.  Then by Proposition \ref{comprop} we find that
\be \psis([\cai( a),\tau^k b]) + 
 \psis(\tau^k b)|_{I\setminus J} \sum_{\substack{(H_1,\dots,H_{|J|})\\ H_\bl}} (-1)^{p+t} a(H_1) \dots a(H_{|J|})  a(H_\bl)\in \ker \Ghom\label{ba1}\ee
where the sum is over all hyperplanes $(H_1,\dots,H_{|J|})$, modulo reordering, of hyperplanes in $\cco J$ and all hyperplanes $H_\bl$ that link $I\setminus J$ and $J$, such that
\begin{enumerate}
\item $(H_1,\dots,H_{|J|})$ is adjacent to the flag $\psis(a)$, with sign $+1$,
\item for some hyperplanes $H'_1,\dots,H'_{|I\setminus J|}$ in $\cco {I\setminus J}$, the tuple $(H_1,\dots,H_{|J|},H'_1,\dots,H'_{|I\setminus J|},  H_\bl)$ is adjacent to the flag $\psis(\tau^k b) = \psis(\Pin b)$, with sign say $(-1)^p$.
\end{enumerate} 
The sign $(-1)^t$ is the sign of the permutation that puts these $H'_1,\dots,H'_{|I\setminus J|}$ into the order in which they appear in the flag $\psis(\Pin b)$. 

We have Lemmas \ref{anyzerolem} and \ref{tozerolem} for the arrangement $\cco I$. Therefore \eqref{ba1} is sufficient to ensure that 
\be [\cai( a),\tau^k b] + 
 \psis^{-1}(\psis(\tau^k b)|_{I\setminus J} \sum_{\substack{(H_1,\dots,H_{|J|})\\ H_\bl}} (-1)^{p+t} a(H_1) \dots a(H_{|J|})  a(H_\bl)\label{ba2}\ee
lies in the kernel of $\Gm( \Pin c, \cdot)$.

Now by definition of $\Gm$, 
\be \Gm\big(\Pin c, \psis^{-1}(\psis(b)|_{I\setminus J})\big) 
= \sum_{\substack{(H'_1,\dots,H'_{|I\setminus J|})\\ H_\wh}} (-1)^t a(H'_1) \dots a(H'_{|I\setminus J|}) a(H_\wh) \label{cb1}\ee
where the sum is over all tuples $(H'_1,\dots,H'_{|I\setminus J|})$, modulo reordering, of hyperplanes in $\cco {I\setminus J}$, and all hyperplanes $H_\wh \in \cco {I\setminus J}$, such that 
\begin{enumerate}
\item $(H'_1,\dots,H'_{|I\setminus J|})$ is adjacent to the flag $\psis(c)$, with sign $+1$,
\item $H'_1\cap\dots\cap H'_{|I\setminus J|}\cap H_\wh = \{0\}$,
\item $(H'_1,\dots,H'_{|I\setminus J|}, H_\wh)$ is adjacent to the flag $\psis(\Pin b v_0)|_{I\setminus J}$, with sign say $(-1)^t$.
\end{enumerate} 

Inspecting \eqref{ba2} and \eqref{cb1}, one indeed sees that  $\Gm(\Pin c, [\cai(\Pin a), \Pin b] )$ is equal and opposite to the sum of terms of type (ii) in the expression \eqref{bca1} for $\Gm( \Pin b, [\Pin c, \Pin a])$.

The argument that terms of type (iii) cancel against $ \Gm(\Pin a, [ \Pin b, \cai(\Pin c)])$ is similar.
\end{proof}

\subsection{Definition of $\Cou$}\label{coudef}
Let $\nb_{\leq 1}$ denote the subspace
\be \nb_{\leq 1} := \bigoplus_{I\subset \{1,\dots,m\}} \nb_{\wtI I}.\nn\ee
That is, $\nb_{\leq 1}$ is the span of all commutators in which  for each $i\in \{1,\dots,m\}$ at most one of the generators $\f ki$, $k\in \ZT$, appears.
Define $(\nb_+)_{\leq 1}$ similarly.
Recall $\gu = \nb \oplus \gu_0 \oplus \nb_+$. Let $\gu_{\leq 1} := \nb_{\leq 1} \oplus \gu_0 \oplus (\nb_+)_{\leq 1}$.

Define a skew-symmetric bilinear form $\Cou:\gus_{\leq 1}\times \gus_{\leq 1} \to \C$ as follows.
For all $I\subset\{1,\dots,m\}$ and all $x\in (\nbs_+)_{\wtI I}$ and $y\in \nbs_{\wtI I}$, set
\be -\Cou(y,x) = \Cou(x,y) := \Gm(\cai(x),y)).\nn\ee 
On all other components in the weight decomposition of $\gus_{\leq 1}\times \gus_{\leq 1}$, set $\Cou(x,y) = 0$.

\begin{prop}\label{kercouprop} Let $x\in \nbs$. If $x\in \ker \Sa_\nb$ then $\Cou(x,y)=0$ for all $y\in \gus$. 
\end{prop}
\begin{proof} If $x \in \ker \Sa_\nb$ then $x\in \ker \Ghom$ as in Corollary \ref{kercor}. Lemma \ref{tozerolem} holds for the arrangement $\cco I$. In view of the definition of $\Gm$, the result follows. 
\end{proof}
\begin{thm}\label{couthm} We have the equality
\be \Cou(x,[y,z]) + \Cou(y,[z,x]) +\Cou(z,[x,y]) = 0\nn\ee
for all $x,y,z\in \gus_{\leq 1}$. 
\end{thm}
\begin{proof} 
When $x\in (\nbs_+)_{\wtI I}$, $y\in \nbs_{\wtI I}$ and $z\in \gus_0$ the third term vanishes and the first two cancel.

Consider the case that  $x\in (\nbs_+)_{\wtI I}$, $y\in \nbs_{\wtI J}$ and $z \in \nbs_{\wtI {I\setminus J}}$ where $J\subsetneq I\subset \{1,\dots,m\}$. It is enough to suppose that $x,y,z$ are projected commutators. Say $x= \cai(\Pin b)$, $y=\Pin c$ and $z=\Pin a$ for commutators $a,b,c\in \nb$. Then what has to be shown is that
\begin{align} 0 &=\Cou( \cai(\Pin b),[\Pin c,\Pin a]) - \Cou([\Pin a,\cai(\Pin b)],\Pin c) -\Cou([\cai(\Pin b),\Pin c],\Pin a) \nn\\
&=\Gm( \Pin b, [ \Pin c, \Pin a]) - \Gm( \cai([\Pin a, \cai(\Pin b)]), \Pin c)
                            - \Gm( \cai([ \cai(\Pin b), \Pin c]), \Pin a) \nn\\
&=  \Gm( \Pin b, [ \Pin c, \Pin a]) + \Gm( [\cai(\Pin a), \Pin b)], \Pin c)
                            + \Gm( [ \Pin b, \cai(\Pin c)], \Pin a)\nn\\
&=  \Gm( \Pin b, [ \Pin c, \Pin a]) + \Gm( \Pin c, [\cai(\Pin a), \Pin b)])
                            + \Gm(\Pin a, [ \Pin b, \cai(\Pin c)]),\nn\end{align}
which is true, by Proposition \ref{Gmprop}.
\end{proof}
\begin{rem} One ``wants'' to say that $\Cou$ is a cocycle, and then use it to define a central extension of the Lie algebra $\gus$. We do not quite have that statement, because $\Cou$ is defined only on the subspace $\gus_{\leq 1}$. After \emph{symmetrization} is introduced in \S\ref{symsec}, it will be possible to make such statements, in \S\ref{cocyclesec}. For now we proceed as follows.
\end{rem}

Let $\gus\oplus \C \K$ denote the trivial extension of $\gus$ by a one-dimensional centre $\C \K$. 
We make $M_0$ into a module over $\gus \oplus \C \K$ by declaring that $\K \on v_0 = v_0$. 
Define a ``modified commutator'', namely a bilinear map 
$[\cdot,\cdot]' : (\nbs_+)_{\leq 1} \times \nbs_{\leq 1} \to \gus \oplus \C \K$,
by
\be [x,y]' := [x,y] + \Cou(x,y)\K. \nn\ee
Define a ``modified action'' $\triangleright$ of $\gus_{\leq 1}$ on $(M_0)_{\leq 1}$ recursively as follows. 
For all $x\in \nbs_{\leq 1}$ and $w\in (M_0)_{\leq 1}$, set $x\triangleright w = xw$. 
For all $x\in (\nbs_+)_{\leq 1}$, set $x \triangleright v_0 := 0$ and 
\be x \triangleright y w  := [x,y]'\triangleright w + y (x \triangleright w) \nn\ee
for all $y\in \nbs_{\leq 1}$ and $w\in (M_0)_{\leq 1}$ such that $yw \in (M_0)_{\leq 1}$. Theorem \ref{couthm} ensures that this is well-defined, \ie $x \triangleright [y_1,y_2] w = x \triangleright y_1 y_2 w - x \triangleright y_2 y_1 w$, and respects the bracket in $\gus_{\leq 1}$, \ie $[x_1,x_2] \triangleright w = x_1 \triangleright x_2 \triangleright w - x_2 \triangleright x_1 \triangleright w$ whenever $x_1,x_2,[x_1,x_2]\in \gus_{\leq 1}$. 
Let $\Se_0$  be the corresponding ``modified Shapovalov form'' on $(M_0)_{\leq 1}$. Namely, define a symmetric bilinear form $\Se_0 : (M_0)_{\leq 1} \times (M_0)_{\leq 1} \to \C$ by $\Se_0(v_0,v_0) = 1$ and then recursively,
\be \Se_0( x v, w) := \Se_0(v, \cai(x) \triangleright w) \nn\ee
for all $x\in \nbs_{\leq 1}$ and $v,w\in (M_0)_{\leq 1}$ such that $xv \in (M_0)_{\leq 1}$.

\begin{exmp} We have $\cai(\Pin{[\f 01,\f k2]}) = - \Pin{[\e 01,\e k2]}$ and hence
\begin{align} & \Se_0( \Pin{[\f 01,\f k2]} v_0, \Pin{ [\f 01, \f l2]} v_0 )\nn\\ 
&=  \Se_0( v_0 , \left[\Pin{[\e 01,\e k2]} \,\Pin{ [\f 01, \f l2]}\right]' v_0 ) \nn\\
&=  \Se_0( v_0 , \left(\delta_{kl}(\Pin h_1 + \Pin h_2) a(\Ht 12 k)
+ \K a(\Ht 12 k) \left( - a(\Ht 12 l) + \delta_{kl} \sum_{p\in \ZT} a(\Ht 12 p)\right) \right)
  v_0 ) \nn\\
&= -a(\Ht 12 k) a(\Ht 12 l) + \delta_{kl} a(\Ht 12 k) \left( a(\Ho 1) + a(\Ho 2) + \sum_{p\in \ZT} a(\Ht 12 p)\right).\nn
\end{align}
\end{exmp}

\begin{thm}\label{ScGthm} 
We have the equality $\Se_0 = G$ of bilinear forms on $(M_0)_{\leq 1}$.  
\end{thm}
\begin{proof}
Let $I_1,\dots,I_n \subset \{1,\dots,m\}$ be pairwise disjoint subsets and let $A_1,\dots,A_n\in \nb$ be commutators such that $\wt A_i = \wtI {I_i}$ for $i=1,\dots,n$. 
Let $J_1,\dots,J_p \subset \{1,\dots,m\}$ be pairwise disjoint subsets and let $B_1,\dots,B_n\in \nb$ be commutators such that $\wt B_i = \wtI {J_i}$ for $i=1,\dots,p$. 
It is enough to show that 
\be\Geom( \Pin A_1 \dots \Pin A_n v_0 , \Pin B_1 \dots \Pin B_p v_0 )
 = \Se_0  ( \Pin A_1 \dots \Pin A_n v_0 , \Pin B_1 \dots \Pin B_p v_0 ).\label{mse}\ee

We proceed by induction on $n+p$. The equality \eqref{mse} is true whenever $n=0$ or $p=0$.

Consider the inductive step. Suppose $n,p \geq 1$. Let us assume that $I_1\sqcup \dots \sqcup I_n =J_1\sqcup \dots \sqcup J_p$. (If not, both sides of \eqref{mse} are zero and we are done.)

Observe that for any permutation $s\in \Sigma_n$, the difference
\be \Pin A_1 \dots \Pin A_n - \Pin A_{s(1)} \dots \Pin A_{s(n)} \nn\ee
as an element of the universal enveloping algebra
is a linear combination of products of $<n$ projected commutators. (For example $\Pin A_1 \Pin A_2 - \Pin A_2 \Pin A_1 = \left[ \Pin A_1 , \Pin A_2 \right] = \sum_{k\in \ZT} \Pin{[ A_1, \tau^k A_2]}$.) Therefore, by the inductive assumption, we effectively have the freedom to permute the commutators $\Pin A_1, \dots, \Pin A_n$ at will. The same applies to $\Pin B_1,\dots \Pin B_p$.  
 
We note the following lemma.
\begin{lem}
Both $\Geom( \Pin A_1  \dots \Pin A_n v_0, \Pin B_1 \dots \Pin B_p v_0 )$ and 
$\Se_0( \Pin A_1 \dots \Pin A_n v_0, \Pin B_1  \dots \Pin B_p v_0 )$ are zero 
unless the following condition holds:
\be \text{There exists $j\in \{1,\dots,p\}$ such that $I_1 \subseteq J_j$ or $I_1\supseteq J_j$}. \label{nestcon}\ee 
\end{lem}
\begin{proof}
Suppose this condition does not hold. We must show that then both the geometric and Shapovalov forms are zero. 

For the Shapovalov form $\Se_0$, it is enough to observe that, for each $j$, the supposition that $I_1 \not\subseteq J_j$ and $I_1\not\supseteq J_j$ implies $[\cai(\Pin A_1), \Pin B_j]'= 0$.

Consider the geometric form $\Geom$. Let $I:= I_1\sqcup \dots \sqcup I_n =J_1\sqcup \dots \sqcup J_p$. 
Inspecting the definition of $\psis$ in Theorem \ref{p1}, we see that the flag $\psis(\Pin A_1  \dots \Pin A_n v_0)\in \Fl^{|I|}(\ccoo I)$ has last two edges 
\be \dots \supset \kv L_{I_1} \cap L^0_{I\setminus I_1} \supset L^0_{I},\nn\ee
for some map $\bm k: \{1,\dots,m\} \to \ZT$. Indeed, the last step is the one in which the swimming island (see \S\ref{efs}) corresponding to $\Pin A_1$ is fixed to zero.
 
To write the penultimate edge, $\kv L_{I_1} \cap L^0_{I\setminus I_1}$, as an intersection of hyperplanes requires   $|I\setminus I_1| + \left(|I_1|-1\right) = |I| -1$ hyperplanes, none of which can be from the following set
\be \{ \Ht ijk \}_{i\in I_1,j\in I\setminus I_1, k\in \ZT}.\label{ghs}\ee
So, by counting, any tuple of hyperplanes adjacent to the flag $\psis(\Pin A_1  \dots \Pin A_n v_0)\in \Fl^{|I|}(\ccoo I)$ contains at most one hyperplane from the set in \eqref{ghs}. 

Therefore we are done if we can show that any tuple of hyperplanes adjacent to the flag $\psis(\Pin B_1  \dots \Pin B_p v_0)$ contains \emph{more} than one element from the set in \eqref{ghs}. By supposition, we may pick $i_1,i_2\in I_1$ such that $i_1\in J_{j_1}$ and $i_2\in J_{j_2}$, with $j_1\neq j_2$ and $J_{j_1}, J_{j_2} \not\subset I_1$. 
Now, by choice of the order in which we apply the rules, (\ref{r1})--(\ref{r3}), of Theorem \ref{p1}, we may suppose that the flag $\psis(\Pin B_1  \dots \Pin B_p v_0)$ contains the edge 
\be  \kvp L_{J_1} \cap \dots \cap \kvp L_{J_p} \nn\ee
for some map $\bm k': \{1,\dots,m\} \to \ZT$. (Informally,  we may choose to ``build all the swimming islands first, before fixing any islands'', cf. \S\ref{sec: flagrels}.) 
Any tuple of hyperplanes adjacent $\psis(\Pin B_1  \dots \Pin B_p v_0)$ therefore contains some hyperplane linking $J_{j_1} \cap I_1$ to $J_{j_1} \setminus I_1$ (to construct $\kvp L_{J_{j_1}}$) and also some hyperplane linking $J_{j_2} \cap I_1$ to $J_{j_2} \cap I_1$ (to construct $\kvp L_{J_{j_2}}$). Both of these two distinct hyperplanes belong to $ \{ \Ht ijk \}_{i\in I_1,j\in I\setminus I_1, k\in \ZT}$. So we are done. 
\end{proof}

Suppose therefore that the condition \eqref{nestcon} does hold. By the freedom to permute the commutators, and if necessary by renaming $A\leftrightarrow B$, $I\leftrightarrow J$, we may suppose that $I_1 \subseteq J_1$.
By definition of $\Se_0$, 
\be \Se_0(\Pin A_1\Pin A_2 \dots \Pin A_n v_0, \Pin B_1 \Pin B_2 \dots \Pin B_p v_0 )
 =  \Se_0(\Pin A_2 \dots \Pin A_n v_0, [ \cai(\Pin A_1), \Pin B_1 ]'\,\Pin B_2 \dots  \Pin B_p v_0 ). \nn\ee
Therefore to complete the inductive step it is enough to establish the following.

\begin{lem}
\be \Geom(\Pin A_1 \Pin A_2 \dots \Pin{A_n} v_0, \Pin B_1 \Pin B_2 \dots \Pin B_p v_0 )
 =  \Geom(\Pin A_2  \dots \Pin A_n v_0, [ \cai(\Pin A_1), \Pin B_1 ]' \,\Pin B_2 \dots  \Pin B_p v_0 ). \nn\ee
\end{lem}
\begin{proof} There are two cases, $I_1\subsetneq J_1$ and $I_1 = J_1$. 

First suppose $I_1\subsetneq J_1$. Then the lemma follows from Proposition \ref{comprop} and the definition of $\Geom$. (Compare the proof of Proposition \ref{Gmprop}.)

Next suppose $I_1= J_1$. By definition of $\Geom$, we have
\begin{multline}  \Geom(\Pin A_1 \Pin A_2 \dots \Pin{A_n} v_0, \Pin B_1\Pin B_2 \dots \Pin B_p v_0 ) \\
 = \Geom(\Pin A_2  \dots \Pin A_n v_0, \Pin B_2 \dots  \Pin B_p v_0 ) \\
 {}\times \left(\Geom(\Pin A_1, \Pin B_1)  \sum_{i\in I_1=J_1} \left( a(\Ho i) +  \sum_{j\in J_{2} \sqcup \dots \sqcup J_p} \sum_{k\in \ZT} a(\Ht ijk)\right) + \Gm( \Pin A_1, \Pin B_1)\right).
\nn\end{multline}
On the other hand we have, using Corollary \ref{comcor} and $\Geom(\Pin A_1, \Pin B_1) = \Sa(\Pin A_1, \Pin B_1)$, 
\begin{multline} [ \cai(\Pin A_1), \Pin B_1 ]' \,\Pin B_{2} \dots  \Pin B_p v_0 
 = \left( \Geom(\Pin A_1, \Pin B_1 ) \sum_{i\in I_1} \frac 1 T \Pin h_i + \Cou(\cai(\Pin A_1), \Pin B_1) \K\right)
\Pin B_{2} \dots  \Pin B_p v_0  \\
 = \Pin B_{2} \dots  \Pin B_p v_0 \left( \Geom(\Pin A_1, \Pin B_1 ) \sum_{i\in I_1=J_1} \left( a(\Ho i) +  \sum_{j\in J_{2} \sqcup \dots \sqcup J_p} \sum_{k\in \ZT} a(\Ht ijk)\right) + \Gm( \Pin A_1, \Pin B_1)\right) 
\nn\end{multline}
The lemma follows. 
\end{proof}
This completes the proof of Theorem \ref{ScGthm}
\end{proof}

Putting together Theorem \ref{ScGthm} and Theorem \ref{SGthm}, we finally obtain the following. Let $\Sa^p$ be the bilinear form on the space $\Wedge p \nbs \ox \Mb$, $p = 0,1,\dots,m$, defined by 
\begin{multline} \Sa^p(\Pin a_p \wx \dots \wx \Pin a_1 \ox x_0 \ox x_1 \ox \dots \ox x_N,
\Pin b_p \wx \dots \wx \Pin b_1 \ox y_0 \ox y_1 \ox \dots \ox y_N)\\
= \det(\Sa_{\nb}(\Pin a_i,\Pin b_j)_{1\leq i,j\leq p}) \Se_0(x_0,y_0) \prod_{i=1}^N \Sa_i(x_i,y_i). \label{Spedef}\end{multline}
Let 
\be \Shom^p: \Wedge p \nbs \ox \Mb \to (\Wedge p \nbs \ox \Mb)^* \label{shompresym}\ee 
denote the corresponding linear map.

Recall the isomorphism $\psis_p: C_p(\nbs,\Mb)_\mones \to \Fl^{m-p}(\cca)$ of Theorem \ref{p1}. It yields an isomorphism of the dual spaces, 
\be (\psis_p^{-1})^* : C_p(\nbs,\Mb)_\mones^* \to  \Fl^{m-p}(\cca)^* \cong \A^{m-p}(\cca).  \nn\ee

\begin{thm}\label{zSGthm} On $C_{p}(\nbs, \Mb)_\mones = \left(\Wedge {p} \nbs \ox \Mb\right)_{\mones}$, we have the equality 
\be\Sa^{p}=  (-1)^{m-p} T^{p} \psis^*_{p}(\Geom^{m-p})\nn\ee 
of symmetric bilinear forms.
Equivalently, the following diagram commutes:
\be\begin{tikzpicture}    
\matrix (m) [matrix of math nodes, row sep=4em,    
column sep=6em, text height=1ex, text depth=1ex]    
{     
C_p(\nbs, \Mb)_{\mones} &  C_p(\nbs,\Mb)_{\mones}^*, \\    
\Fl^{m-p}(\cca)  & \A^{m-p}(\cca).  \\    
};    
\path[->,font=\scriptsize,shorten <= 2mm,shorten >= 2mm]    
(m-1-1) edge node [above] {$\Shom^p$} (m-1-2)    
(m-2-1) edge node [above] {$T^p (-\Ghom)^{m-p}$} (m-2-2);    
\path[->,shorten <= 0mm,shorten >= 2mm]    
(m-1-1) edge node [right] {$\sim$} node [left] {$\psis_p$} (m-2-1)    
(m-1-2) edge node [right] {$\sim$} node [left] {$(\psis_p^{-1})^*$} (m-2-2);    
\end{tikzpicture}\nn\ee    
\qed\end{thm}

\section{Symmetrization}\label{symsec}
\subsection{The Lie algebras $\no$ and $\ns$}
Let $\no$ denote the free Lie algebra in generators $F_i$, $i\in \I$, $\ru\in \Z_{\geq 1}$. 
For a permutation $\siod: \I \to \I$, 
define the automorphism $\sio:\no\to \no$ by
\be \sio(F_i) = F_{\siod(i)}
 , \qquad i\in \I.\ee
Let $T\in \Z_{\geq 1}$ be the order of $\sio$. 

The set $\I$ is the union of some number $\rf$ of disjoint cycles for the permutation $\siod$. 
Pick an injection $\iota: \Is \to \I$ which maps distinct elements into distinct cycles, so that $\iota(i)$ is a representative of the $i$th cycle. Let $T_i$ denote the length of the $i$th cycle, \ie $T_i := |\siod^\Z \iota(i) |$. Each $T_i$ divides $T$. 

By definition $\sum_{i=1}^{\rf} T_i = \ru$. Observe that 
\be \{ \sio^k F_{\iota(i)} : k=0,\dots, T_i-1,  i \in \Is \}\label{newgens}\ee
are the set of generators of $\no$. 



As with $\nb$ above, we have a weight decomposition of $\no$, namely the $\Z_{\geq 0}^{\rf}$-grading of $\no$ as a Lie algebra, 
\begin{subequations}\label{wd}
\be \no = \bigoplus_{(n_1,\dots,n_{\rf})\in \Z_{\geq 0}^{\rf}} \no_{[n_1,\dots,n_{\rf}]},\ee
such that 
\be\no_{[0,\dots,0,\at i 1,0,\dots,0]} := \Span_\C(\sio^k F_{\iota(i)})_{k=0,1,\dots,T_i-1},\quad\text{for each } i\in \Is.\ee
\end{subequations} 
We have also the obvious refinement of this decomposition, namely the $\Z_{\geq 0}^{\ru}$-grading of $\no$ as a Lie algebra
\begin{subequations}\label{wdu}
\be \no = \bigoplus_{(n_1,\dots,n_{\ru})\in \Z_{\geq 0}^{\ru}} \no_{(n_1,\dots,n_{\ru})},\ee
such that 
\be\no_{(0,\dots,0,\at i 1,0,\dots,0)} := \Span_\C(F_{i}),,\quad\text{for each } i\in \I.\ee
\end{subequations} 

\begin{exmp}\label{ex1} Suppose $\ru = 3$ 
and $\siod$ is the permutation $(13)$. 
We have $\rf = 2$ and, say, $\iota(1)=1$, $\iota(2) = 2$. Then $T_1 = 2$, $T_2=1$, and 
\be F_1 = \sio^0 F_{\iota(1)}, \quad F_2 = \sio^0 F_{\iota(2)}, \quad F_3 = \sio^1 F_{\iota(1)}. \nn\ee
Thus $\n_{[1,0]} = \n_{(1,0,0)} \oplus \n_{(0,0,1)} = \Span_\C(F_1,F_3)$ and  $\n_{[0,1]} = \n_{(0,1,0)} = \Span_\C(F_2)$.
\end{exmp}

Let $\nos\subset \no$ denote the subalgebra of $\sio$-invariants. Define
\be \Mo := U(\nos) \ox \bigotimes_{i=1}^N U(\no) ,\label{Modef}\ee
cf. \eqref{Mbdef}. 
We have the standard chain complex $(C_\bl(\nos,\Mo),d)$, cf. \eqref{ddef}. The automorphism $\sigma$ respects the weight decomposition \eqref{wd} (though not the decomposition \eqref{wdu} in general). Thus we have also a weight decomposition of $\nos$, and thence of $\Mo$, $C_p(\nos,\Mo)$, etc. 

\subsection{The weight $\lambda$ and symmetrization in $C(\nbs,\Mb)$} 
Now pick and fix an arbitrary weight $\lambda:= (\lambda_1,\dots,\lambda_{\rf})\in \Z_{\geq 0}^{\rf}$. Set $m:= |\lambda| := \sum_{i=1}^{\rf} \lambda_i$. We have the free Lie algebra $\nb$ in generators $\f ki$, $k\in \ZT$, $i=1,\dots, m$, as in \S\ref{fla}.

(Let us stress that $m$ and $\nb$ now depend on the choice of $\lambda$.)

It will be useful to introduce a notation for the elements of the set $\{1,\dots,m\}$, adapted to the composition $m=\lambda_1+\dots + \lambda_\rf$ of $m$. To that end, we shall write
\be \clc i n := \sum_{j=1}^{i-1} \lambda_j +n, \qquad n=1,\dots,\lambda_i, \quad i=1,\dots,\rf, \label{innot}\ee
so that
\be \{1,\dots,m\} = \{\clc 1 1,\dots,\clc 1{\lambda_1},\clc 2 1,\dots,\clc 2{\lambda_2},\dots,\clc\rf 1,\dots,\clc\rf{\lambda_\rf)} \}.\nn\ee 

For $i=1,\dots,\rf$, let $\Sw_{\lambda_i}$ denote the wreath product, $\Sigma_{\lambda_i} \wr \Z_{T/T_i}$. That is, let 
\be \Sw_{\lambda_i} :=  \Sigma_{\lambda_i} \ltimes \left(\Z_{T/T_i}\right)^{\lambda_i} ,\nn\ee
where the semi-direct product is
\be ( s', (k'_n)_{1\leq n\leq \lambda_i}) ( s, (k_n)_{1\leq n\leq \lambda_i} ) = (ss', ( k_n + k'_{s(n)})_{1\leq n\leq \lambda_i}).\nn\ee    
We let an element $(s,(k_n)_{1\leq n\leq \lambda_i})\in \Sw_{\lambda_i}$ act on the generators $\f k {\clc i p}$, $k\in \ZT$, $p\in \{1,\dots,\lambda_i\}$ by
\be (s,(k_n)_{1\leq n\leq \lambda_i}) \on \f k {\clc i p} := \f{k+T_ik_p}{\clc i{s(p)}} ,\nn\ee
and let it act trivially on the remaining generators, $\f k {\clc j p}$, $j\neq i$, $k\in \ZT$, $p\in \{1,\dots,\lambda_j\}$.
This defines an action by automorphisms of the group
\be \Sw_\lambda := \Sw_{\lambda_1} \times \dots \times \Sw_{\lambda_{\rf}}. \nn\ee
on the free Lie algebra $\nb$, and hence on $U(\nb)$. The subalgebra $\nbs$ is stable under this action. In this way we have an action of $\Sw_\lambda$ on $\Wedge \bl \nbs\ox \Mb$. 

This action of $\Sw_\lambda$ on $\Wedge \bl \nbs\ox \Mb$ commutes with the projection  $\Wedge \bl \nbs\ox\Mb \to (\Wedge \bl \nbs \ox \Mb)_{\mones}$ to the weight $(1,\dots,1) \in \Z_{\geq 0}^m$ subspace. 
  
Recall the generators of $\no$ given in \eqref{newgens}.
\begin{prop}\label{outhomprop}
There is a homomorphism of Lie algebras  $\jmath:\no \to \nb$ defined by
\be \jmath:\sio^k F_{\iota(i)} \mapsto \sum_{l=0}^{T/T_i-1}  \sum_{n=1}^{\lambda_i} \f {k+T_il} {\clc i n} 
, \qquad k= 0,\dots,T_i-1,\quad  i\in \Is.\nn\ee
It satisfies $\jmath \circ \sio = \sib \circ \jmath$. Hence its restiction to $\nos$ defines a homomorphism $\nos\to \nbs$.
\end{prop}
\begin{proof} For $k = 0,\dots,T_i-2$ it is immediate that
\be (\jmath\circ\sio) (\sio^{k} F_{\iota(i)}) =  \sum_{l=0}^{T/T_i-1}  \sum_{n=1}^{\lambda_i} \f {k+1+T_il} {\clc i n} 
= \sib\sum_{l=0}^{T/T_i-1}  \sum_{n=1}^{\lambda_i} \f {k+T_il} {\clc i n} 
= (\sib\circ \jmath)(\sio^k F_{\iota(i)}).\nn\ee
It remains to observe that
\begin{multline} (\jmath\circ\sio)(\sio^{T_i-1} F_{\iota(i)})=  \jmath(\sigma^{T_i} F_{\iota(i)}) = \jmath(F_{\iota(i)}) 
             =  \sum_{l=0}^{T/T_i-1}  \sum_{n=1}^{\lambda_i} \f {T_il} {\clc i n} 
\\
            =  \sum_{l=0}^{T/T_i-1}  \sum_{n=1}^{\lambda_i} \f {T_i(l-1)} {\clc i n} 
             =  \sib \sum_{l=0}^{T/T_i-1}  \sum_{n=1}^{\lambda_i} \f {T_i-1+T_il} {\clc i n} 
            = (\sib \circ \jmath)( \sio^{T_i-1} F_{\iota(i)})\nn.
\end{multline}
\end{proof}
The homomorphism of Proposition \eqref{outhomprop} induces a homomorphism of complexes  $C_\bl(\nos,\Mo) \to C_\bl(\nbs, \Mb)^\Swl$. By composing this homomorphism with projection onto the $\mones$ weight space, we get a homomorphism of complexes
\be\sym: C_{\bl}(\nos,\Mo)_{\lambda} \to C_\bl(\nbs, \Mb)_{\mones}^\Swl .\label{shomdef}\ee

\begin{exmp}\label{ex2} We continue Example \ref{ex1} above. 
For simplicity, take $N=1$ and consider elements in $1\ox U(\no) \subset \Wedge 0 \nos \ox U(\nos) \ox U(\no) = C_0(\nos,\Mo)$. Under the map $\sym$, we have
\be 1 \ox F_1 \mapsto 1 \ox \f 0 {\clc 1 1}, \qquad 1 \ox F_2 \mapsto 1 \ox (\f 0 {\clc 2 1} +\f 1 {\clc 2 1}), \qquad 1\ox F_3 \mapsto 1 \ox \f 1 {\clc 1 1} \nn \ee
and
\begin{align} 1 \ox F_1 F_1 &\mapsto 1 \ox (\f 0 {\clc 1 1} \f 0{\clc 1 2} +\f 0 {\clc 1 2} \f 0{\clc 1 1}), \nn\\ 
    1 \ox F_1 F_3 &\mapsto 1 \ox (\f 0 {\clc 1 1} \f 1{\clc 1 2} + \f 0{\clc 1 2} \f1{\clc 1 1})\nn\\
 1 \ox F_1 F_2 &\mapsto 1 \ox (\f 0 {\clc 1 1} \f 0 {\clc 2 1} + \f0{\clc 1 1} \f 1 {\clc 2 1}) \nn\\
 1 \ox F_2 F_2 &\mapsto 1 \ox (\f 0 {\clc 2 1} \f 0 {\clc 2 2} + \f0{\clc 2 2} \f 0 {\clc 2 1}
                              +\f 0 {\clc 2 1} \f 1 {\clc 2 2} + \f1{\clc 2 2} \f 0 {\clc 2 1}\nn\\
                & \qquad\,\,              +\f 1 {\clc 2 1} \f 0 {\clc 2 2} + \f0{\clc 2 2} \f 1 {\clc 2 1}  
                              +\f 1 {\clc 2 1} \f 1 {\clc 2 2} + \f1{\clc 2 2} \f 1 {\clc 2 1} ). \nn
\end{align}
\end{exmp}

In the other direction, there is a homomorphism of Lie algebras $\nb \to \no$ defined by
\be \f k {\clc i n} \mapsto \sio^k F_{\iota(i)}, \qquad k\in \ZT, \quad i=1,\dots,\rf,\quad n=1,\dots,\lambda_i.\ee
This homomorphism induces a homomorphism $C_\bl(\nbs,\Mb)_{\mones} \to C_\bl(\nos,\Mo)_\lambda$. Let $\pi$ denote the restriction of the latter to $C_\bl(\nbs,\Mb)_{\mones}^\Swl$,
\be \pi : C_\bl(\nbs,\Mb)_{\mones}^\Swl \to C_\bl(\nos,\Mo)_\lambda.\label{pihomdef}\ee
\begin{lem}\label{isolem} The compositions $\pi \circ \sym$ and $\sym\circ\pi$ are equal to multiplication by \be |\Sw_\lambda| =  \prod_{i=1}^{\rf} (T/T_i)^{\lambda_i}\lambda_i!.\ee
\qed
\end{lem}
\begin{cor}\label{isocor} The maps $\sym$ and $\pi$ are isomorphisms of complexes. \qed
\end{cor}

\subsection{Symmetrization of flags and forms}\label{symflags}

Define an action of $\Sw_{\lambda_i}$ on the coordinate functions of $\C^m$ by the following formula.
For $i=1,\dots,\rf$, let an element $ (s, (k_n)_{1\leq n\leq \lambda_i}) \in \Sw_{\lambda_i}$ transform the coordinate functions $t_{\clc i p}$, $p\in\{1,\dots,\lambda_i\}$, on $\C^m$ according to
\be (s, (k_n)_{1\leq n\leq \lambda_i}) \on t_{\clc i p} := \omega^{T_i k_p}t_{\clc i {s(p)}} \label{gont}\ee
and let it act trivially on the remaining coordinate functions $t_{\clc j n}$, $j\neq i$, $n\in \{1,\dots,\lambda_j\}$.
This action defines an action of $\Sw_{\lambda_i}$ on $\C^m$. 

The action \eqref{gont}   gives rise to an action of the group $\Swl$ on the set of hyperplanes of the arrangement $\cca$. 
This yields a canonical action of $\Swl$ on the spaces $\A^p(\cca)$ and $\F^p(\cca)$.

Namely, for each $i=1,\dots,\rf$, an element $g = (s, \vec k) \in \Sw_{\lambda_i}$ acts on $\A^p(\cca)$ according to 
\be g \on (H_1\wx\dots \wx H_p) = (g\on H_1)\wx \dots\wx (g\on H_p).\nn\ee
For hyperplanes $H_1,\dots,H_p$, let $F(H_1,\dots,H_p)$ denote the flag $H_1\supset H_1\cap H_2 \supset \dots \supset H_1\cap \dots \cap H_p$. For each $i=1,\dots,\rf$, we let an element $g = (s, \vec k) \in \Sw_{\lambda_i}$ act on the set of flags as follows 
\be g \on F(H_1,\dots,H_p) = F((g\on H_1),\dots, (g\on H_p)).\nn\ee
These actions respect the pairing of \S\ref{pf}, \ie $\bra g\on F,g \on (H_1\wx\dots \wx H_p) \ket = \bra F, (H_1\wx\dots \wx H_p) \ket$. 

For $g = ( (s_1, \vec k_1), \dots, (s_\rf,\vec k_\rf)) \in \Swl$ let us write
\be (-1)^{|g|} := (-1)^{|s_1|} \dots (-1)^{|s_\rf|}.\label{modgdef}\ee

\begin{prop}\label{ep} The isomorphisms $\psis_\bl: \Wedge \bl \nbs\ox\Mb \to\Fl^{m-\bl}(\cca)$ of Theorem \ref{p1} have the following equivariance property with respect to these actions of $\Sw_\lambda$. For all $g\in \Swl$ and $m\in \Wedge \bl \nbs\ox\Mb$, 
\be \psis_\bl( g \on m ) = (-1)^{|g|} g\on \psis_\bl(m)\nn\ee
\qed\end{prop}

Let $\Fl^{m-k}(\cca)^\Swl$ denote the subspace of $\Fl^{m-k}(\cca)$ spanned by flags $F$ such that $(-1)^{|g|} g\on F = F$. 
We arrive at the following.
\begin{thm}\label{symthm}
The map
\be \psis_{k}\circ \sym : C_k(\nos,\Mo)_\lambda \to \Fl^{m-k}(\cca)^\Swl \ee
is an isomorphism of complexes, for $k=0,1,\dots,m$.\qed
\end{thm}
\begin{exmp} In the setting of Example \ref{ex2} above, 
\begin{align} 
(\psis_0\circ\sym)(1\ox F_2) = \psis_0\left(1 \ox (\f 0 {\clc21}  + \f 1{\clc21} )\right) =&
\left(\C^2 \supset (t_{\clc21} = z_1)  \right)\nn\\
&+  \left(\C^2 \supset (-t_{\clc 2 1} = z_1)\right).\nn
\end{align}
and
\begin{align} 
(\psis_0\circ\sym)(1\ox F_1F_1) = \psis_0\left(1 \ox (\f 0 {\clc 1 1} \f 0{\clc 1 2} + \f 0{\clc 1 2} \f0{\clc 1 1})\right) =&
\left(\C^2 \supset (t_{\clc 1 2} = z_1) \supset (t_{\clc 1 1}=t_{\clc 1 2} = z_1) \right)\nn\\
&-  \left(\C^2 \supset (t_{\clc 1 1} = z_1) \supset (t_{\clc 1 2}=t_{\clc 1 1} = z_1) \right),\nn\\
(\psis_0\circ\sym)(1\ox F_1F_3) = \psis_0\left(1 \ox (\f 0 {\clc 1 1} \f 1{\clc 1 2} + \f 0{\clc 1 2} \f1{\clc 1 1})\right) =&
\left(\C^2 \supset (-t_{\clc 1 2} = z_1) \supset (t_{\clc 1 1}=-t_{\clc 1 2} = z_1) \right)\nn\\
&-  \left(\C^2 \supset (-t_{\clc 1 1} = z_1) \supset (t_{\clc 1 2}=-t_{\clc 1 1} = z_1) \right).\nn
\end{align}
\end{exmp}

\section{Cartan data and diagram automorphism}\label{sec: cd}
\subsection{Definitions of $\g$ and $\gs$}
Suppose we are given the following data:
\begin{enumerate}[(1)]
\item A finite-dimensional complex vector space $\h$;
\item A non-degenerate symmetric bilinear form $(\cdot,\cdot): \h\times \h\to \C$;
\item A collection $\alpha_1,\dots,\alpha_\ru\in \h^*$ of linearly independent elements, called the \emph{simple roots}.
\end{enumerate}
Let $\bra\cdot,\cdot\ket: \h\times \h^* \to \C$ denote the canonical pairing, $\bra H,\lambda\ket := \lambda(H)$. 

The form $(\cdot,\cdot)$ induces an isomorphism $\h^*\to \h$ and a symmetric bilinear form on $\h^*$ which we also denote by $(\cdot,\cdot)$. Define $H_i\in \h$, $i\in \I$, to be the elements such that $\bra H_i,\cdot\ket = (\alpha_i,\cdot)$. 
Let $B = (b_{ij})_{i,j\in \I}$ be the symmetric matrix with elements $b_{ij} := (\alpha_i,\alpha_j) = \bra H_i,\alpha_j\ket$.

Let $\g= \g(B)$ be the Lie algebra with generators $E_i$, $F_i$, $i\in \I$, and $H\in \h$, subject to the defining relations
\be [E_i,F_j] = \delta_{ij} H_j,\quad [H,E_i] = \bra H,\alpha_i\ket E_i, \quad [H,F_i]  = -\bra H,\alpha_i\ket F_i, \quad [H,H'] = 0, \nn\ee
for all $i,j\in \I$ and all $H,H'\in \h$. 
(Thus, $\g$ is a ``Kac-Moody Lie algebra without Serre relations'', and $B$ is the symmetrized Cartan matrix of $\g$; see Remark \ref{KMrem} below.)

We have the triangular decomposition
\be \g = \n \oplus \h \oplus \n^+ \nn\ee
of $\g$ where $\n$ (resp. $\n^+$) is the subalgebra generated by the $F_i$ (resp. $E_i$), $i\in \I$. The subalgebras $\n$ and $\n_+$ are free Lie algebras. 

Now suppose that in addition to (1--3) above we are given
\begin{enumerate}[(4)]
\item An automorphism of $\sio : \g \to \g$ such that
\be 
\sio(F_i) =  F_{\siod(i)}, \quad 
\sio(E_i) =  E_{\siod(i)}, \quad 
\sio(H_i) =  H_{\siod(i)}  \qquad i\in \I,\label{sigdef}\ee
for some permutation $\siod \in \Sigma_\ru$.
\end{enumerate}
The fact that $\sio$ is an automorphism implies that the permutation $\siod$ must be such that 
\be (\alpha_i,\alpha_j) = (\alpha_{\siod(i)},\alpha_{\siod(j)}),\label{ssym}\ee
for all $i,j\in \I$. 
Such automorphisms are called \emph{diagram automorphisms}.

\begin{rem} 
If $B$ is a symmetrized Cartan matrix of finite type then $\{H_i\}_{i\in \I}$ span $\h$ and the conditions \eqref{sigdef} completely define $\sigma$. More generally they define $\sigma$ only on the derived subalgebra $[\g,\g]$ of $\g$. See \cite{FSS}.
\end{rem}

Let $T$ be the order of $\sigma$. Recall from \S\ref{symsec} the definitions of the numbers $T_i$, $i\in \I$, and the injection $\iota:\Is \to \I$. 

\subsection{Symmetrized commutator lemmas}\label{clsec}
We have the weight decomposition \eqref{wd} of $\n$, and similarly of $\n_+$:
\begin{subequations}\label{wdp}
\be \no_+ = \bigoplus_{(n_1,\dots,n_{\rf})\in \Z_{\geq 0}^{\rf}} (\no_+)_{[n_1,\dots,n_{\rf}]},\ee
with 
\be( \no_+)_{[0,\dots,0,\at i 1,0,\dots,0]} := \Span_\C(\sio^k E_{\iota(i)})_{k=0,1,\dots,T_i-1},\quad\text{for each } i\in \Is.\ee
\end{subequations} 

Let $\lambda = (\lambda_1,\dots,\lambda_r) \in \Z_{\geq 0}^\rf$ be a nonzero weight and let $m:= \lambda_1+\dots+ \lambda_\rf$. Define
\begin{subequations}\label{syml}
\be \sym: \n_\lambda \to \nb^\Swl_\mones\ee
to be the composition of the homomorphism $\jmath$ of Proposition \ref{outhomprop} with projection onto to the weight $\mones$ subspace. It is a linear isomorphism, cf. Lemma \ref{isolem}. 
Similarly (by replacing $f\to e$, $F\to E$ everywhere) we have a linear isomorphism
\be \sym: (\n_+)_\lambda \to (\nb_+)^\Swl_\mones.\ee
\end{subequations}

These maps restrict to linear isomorphisms $(\nos)_\lambda \to \nbs_{\mones}^\Swl$ and $(\nos_+)_{\lambda} \to  (\nbs_+)_\mones^\Swl$ respectively. 

Suppose $\mu= (\mu_1,\dots,\mu_r)\in \Z_{\geq 0}^\rf$, $\mu \neq \lambda$, is a nonzero weight dominated by $\lambda$. That is, $\mu_i\leq \lambda_i$ for each $i=1,\dots,\rf$. 
For each $i=1,\dots,r$, pick any subset $I_i \subset \{1,\dots, \lambda_i\}$ of size $\mu_i$. Having made this choice, define a subset $I\subset \{1,\dots,m\}$ by, in the notation of \eqref{innot}, 
\be I = \{ \clc i t : 1\leq i\leq r,\, t\in I_i\}. \label{Ildef}\ee
We have the corresponding weight $\wtI I$, in the notation introduced before \eqref{genpsis}. Let $\sym_I$ denote the composition of the homomorphism $\jmath$ of Proposition \ref{outhomprop} with projection onto the weight $\wtI I$ subspace. We get a linear isomorphism
\be \sym_I: \no_\mu \to \nb_{\wtI I}^{\Sw_\mu}.\nn\ee
Similarly we have a linear isomorphism
\be \sym_I: (\no_+)_\mu \to (\nb_+)_{\wtI I}^{\Sw_\mu}.\nn\ee

Let $\mc I$ be the set of \emph{all} subsets $I\subset \{1,\dots,m\}$ of the form \eqref{Ildef}, \ie corresponding to all possible choices of subsets $I_i\subset \{1,\dots, \lambda_i\}$, $i=1,\dots,\rf$. Note that 
\be |\mc I| = \prod_{i=1}^\rf\binom{\lambda_i}{\mu_i} = \prod_{i=1}^\rf \frac{\lambda_i!}{\mu_i!(\lambda_i-\mu_i)!} = \frac{|\Sigma_{\lambda}|}{|\Sigma_{\mu}||\Sigma_{\lambda-\mu}|}
 = \frac{|\Sw_{\lambda}|}{|\Sw_{\mu}||\Sw_{\lambda-\mu}|}, \label{isize}\ee
since $|\Sw_\lambda| =  \prod_{i=1}^{\rf} (T/T_i)^{\lambda_i}\lambda_i!$.  
\begin{lem}\label{symcomlem2}
Suppose $a\in \no_{\mu}$ and $b\in \no_{\lambda-\mu}$. Then
\be \sym([a,b]) = \sum_{I\in \mc I}\left[\sym_I(a), \sym_{\Ip}(b)\right] \nn\ee
as an equality in $\nb_{\mones}^\Swl$.
\qed
\end{lem}

Recall that $\cco m$ denotes the arrangement in $\C^m$ whose hyperplanes are $\Ht ijk : t_i = \omega^k t_j$ with $i,j\in \{1,\dots,m\}$, $k\in \ZT$. 
The Cartan data defines a weighting of this arrangement, in the sense of \S\ref{sg}, given by
\be    a\big(\Ht {\clc i n}{\clc{i'}{n'}} k\big) :=  (\si^k\alpha_{\iota(i)},\alpha_{\iota(i')}),\label{HdC}\ee
for $i,i'\in \Is$, and $n\in \{1,\dots,\lambda_i\}$, $n'\in \{1,\dots,\lambda_{i'}\}$. 

Note that $[ \hh {k+T_il}{\clc i n}, \f p {\clc j r} ] 
  = - ( \si^{p-k-T_il} \alpha_i,\alpha_j)\,\, \f p {\clc j r} 
  = - ( \si^{p-k} \alpha_i,\alpha_j)\,\, \f p {\clc j r}$ whenever $\clc i n\neq \clc jr$.
Thus, whenever $\clc i n \notin I \in \mc I$ and $d\in \n_\mu$, we have
\be \sum_{l=0}^{T/T_i-1} [\hh {k+T_il}{\clc i n} , \sym_I(d)] 
= \sym_I([\sigma^k H_{\iota(i)}, d]) \frac{T}{T_i}, \label{hsym1}\ee
and hence
\be [\Pin h_{\clc i n} , \sym_I(d)] 
= \sym_I([\Pin H_{\iota(i)}, d]) . \label{hsym2}\ee

\begin{lem}\label{symcomlem}
Suppose $a\in (\no_+)_\mu$ and $d\in \no_\lambda$. Then
\be\big[\sym_I(a), \sym(d)\big] =  \sym_{\Ip}([a,d]) \,|\Sw_\mu|  \nn\ee
as an equality in $\nb_{\wtI{\Ip }}^{\Sw_{\lambda-\mu}}$.
\end{lem}
\begin{proof}
By induction on $|\mu| := \mu_1+\dots+\mu_\rf$. 

For the base case, we may suppose that $a= \si^k E_{\iota(i)}$, $i\in \{1,\dots,\rf\}$, $k\in \{1,\dots, T_i\}$, and we have $I=\{\clc i n\}$ for some $n\in \{1,\dots,\lambda_i\}$. Thus $\sym_{I}(a) = \sum_{l=0}^{T/T_i-1} \e {k+T_il}{\clc i n}$, and what must be shown is that
\be 
\sum_{l=0}^{T/T_i-1} [\e {k+T_il}{\clc i n} , \sym(d)]
= \sym_{\{1,\dots,m\}\setminus\{\clc i n\}}([\si^k E_{\iota(i)}, d]) \frac{T}{T_i} .\nn\ee
This is true by inspection, using \eqref{hsym1} and the fact that
\be \left[\sum_{l=0}^{T/T_i-1} \e {k+T_il}{\clc i n}, \sum_{l'=0}^{T/T_i-1} \f {k'+T_il'}{\clc i {n'}} \right] 
= \begin{cases} 
\sum_{l=0}^{T/T_i-1} \hh {k+T_il}{\clc i n} & n=n' \text{ and } k \equiv k' \!\!\!\mod T_i \\
0 & \text{otherwise}. \end{cases}
\nn\ee

For the inductive step, it is enough to establish the result when $a$ is a commutator, $a=[b,c]$, since these span $(\n_+)_{\mu}$. Let $\rho = (\rho_1,\dots,\rho_\rf) \in \Z_{\geq 0}^\rf$ be the weight of $b$. For each $I\in \mc I$ as in \eqref{Ildef}, Lemma \ref{symcomlem2} yields
\be \sym_I([b,c]) = \sum_{J\in \mc J_I} \left[ \sym_J(b), \sym_{I\setminus J}(c)] \right] \nn\ee
as an equality in $\nb_{\wtI I}^{\Sw_\mu}$, where $\mc J_I$ is defined by analogy with $\mc I$. Namely, let $\mc J_I$ be the set of subsets $J\subset I$ such that $J = \{\clc i t: 1\leq i \leq \rf, t\in J_i\}$ for some subsets $J_i\subset I_i$ with $|J_i| = \rho_i$. 
Hence, by the Jacobi identity and the inductive assumption, we have
\begin{align} [\sym_I([b,c]), \sym(d) ]  
&=  \sum_{J\in \mc J_I} \left[ \left[ \sym_J(b), \sym_{I\setminus J}(c)] \right], \sym(d)\right] \nn\\
&=  \sum_{J\in \mc J_I} \left( 
\left[ \sym_J(b) ,\left[ \sym_{I\setminus J}(c), \sym(d)\right] \right] 
- \left[ \sym_{I\setminus J}(c), \left[\sym_J(b) ,  \sym(d)\right] \right]\right) 
\nn\\
&=  \sum_{J\in \mc J_I} \left( 
\left[ \sym_J(b) ,\sym_{(\Ip) \cup J} \left(\left[ c, d\right]\right) \right] |\Sw_{\mu-\rho}|
- \left[ \sym_{I\setminus J}(c), \sym_\Jp\left(\left[ b,  d \right]\right) \right] |\Sw_{\rho}|
\right) 
\nn\\
&=  \sum_{J\in \mc J_I} \left( 
\sym_{\Ip} \left( \left[b,\left[ c, d\right]\right] 
- \left[ c , \left[ b,  d \right] \right]\right) 
\right) |\Sw_{\rho}| |\Sw_{\mu-\rho}|
\nn\\
&= \sym_{\Ip}\left(\left[ \left[ b,c\right] , d\right] \right)
 |\mc J_I| |\Sw_{\rho}| |\Sw_{\mu-\rho}|. 
\nn
\end{align}
Now, $|\mc J_I| = \frac{|\Sigma_\mu|}{|\Sigma_\rho| |\Sigma_{\mu-\rho}|} =  \frac{|\Sw_\mu|}{|\Sw_\rho| |\Sw{\mu-\rho}|}$, cf. \eqref{isize}.  
This completes the inductive step. 
\end{proof}

\begin{lem}\label{SSlem}
Let $a,d\in \n_\lambda$. Then $\frac{1}{|\Sw_\lambda|} \Sa_\nb(\sym(a),\sym(d)) = S_{\g}(a,d)$.
\end{lem}
\begin{proof} By induction on $|\lambda|$. (The argument is similar to that in the proof of Lemma \ref{symcomlem}, or of Proposition \ref{SSm0lem} below.)
\end{proof}

\subsection{Definition of the cocycle $\Co$ on $\gs$ and central extension $\gse$}\label{cocyclesec} Define a skew-symmetric bilinear form $\Co:\gs\times \gs \to \C$ as follows. 
Recall the map $\Cou: \gus_{\leq 1}\times \gus_{\leq 1} \to \C$ of \S\ref{coudef}. For any nonzero weight $\lambda= (\lambda_1,\dots,\lambda_\rf)\in \Z_{\geq 0}^\rf$, and any $x\in (\nos)_\lambda$, $y\in (\nos_+)_\lambda$, set
\be -\Co(y,x) = \Co(x,y) := \frac{1}{|\Sw_\lambda|}\Cou(\sym(x),\sym(y)),\nn\ee
where $\sym$ is the symmetrization map in \eqref{syml}.
On all other components in the weight decomposition of $\gs\times \gs$, set $\Co(\cdot,\cdot)$ to zero.

\begin{prop}\label{kercoprop} Let $x\in \gs$. If $x\in \ker \Sg$ then $\Co(x,y)=0$ for all $y\in \gs$.
\end{prop}
\begin{proof} Suppose $x\in \ns_\lambda$ for some $\lambda \in \Z_{\geq 0}^\rf$  (the argument is similar if $x\in \ns_+$ and trivial if $x\in \h$). Then the result follows from Lemma \ref{SSlem} and Proposition \ref{kercouprop}.  
\end{proof}
\begin{thm} $\Co$ is a cocycle on $\gs$. That is, we have the equality
\be \Co(x,[y,z]) + \Co(y,[z,x]) +\Co(z,[x,y]) = 0\nn\ee
for all $x,y,z\in \gs$. 
\end{thm}
\begin{proof}
When $x\in (\nos_+)_{\lambda}$, $y\in (\nos)_{\lambda}$ and $z\in \h^\si$, for some nonzero $\lambda\in \Z_{\geq 0}^\rf$, the third term vanishes and the first two cancel.

Let $\lambda = (\lambda_1,\dots,\lambda_r)\in \Z_{\geq 0}^\rf$ and $\mu= (\mu_1,\dots,\mu_r)\in \Z_{\geq 0}^\rf$ be  nonzero weights such that $\lambda$ dominates $\mu$, as in \S\ref{clsec}.
When $x\in (\nos_+)_{\lambda}$, $y\in (\nos)_{\mu}$ and $z \in (\nos)_{\lambda-\mu}$, 
we have
\begin{align}   \Co( x,[y,z]) |\Sw_\lambda|
&= \Cou( \sym(x), \sym([y,z])) \nn\\
&= \sum_{I\in \mc I} \Cou(\sym(x), [\sym_{I}(y), \sym_{\Ip}(z)])\nn\\
&= \sum_{I\in \mc I} \left( \Cou\left( \left[ \sym_\Ip(z), \sym(x)\right] , \sym_I(y)\right) + 
                         \Cou\left( \left[ \sym(x), \sym_I(y)\right] , \sym_\Ip(z)\right) \right) \nn\\
&= \sum_{I\in \mc I} \bigg( \Cou\left( \sym_I \left(\left[ z, x\right]\right) , \sym_I(y) \right) |\Sw_{\lambda-\mu}|\nn\\ 
&\qquad\quad +   \Cou\left( \sym_{\Ip}\left(\left[ x, y\right]\right) , \sym_\Ip(z)\right) |\Sw_{\mu}| \bigg) \nn\\
&= \sum_{I\in \mc I} \left(  \Co\left(\left[ z, x\right] , y \right) 
+ \Co\left( \left[ x, y\right] , z\right) \right) |\Sw_{\mu}| |\Sw_{\lambda-\mu}|  \nn
\nn\end{align}
Here we used Lemma \ref{symcomlem2}, then Theorem \ref{couthm}, then Lemma \ref{symcomlem} and the definition of $\Co$. 
Since $|\mc I| |\Sw_{\mu}| |\Sw_{\lambda-\mu}| = |\Sw_{\lambda}|$ as in \eqref{isize}, we indeed have the equality $\Co( x,[y,z]) = \Co([z,x],y) +\Co([x,y],z)$.

Up to obvious symmetries, these are the only non-trivial cases. 
\end{proof}

Now let $\gse$ denote extension of $\gs$ by a one-dimensional centre $\C\K$ defined by this cocycle $\Co$. That is, $\gse \cong_\C \gs \oplus \C\K$ as a vector space, with Lie bracket 
\be [ x,y]' :=  [x,y] + \Co(x,y) \K.\ee
(For clarity we will always use $[\cdot,\cdot]'$ to denote the bracket on $\gse$.)


\subsection{The Cartan anti-involution $\cai$}
Let $\cai:\g\to \g$ denote the involutive ($\cai^2=\id$) anti-automorphism ($\cai([x,y]) = - [ \cai(x),\cai(y)]$) of $\g$ defined by
\be \cai(H) = H,\qquad \cai(E_i) = F_i,\qquad \cai(F_i)=E_i \nn\ee
for all $i\in \I$ and all $H\in \h$. 

Since $\si$ is a diagram automorphism, 
\be \cai\circ \si = \si \circ \cai.\label{sicaicom}\ee
Hence, $\cai$ restricts to give an involutive anti-automorphism $\gs \to \gs$. 
By definition of the cocycle $\Co$ (in terms of $\Cou$ and hence $\Gm$) we have that
$\Co(a,b) = - \Co(\cai(a),\cai(b))$ for all $a,b\in \gs$.
Therefore if we define $\cai(\K) = \K$ then $\cai$ is also an involutive anti-automorphism $\gse \to \gse$ of the centrally extended fixed point subalgebra $\gse$. 

\subsection{The bilinear forms $\Kill$ and $\Sg$ on $\g$}
Let $\Kill(\cdot,\cdot): \g\times\g\to\C$ be the bilinear form such that
\begin{enumerate}[(i)]
\item $\Kill$ coincides with $(\cdot,\cdot)$ on $\h$; $\Kill$ is zero on $\n$ and $\n_+$; $\h$ and $\n\oplus \n_+$ are orthogonal with respect to $\Kill$;
\item $\Kill(F_i,E_j) = \Kill(E_i,F_j) = \delta_{ij}$, $i,j\in \I$;
\item $\Kill$ is $\g$-invariant, \ie $\Kill([x,y],z) = \Kill(x,[y,z])$ for all $x,y,z\in\g$.
\end{enumerate}
Such a form exists, is unique, and is symmetric \cite[\S2.2]{KacBook}. 

The form $\Kill$ is $\cai$-invariant, \ie $\Kill(\cai(x),\cai(y)) = \Kill(x,y)$ for all $x,y\in\g$.
Define the bilinear form $\Sg(\cdot,\cdot):\g\times\g\to\C$ by 
\be \Sg(x,y) := -\Kill(\cai(x),y),\qquad x,y\in \g.\nn\ee
It is symmetric, $\cai$-invariant, and obeys
\be \Sg([x,y],z) = -\Sg(y,[\cai(x),z])\nn\ee
for all $x,y,z\in \g$. The subspaces $\h$, $\n$, and $\n_+$ are pairwise orthogonal with respect to $\Sg$. 

We have $\ker \Kill = \ker \Sg \subset \g$.

\begin{rem}\label{KMrem}
Whenever we have $b_{ii}\neq 0$ for each $i\in \I$, we may define a matrix $A=(a_{ij})_{i,j\in \I}$ by $a_{ij} = 2 b_{ij}/b_{ii}$. By construction $a_{ii}=2$ for each $i\in \I$. If in addition $a_{ij}\in \Z_{\leq 0}$ whenever $i\neq j$ then $A$ is a symmetrizable generalized Cartan matrix. In that case  the quotient $\gt := \g/\ker \Kill= \g/\ker \Sg$ is the Kac-Moody Lie algebra associated with $A$, and the kernel $\ker \Sg$ is generated by the Serre elements $\ad F_i^{-a_{ij}+1} E_j$, $i\neq j$.  
\end{rem}

Given \eqref{ssym}, one sees that $\Kill$ is $\sio$-invariant, \ie $\Kill(\sio(x),\sio(y)) = \Kill(x,y)$ for all $x,y\in \g$. Hence, in view of \eqref{sicaicom}, so too is $\Sg$, \ie $\Sg(\si (x), \si (y)) = \Sg(x,y)$ for all $x,y\in \g$. Thus $\Sg$ restricts to give a bilinear form on $\gs$. Also $\sigma(\ker \Sg) = \ker \Sg$.


\begin{lem}\label{isomlem} There are canonical Lie algebra isomorphisms
\be \left( \g \big/ \ker \Sg\right)^\sigma 
\cong \gs\big/ (\ker \Sg)^\sigma
\cong \gs/\ker((\Sg)|_{\gs}). \nn\ee
\end{lem} 
\begin{proof} We shall show that these are linear isomorphisms; it is then clear that the Lie algebra structure coincide.
Let $\Pi_k: \g \to \g$, $k\in \ZT$, denote the projectors $\Pi_k := \frac 1 T\sum_{j\in \ZT} \omega^{-jk} \si^j$ onto the eigenspaces of $\si$.  
We have $\g = \gs \oplus \g^\perp$ where $\g^\perp = \bigoplus_{k\not\equiv 0} \Pi_k \g$.
By $\sio$-invariance of $\Sg$, the subspaces $\gs$ and $\g^\perp$ are orthogonal with respect to $\Sg$.
Therefore if $x\in \gs$ is such that $\Sg(x,y)=0$ for all $y\in \gs$ then in fact $\Sg(x,y) = 0$ for all $y\in \g$. Thus $\ker((\Sg)|_{\gs})\subset  \ker \Sg \cap \gs= (\ker \Sg)^\sigma$. The reverse containment, $\ker \Sg \cap \gs \subset \ker ((\Sg)|_{\gs})$, is obvious. This establishes the second isomorphism.  Consider the first isomorphism. Elements of $(\g/\ker \Sg)^\sigma$ are by definition cosets $x+ \ker \Sg$ such that $\sigma( x+ \ker \Sg) = x+ \ker \Sg$, or equivalently such that $\sigma x - x \in \ker \Sg$. We have $x = x^\sigma + \sum_{k\not\equiv 0}x^{(k)}$ for unique $x^\sigma \in \gs$ and $x^{(k)} \in \Pi_k \g$. Now $\sigma x - x \in \ker \Sg$ implies that $(\omega^k - 1) x^{(k)} \in \Pi_k \ker S$ for each $k\not \equiv 0$. Hence $x^{(k)} \in \ker S$ for each $k\not\equiv 0$. Therefore in fact
$x+ \ker \Sg = x^\sigma + \ker \Sg$, and here the element $x^\sigma$ is defined up to addition of an element of $(\ker \Sg)^\sigma = \ker \Sg \cap \gs$. Therefore we have the first isomorphism. 
\end{proof}

\begin{prop}\label{comprop}
Suppose $a,b\in \n_{(m_1,\dots,m_\ru)}$ for some $(m_1,\dots,m_\ru)\in \Z_{\geq 0}^\ru$, cf. \eqref{wd}. Then
\be [\cai(a), b] = \Sg(a,b) \sum_{i=1}^\ru m_i H_i 
\nn\ee\qed
\end{prop}
\begin{proof} By induction on $|\sum_im_i|$. The statement is true when $|\sum_im_i|=1$. For the inductive step, by linearity and the Jacobi relations, we may suppose that $a = [a',F_i]$, for some $i$ such that $m_i\geq 1$. Then $\Sg( a,b) = \Sg([a',F_i],b) = -\Sg( F_i, [\cai(a'),b])$. On weight grounds $ [\cai(a'),b] \sim F_i$. Since $\Sg(F_i,F_i) = 1$, we must have $- F_i \Sg(a,b) =[\cai(a'), b]$. Hence
\begin{align} - H_i \Sg(a,b) &= [E_i,[\cai(a'), b]]\nn\\ 
&= [ [E_i,\cai(a')], b  ] + [\cai(a'), [E_i,b]] \nn\\
&= [\cai(a), b] + [\cai(a'), [E_i,b]].\label{leqo}\end{align} 
Inductively we have $[\cai(a'), [E_i,b]] = \Sg(a', [E_i,b]) (-H_i + \sum_{j=1}^\ru m_jH_j) = \Sg(a,b) (-H_i + \sum_{j=1}^\ru m_jH_j)$. Inserting this into \eqref{leqo} and rearranging, we find the required equality.
\end{proof} 

\begin{cor}\label{comcorgen}
Suppose $x,y\in \nos_{[\lambda_1,\dots,\lambda_\rf]}$ for some $\lambda = (\lambda_1,\dots,\lambda_\rf)\in \Z_{\geq 0}^\rf$, cf. \eqref{wdu}. Then in $\gs$ we have the equality
\be [\cai(x), y] = \frac 1T S_{\g}(x,y) \sum_{i=1}^\rf \lambda_i \Pin H_{\iota(i)}, 
\nn\ee
and hence in $\gse$ we have the equality
\be [\cai(x), y]' = \frac 1T S_{\g}(x,y) \sum_{i=1}^\rf \lambda_i\Pin H_{\iota(i)} + \Co(x, y) \K .
\nn\ee\qed
\end{cor}
\begin{proof}
We may suppose that $x= \Pin a$, $y= \Pin b$ for some commutators $a,b\in \n$, since such projected commutators span $\ns$. Since $\Pin a$ and $\Pin b$ are of the same $\gs$-weight, 
certainly $b$ has the same length (``total number of $F$'s'') as $a$. Therefore, for every $k\in \ZT$, if $\sigma^k b$ and $a$ have different $\g$-weights then $[\cai(a), \sigma^k b] = 0$. Also $\Sg(a,\sigma^k b)=0$. Hence, if $\sigma^k b$ and $a$ have different $\g$-weights for \emph{all} $k\in \ZT$ then both $[\cai(\Pin a), \Pin b] =0$ and $\Sg(\Pin a, \Pin b)=0$ and we are done.

Suppose therefore that $\sigma^kb$ and $a$ share the same $\g$-weight for some $k\in \ZT$. By choice of $b$ we can take $k=0$ without loss of generality. So suppose $a$ and $b$ are both of $\g$-weight $(m_1,\dots,m_\ru)\in \Z_{\geq 0}^\ru$.
Then, using Proposition \ref{comprop}, we have
\begin{align}
[\cai(\Pin a) , \Pin b] = \sum_{k,p\in \ZT} [\cai(\sigma^k a), \sigma^p b] 
&= \sum_{k,l\in \ZT} \sigma^k[\cai( a), \sigma^l b]\nn\\
&= \sum_{k,l\in \ZT} \Sg(a, \sigma^l b) \,\sigma^k \sum_{i=1}^\ru m_i H_i \nn\\
&= \Sg(a, \Pin b) \sum_{i=1}^\ru m_i \Pin H_i= \frac{1}{T} \Sg(\Pin a, \Pin b) \sum_{i=1}^\ru m_i \Pin H_i\nn.
\end{align}
Now
\be \sum_{i=1}^\ru m_i \Pin H_i 
= \sum_{i=1}^\rf \sum_{k=1}^{T_i} m_{\sigma^k \iota(i)} \Pin H_{\sigma^k \iota(i)} 
= \sum_{i=1}^\rf \left(\sum_{k=1}^{T_i} m_{\sigma^k \iota(i)}\right) \Pin H_{\iota(i)}  
= \sum_{i=1}^\rf \lambda_i \Pin H_{\iota(i)} \nn\ee
and we are done. 
\end{proof}

The above proposition and corollary have their analogs in $\nb$ and $\nbs$, as follows.
\begin{prop}\label{comprop}
Suppose $a$ and $b$ are commutators in $\nb$ both of weight $\wtI I$, for some $I\subset \{1,\dots,m\}$. Then
\be [\cai(a), b] = \Sa(a,b) \sum_{i\in I} h_i 
\nn\ee\qed
\end{prop}
\begin{cor}\label{comcor}
Suppose $x, y\in \nbs_{\wtI I}$, for some $I\subset \{1,\dots,m\}$. Then
\be [\cai(x),  y] = \Sa(x,y) \sum_{i\in I} \frac 1 T\Pin h_i ,
\nn\ee
and hence 
\be [\cai(x), y]' = \Sa(x,y) \sum_{i\in I} \frac 1 T\Pin h_i + \Cou(x, y) \K .
\nn\ee\qed
\end{cor}

\subsection{Verma and contragredient Verma modules over $\g$}
Given $\Lambda\in \h^*$ we denote by $M(\Lambda)$ the induced $\g$-module
\be M(\Lambda) := U(\g) \otimes_{U(\h\oplus\n_+)} \C v_\Lambda \ee
where $\C v_\Lambda$ is the one-dimensional $(\h\oplus\n_+)$-module generated by a vector $v_\Lambda$ obeying $\n_+ \on v_\Lambda = 0$ and $H\on v_\Lambda = v_\Lambda \Lambda(H)$ for all $H\in \h$. Call $M(\Lambda)$ the \emph{Verma module} over the Lie algebra $\g$. (One should keep in mind that $\g$ has ``no Serre relations''.)
By the PBW basis theorem $U(\g) \cong_\C U(\n) \otimes_\C U(\h \oplus \n_+)$. Hence there is an isomorphism of left $U(\n)$-modules
\be M(\Lambda) \cong_\n U(\n) \otimes_\C \C v_\Lambda \cong_\n U(\n).\nn\ee
The $\Z_{\geq 0}^\ru$-grading on $\n$ defined in \eqref{wdu} induces a $\Z_{\geq 0}^\ru$-grading on the envelope $U(\n)$, 
\be U(\n) = \bigoplus_{(n_1,\dots,n_{\ru})\in \Z_{\geq 0}^{\ru}} U(\n)_{(n_1,\dots,n_{\ru})}.\ee
Then we have $M(\Lambda) =_\h \bigoplus_{(\lambda_1,\dots,\lambda_\ru)\in \Z_{\geq 0}^\ru} M(\Lambda)_{(\lambda_1,\dots,\lambda_\ru)}$ where $M(\Lambda)_{(\lambda_1,\dots,\lambda_\ru)} := U(\n)_{(\lambda_1,\dots,\lambda_\ru)} \on v_\Lambda$. Equivalently,
\be M(\Lambda)_{(\lambda_1,\dots,\lambda_\ru)} := \left\{ v\in M(\Lambda): H\on v  = v\bra H, \Lambda - \sum_{p=1}^\ru \lambda_p \alpha_p\ket  \text{ for all } H\in \h\right\}.\nn\ee 
These weight subspaces $M(\Lambda)_{(\lambda_1,\dots,\lambda_\ru)}$ are all of finite dimension. Let $M^*(\Lambda)$ denote the restricted dual space of $M(\Lambda)$ with respect to this decomposition, 
\be M^*(\Lambda) := \bigoplus_{(\lambda_1,\dots, \lambda_\ru) \in \Z_{\geq 0}^\ru } (M(\Lambda)_{(\lambda_1,\dots,\lambda_\ru)})^*.\nn\ee 
The \emph{contragredient Verma module} $M^*(\Lambda)$ is by definition this restricted dual endowed with the $\g$-module structure given by $(g\on f)(v) := f( \cai(g) \on v)$, for $f\in M(\Lambda)^*$, $v\in M(\Lambda)$, $g\in \g$.

\subsection{Verma and contragredient Verma modules over $\gse$}
Recall the definition of the centrally extended Lie algebra $\gse\cong_\C \gs \oplus \C\K$ from \S\ref{cocyclesec}.
Given $\Lambda\in (\h^\si)^*$ we denote by $M^\si(\Lambda)$ the Verma module over $\gse$,
\be M^\si(\Lambda) := U(\gs) \otimes_{U(\h^\si\oplus\n^\si_+\oplus \C\K)} \C v^\si_\Lambda \ee
where $\C v^\si_\Lambda$ is the one-dimensional $(\h^\si\oplus\n^\si_+\oplus \C\K)$-module generated by a vector $v^\si_\Lambda$ obeying $\n^\si_+ \on v^\si_\Lambda = 0$, $H\on v^\si_\Lambda = v^\si_\Lambda \Lambda(H)$ for all $H\in \h^\si$ and $\K\on v^\si_\Lambda = v^\si_\Lambda$.
As a module over $\ns$,
\be M^\si(\Lambda) \cong_{\ns} U(\ns) \otimes_\C \C v^\si_\Lambda \cong_{\ns} U(\ns).\label{Mn}\ee

We have the $\Z_{\geq 0}^\rf$-grading of $U(\ns)$ induced by \eqref{wd}. Then $M^\si(\Lambda) = \bigoplus_{[\lambda_1,\dots,\lambda_\rf]\in \Z_{\geq 0}^\rf} M^\si(\Lambda)_{[\lambda_1,\dots,\lambda_\rf]}$ where $M^\si(\Lambda)_{[\lambda_1,\dots,\lambda_\rf]} := U(\ns)_{[\lambda_1,\dots,\lambda_\rf]} \on v^\sigma_\Lambda$.
The weight subspaces $M^\si(\Lambda)_{[\lambda_1,\dots,\lambda_\rf]}$ are of finite dimension. Let $M^{\si,*}(\Lambda)$ denote the restricted dual vector space endowed with the contragredient dual $\gse$-module structure.

\subsection{Shapovalov form on Verma modules}\label{sfd}
For $\Lambda\in \h^*$ (resp. $\Lambda\in (\h^\si)^*$) there is a unique bilinear form $S$ on the $\g$-module $M(\Lambda)$ (resp. the $\gse$-module $M^\si(\Lambda)$) defined by 
\be S(v_\Lambda,v_\Lambda) = 1,\qquad S(x \on v, w) = S(v, \cai(x) \on w), \qquad \nn\ee
for all $x\in \g$ (resp. $x\in \gse$) and all $v,w$ in $M(\Lambda)$ (resp. in $M^\si(\Lambda)$). This form is symmetric and the weight subspaces are pairwise orthogonal.

\subsection{Shapovalov form on $C_p(\ns, M(\bm \Lambda)) := \Wedge p \ns \ox M(\bm \Lambda)$}\label{waf}
Let us fix now weights $\Lambda_0\in (\h^\si)^*$ and $\Lambda_1,\dots,\Lambda_N\in \h^*$. We shall write $\bm \Lambda := (\Lambda_0,\Lambda_1,\dots,\Lambda_N)$ and
\be M(\bm \Lambda) := M^\si(\Lambda_0) \ox \bigotimes_{i=1}^N M(\Lambda_i).\nn\ee
Note that $M(\bm\Lambda)$ is a module over $\ns$. With $\Mo$ as in \eqref{Modef}, we have
\be M(\bm \Lambda) = \Mo \on (v^\sigma_{\Lambda_0} \ox v_{\Lambda_1} \ox \dots \ox v_{\Lambda_N}), \label{mme}\ee
as an equality of vector spaces, and in fact of $\ns$-modules.
There is a symmetric bilinear form $S^p$ on the space 
\be C_p(\ns, M(\bm \Lambda)) := \Wedge p \ns \ox M(\bm \Lambda),\qquad p = 0,1,\dots,\nn\ee
defined by
\begin{multline} S^p(\Pin a_p \wx \dots \wx \Pin a_1 \ox x_0 \ox x_1 \ox \dots \ox x_N,
\Pin b_p \wx \dots \wx \Pin b_1 \ox y_0 \ox y_1 \ox \dots \ox y_N)\\
= \det(\Sg(\Pin a_i,\Pin b_j)_{1\leq i,j\leq p}) S(x_0,y_0) \prod_{i=1}^N S(x_i,y_i). \label{Spedefsym}\end{multline}
It respects the weight decomposition.
Let $\Shom^p$ denote the linear map defined by $S^p$, 
\be \Shom^p : \left(\Wedge p\ns \ox M(\bm \Lambda)\right)_\lambda  \to \left(\Wedge p\ns \ox M(\bm \Lambda) \right)_\lambda^*. \label{shompostsym}\ee

Let $C_p(\ns, M(\bm \Lambda)_\lambda^*$ denote the restricted dual of $C_p(\ns, M(\bm \Lambda)$ and

\subsection{Main theorem}
Fix any $\lambda = (\lambda_1,\dots,\lambda_\rf)\in \Z_{\geq 0}^\rf$. Set $m= \sum_{i=1}^\rf \lambda_i$. 
Recall that the Cartan matrix entries $b_{ij} := (\alpha_i,\alpha_j)$ defined a weighting on the ``diagonal'' hyperplanes $\Ht ijk$, as in \eqref{HdC}. The choice of weights $\bm \Lambda = (\Lambda_0;\Lambda_1,\dots,\Lambda_N)$ defines a weighting of the remaining hyperplanes of the arrangement $\cca$.
Namely,
\be a\big(\Hz {\clc i n}jk\big) := -\frac {T_i} T (\si^k\alpha_{\iota(i)},\Lambda_j),\qquad  
    a\big(\Ho{\clc i n}\big) := -(\alpha_{\iota(i)},\Lambda_0),\label{rhw}  \ee
for $j\in \{1,\dots,N\}$, $i\in \Is$, and $n\in \{1,\dots,\lambda_i\}$.

\begin{thm}\label{mt}
For each $p\in \Z_{\geq 0}$, the following diagram commutes:
\be\begin{tikzpicture}    
\matrix (m) [matrix of math nodes, row sep=4em,    
column sep=6em, text height=1ex, text depth=1ex]    
{     
C_p(\ns,M(\bm \Lambda))_{\lambda} &  C_p(\ns, M(\bm \Lambda))_{\lambda}^* \\    
C_p(\nbs, \Mb)^\Swl_{\mones} &  (C_p(\nbs,\Mb)^\Swl)_{\mones}^* \\    
\Fl^{m-p}(\cca)^\Swl  & \A^{m-p}(\cca)^\Swl.  \\    
};    
\path[->,font=\scriptsize,shorten <= 2mm,shorten >= 2mm]    
(m-1-1) edge node [above] {$\Shom^p$} (m-1-2)    
(m-2-1) edge node [above] {$\Shom^p$} (m-2-2)    
(m-3-1) edge node [above] {$(-1)^{m-p} T^p \Ghom^{m-p}$} (m-3-2);    
\path[->,shorten <= 0mm,shorten >= 2mm]    
(m-1-1) edge node [right] {$\sim$} node [left] {$\sym$} (m-2-1)    
(m-1-2) edge node [right] {$\sim$} node [left] {$\pi^*$} (m-2-2)   
(m-2-1) edge node [right] {$\sim$} node [left] {$\psis_p$} (m-3-1)    
(m-2-2) edge node [right] {$\sim$} node [left] {$(\psis_p^{-1})^*$} (m-3-2);    
\end{tikzpicture}\nn\ee  
\end{thm}
The theorem is proved in \S\ref{proofofmt} below.
\begin{rem} 
The map $\Ghom ^{m-\bl} : \Fl^{m-\bl}(\cca)^\Swl  \to \A^{m-\bl}(\cca)^\Swl$ is a map of complexes, cf. Theorem \ref{cmap}. 
The isomorphisms on the left, $\sym$ and $\psis_p$, are isomorphisms of complexes, as in Theorem \ref{symthm}. 
We may regard the isomorphisms on the right, $\pi^*$ and $(\psis^{-1})^*$, as \emph{defining} the structure of a differential complex on $(C_p(\nbs,\Mb)^\Swl)_{\mones}^*$ and $C_p(\ns, M(\bm \Lambda))_{\lambda}^*$. It would be interesting to give an intrinsic definition of the differentials on  $(C_p(\nbs,\Mb)^\Swl)_{\mones}^*$ and $C_p(\ns, M(\bm \Lambda))_{\lambda}^*$ in the spirit of \cite[\S6.16]{SV}. 
\end{rem}

Define $L(\Lambda_i) := M(\Lambda_i)/\ker S$ for $i=1,\dots,N$, $L^\si(\Lambda_0) := M^\si(\Lambda_0)/\ker S$,  and $L(\bm \Lambda) := L^\si(\Lambda_0) \ox \bigotimes_{i=1}^N L(\Lambda_i)$. Define $\nts:= \ns/(\ker\Sg\cap \ns)$. 
\begin{cor}\label{irrepcor}
The maps $\psis_\bl\circ \sym$ induce isomorphisms of complexes
\be C_\bl(\nts, L(\bm\Lambda))_\lambda \xrightarrow\sim \Ghom(\Fl^{m-\bl}(\cca)^\Swl) \nn\ee\qed 
\end{cor}

\subsection{Proof of Theorem \ref{mt}}\label{proofofmt}
In view of Theorems \ref{zSGthm} and \ref{symthm}, what remains to be established is the commutativity of the upper square in Theorem \ref{mt}. We will give the details of the proof  for the special case $N=0$, $p=0$, \ie we will prove the commutativity of the square
\be\begin{tikzpicture}    
\matrix (m) [matrix of math nodes, row sep=4em,    
column sep=4em, text height=1ex, text depth=1ex]    
{     
M^\si(\Lambda_0)_{\lambda} &  M^\si(\Lambda_0)_{\lambda}^* \\    
(M_0)_{\mones} &  (M_0)_{\mones}^*. \\    
};    
\path[->,font=\scriptsize,shorten <= 2mm,shorten >= 2mm]    
(m-1-1) edge node [above] {$\Shom$} (m-1-2)    
(m-2-1) edge node [above] {$\Shom$} (m-2-2);    
\path[->,shorten <= 0mm,shorten >= 2mm]    
(m-1-1) edge node [left] {$\sym$} (m-2-1)    
(m-1-2) edge node [right] {$\pi^*$} (m-2-2);    
\end{tikzpicture}\label{zesq}\ee    
The general proof is similar. 

We keep the conventions of \S\ref{clsec}. In particular, we suppose that $\lambda$ dominates $\mu$. 
We have, for each $I\in \mc I$, the symmetrization map
\be \sym_I : M^\si(\Lambda_0)_{\mu} \to  (M_0)_{\wtI I} .\nn\ee
\begin{prop}\label{swapprop}
Suppose $x \in \ns_\mu$ and $w\in M^\si(\Lambda_0)_\lambda$. Then, for all $I\in \mc I$, 
\be \sym_\Ip( \cai(x) w ) = \frac{1}{|\Sw_\mu|}\sym_I(\cai(x)) \triangleright \sym(w)\nn\ee
as an equality in $(M_0)_{\wtI \Ip}$.
\end{prop}
\begin{proof}
By linearity, we may suppose that $w = \Pin b_1 \dots \Pin b_p v_0$ for some commutators $b_1,\dots,b_p\in \n$. We must show that 
\be  \frac{1}{|\Sw_\mu|} \sym_I(\cai(x))  \triangleright \sym(\Pin b_1 \dots \Pin b_p v_0) = \sym_\Ip( \cai(x) \Pin b_1 \dots \Pin b_p v_0 ) .\nn\ee
We proceed by induction on $p$. 

We have $\mu = \wt x$. Let $\rho = \wt \Pin b_1$.  
For the inductive step, there are four cases to consider: 
\begin{enumerate}[(i)]
\item\label{nd} $\mu \neq \rho$ and neither dominates the other.
\item\label{bdm}  $\mu \neq \rho$ and $\rho$ dominates $\mu$;
\item\label{mdb}   $\mu \neq \rho$ and $\mu$ dominates $\rho$;
\item\label{eq}  $\mu = \rho$;
\end{enumerate}
(For the base case $p=1$ the argument is the same except that then $\rho = \lambda$, which dominates $\mu$, so we can only be in case (\ref{bdm}) or case (\ref{eq}).)

For brevity, set $y:= \Pin b_2 \dots \Pin b_p v_0\in M^\si(\Lambda_0)$. By an obvious analog of Lemma \ref{symcomlem2}, we have
\be  
 \frac{1}{|\Sw_\mu|} \sym_I(\cai(x))  \triangleright \sym(\Pin b_1 y) 
=  \frac{1}{|\Sw_\mu|} \sym_I(\cai(x))  \triangleright \sum_{J\in \mc J} \sym_J (\Pin b_1) \sym_{\Jp}(y), \label{Jsum}\ee
where $\mc J$ denotes the set of subsets $J\subset \{1,\dots,m\}$ such that $J = \{\clc i t: 1\leq i \leq \rf, t\in J_i\}$ for some subsets $J_i\subset \{1,\dots,\lambda_i\}$ with $|J_i| = \rho_i$. 

Consider case (\ref{nd}). The fact that neither of $\mu$ and $\rho$ dominates the other implies that, for every $J$ in the sum \eqref{Jsum}, neither of the sets $I,J$ contains the other. It follows that $[\sym_{I}(\cai(x)), \sym_J(\Pin b_1)]' = 0$. If $J\cap I \neq \emptyset$ then additionally $\sym_I(\cai(x))  \triangleright \sym_{\Jp}(y) = 0$, and hence $\sym_I(\cai(x)) \triangleright \sym_J(\Pin b_1) \sym_\Jp (y) = 0$. Thus the only contributions come from $J$ such that $J \cap I = \emptyset$, and we have
\be \frac 1{|\Sw_\mu|}\sym_I(\cai(x))  \triangleright \sym(\Pin b_1 y) 
= \frac 1{|\Sw_\mu|} \sum_{\substack{ J \in \mc J \\ J\subset \Ip}} \sym_J (\Pin b_1)  \sym_I(\cai(x))  \triangleright \sym_{\Jp}(y).\nn\ee
Using the inductive assumption, we see that this is equal to 
\be \sum_{\substack{ J \in \mc J \\ J\subset \Ip}} \sym_J (\Pin b_1) \sym_{\Jp \setminus I}(\cai(x)  y) = \sym_{\Ip}( \Pin b_1 \cai(x) y).\nn\ee 
Since neither of $\mu$, $\rho$ dominates the other, $[\cai(x), \Pin b_1] = 0$, and so this last expression is equal to $\sym_\Ip( \cai(x) \Pin b_1 y )$, as required.

Consider case (\ref{bdm}). Since $\rho$ dominates $\mu$, for every $J$ in the sum \eqref{Jsum} it must be that either $I \subsetneq J$ or neither of $I,J$ contains the other. When $I \subsetneq J$ we have $\sym_I(\cai(x))  \triangleright\sym_{\Jp}(y) = 0$ and
\be \frac 1{|\Sw_\mu| }[\sym_I(\cai(x)) , \sym_J(\Pin b_1)]' = \frac 1{|\Sw_\mu| }[\sym_I(\cai(x)) , \sym_J(\Pin b_1)] = \sym_{J\setminus I}( [ \cai(x), \Pin b_1] ) \nn\ee
by Lemma \ref{symcomlem}. 
The contribution from terms in the sum on $J$ in \eqref{Jsum} such that neither of $I,J$ contains the other is $\sym_{\Ip}( \Pin b_1 \cai(x) y)$, for the same reason as in case \eqref{nd} above. Thus, in total,
\begin{align} \frac 1{|\Sw_\mu|} \sym_I(\cai(x))  \triangleright\sym(\Pin b_1 y) 
 &= \sum_{\substack{J\in \mc J\\ J\supsetneq I}} \sym_{J\setminus I}( [ \cai(x), \Pin b_1] ) \sym_{\Jp}(y)  + \sym_{\Ip} ( \Pin b_1 \cai(x) y)\nn\\
&= \sym_{\Ip} ( [ \cai(x), \Pin b_1 ] y )  
   + \sym_{\Ip} ( \Pin b_1 \cai(x) y) \nn\\ &= \sym_{\Ip} (\cai(x) \Pin b_1 y) \nn
\end{align}
as required. 

Consider case (\ref{mdb}). Since $\mu$ dominates $\rho$, for every $J$ in the sum \eqref{Jsum} it must be that either $J \subsetneq I$ or neither of $I,J$ contains the other. When $J \subsetneq I$ we have $\sym_I(\cai(x)) \triangleright \sym_{\Jp}(y) = 0$ and
\begin{multline}
 [\sym_I(\cai(x)) , \sym_J(\Pin b_1)]' = [\sym_I(\cai(x)) , \sym_J(\Pin b_1)] =  \cai\left([\sym_J(\cai(\Pin b_1), \sym_I(x) ]\right)\\
 =  |\Sw_\rho| \, \cai\left(\sym_{I \setminus J}([ \cai(\Pin b_1), x])\right)  
 =  |\Sw_\rho| \, \sym_{I\setminus J}( [ \cai(x), \Pin b_1] ) \nn\end{multline}
by Lemma \ref{symcomlem}. 
The contribution from terms in the sum on $J$ in \eqref{Jsum} such that neither of $I,J$ contains the other is still $\sym_{\Ip}( \Pin b_1 \cai(x) y)$, as in case \eqref{nd}.
Hence
\be \frac 1{|\Sw_\mu|} \sym_I(\cai(x))  \triangleright \sym(\Pin b_1 y) 
 = \frac{|\Sw_\rho|}{|\Sw_\mu|} \sum_{\substack{J\subset \mc J\\ J\subsetneq I}} \sym_{I\setminus J}( [ \cai(x), \Pin b_1] )  \triangleright \sym_{\Jp}(y) +\sym_{\Ip} ( \Pin b_1 \cai(x) y) .\nn
\ee
In every term in the sum on $J$ here, the inductive assumption allows us to re-write the summand in a form independent of $J$:
\be \sym_{I\setminus J}( [ \cai(x), \Pin b_1] )  \triangleright \sym_{\Jp}(y) 
= |\Sw_{\mu - \rho}| \,\sym_{\Ip} ([ \cai(x),\Pin b_1] y)
\ee
The number of terms in the sum is $\frac{|\Sw_\mu|}{|\Sw_\rho||\Sw_{\mu - \rho}| }$, cf. \eqref{isize}. Therefore the combinatorial factors cancel and we find
\be \frac 1{|\Sw_\mu|} \sym_I(\cai(x))  \triangleright \sym(\Pin b_1 y) 
 = \sym_{\Ip} ([ \cai(x),\Pin b_1] y) +\sym_{\Ip} ( \Pin b_1 \cai(x) y)  =
\sym_{\Ip} (\cai(x) \Pin b_1 y) \nn
\ee
as required.

Finally, consider case (\ref{eq}). Since $\mu=\rho$, for every $J$ in the sum \eqref{Jsum} either $J=I$ or neither of $I$, $J$ contains the other. The contribution from the latter class of terms is still $\sym_{\Ip}( \Pin b_1 \cai(x) y)$. Consider the $J=I$ term in the sum. What we must show is that
\be \frac 1{|\Sw_\mu|} [\sym_I(\cai(x)), \sym_I(\Pin b_1)]' \sym_{\Ip} (y)
= \sym_{\Ip} ( [\cai(x),\Pin b_1] y),\nn\ee
as an equality in $\sym(M^\si(\Lambda_0)_\lambda) = (M_0)_{\Ip}$. 
And indeed, using Corollary \ref{comcorgen} and \eqref{hsym2}, we have
\begin{align} \sym_{\Ip} ( [\cai(x),\Pin b_1] y) 
&=\sym_{\Ip} \left(  \frac 1 T  S_{\g}(x,\Pin b_1) \sum_{i=1}^\rf \mu_i \Pin H_{\iota(i)} y + \Co(\cai(x),\Pin b_1) \K y \right) \nn\\
&= \left(\frac{1}{T}S_{\g}(x,\Pin b_1) \sum_{i\in I} \Pin h_i +  \Co(\cai(x),\Pin b_1) \K\right) \sym_{\Ip}(y)\nn
\end{align}
By Lemma \ref{SSlem} and the definition of $\Co$, this is equal to 
\begin{align} 
& \frac{1}{|\Sw_\mu|} 
\left( \frac 1 T\Se(\sym_I(x), \sym_I(\Pin b_1)) \sum_{i\in I} \Pin h_i 
+ \Cou(\sym_I(\cai(x)), \sym_I(\Pin b_1)) \K \right)   \sym_{\Ip}(y)\nn\\
&= \frac{1}{|\Sw_\mu|} \left[ \sym_I(x), \sym_I(\Pin b_1) \right]' \sym_{\Ip}(\Pin b_1)
\end{align}
as required.
\end{proof}
\begin{prop}\label{SSm0lem}
Let $v,w\in M^\si(\Lambda_0)_\lambda$. Then $\frac{1}{|\Sw_\lambda|} \Se(\sym(v),\sym(w)) = S(v,w)$.
\end{prop}
\begin{proof} By complete induction on $|\lambda|= \lambda_1+\dots+\lambda_\rf$. 
The result is true for $|\lambda|=0$. For the inductive step, suppose $v = x u$ for some $x\in \ns_\mu$ and $u\in M^\si(\Lambda_0)_{\lambda-\mu}$. Then, by definition of $S$ and the inductive assumption, we have
\be S(x u, w) = S(u, \cai(x) w) =\frac 1{|\mc I|} \sum_{I\in \mc I} \frac{1}{|\Sw_{\lambda-\mu}|}  \Se( \sym_{\Ip}(u), \sym_{\Ip} (\cai(x) w) ).\nn\ee
Therefore 
\begin{align} S(x u, w) &= \frac 1{|\mc I|} \sum_{I\in \mc I} \frac{1}{|\Sw_{\lambda-\mu}|}  \Se( \sym_{\Ip}(u), \frac{1}{|\Sw_\mu|}\sym_I(\cai(x)) \triangleright \sym(w))\nn\\
&= \frac 1{|\mc I| |\Sw_\mu||\Sw_{\lambda-\mu}|} \sum_{I\in \mc I} \Se(\sym_I(x) \sym_{\Ip}(u), \sym(w))\nn\\
&= \frac 1{|\Sw_{\lambda}|} \Se( \sym(xu), \sym(w)),\nn
\end{align}
where we used Proposition \ref{swapprop} and \eqref{isize}. This completes the inductive step.
\end{proof}

Recall that $\sym \circ \pi =  \pi \circ \sym = |\Sw_{\lambda}| \id$, as in Lemma \ref{isolem}. 
Therefore Proposition \ref{SSm0lem} is equivalent to the statement that
\be S(v, \pi( \cdot)) = \Sa( \sym(v), \cdot), \qquad\text{for all } v\in M^\si(\Lambda_0)_\lambda, \nn\ee
as an equality in $(M_0)_\mones^*$. Thus the square in \eqref{zesq} is indeed commutative.

\section{Properties of the cocycle $\Co$}\label{sec: cop}
We keep the conventions of \S\ref{sec: cd}. 
We have the quotient Lie algebra $\gt := \g/\ker\Sg$, and its Lie subalgebra of $\si$-invariants,
\be \gts:= (\g/\ker \Sg)^\sigma  \cong \gs/\ker((\Sg)|_{\gs}) , \nn\ee
as in Lemma \ref{isomlem}.
In \S\ref{cocyclesec} we defined a cocycle $\Co$ on $\gs$. According to Proposition \ref{kercoprop},  $\Co$ vanishes on $\ker \Sg \subset \gs$. Therefore $\Co$ descends to a well-defined cocycle on the quotient $\gts$. 

\begin{prop}\label{cbprop} If $\Co$ is coboundary then it is zero.
\end{prop}
\begin{proof}
Suppose that $\Co$ is coboundary, \ie that
\be \Co(x,y) = \alpha([x,y]) \nn\ee
for some linear map $\alpha : \gts \to \C$. What we must show is that $\alpha=0$. 
Let $\nts := \ns/\ker \Sg$ and $\nts_+ := \ns_+/\ker\Sg$. Suppose $x\in \nts_{\lambda}$ and $y\in (\nts_+)_{\mu}$. By definition $\Co(x,y)=0$ unless $\mu= \lambda$, in which case $[x,y] \in \hs$. Thus $\alpha$ is non-zero at most on $\hs \cap [\gs ,\gs]$. 
The space $\hs \cap [\gs ,\gs]$ spanned by the elements $\Pin H_{\iota(i)} = [\Pin E_{\iota(i)}, \Pin F_{\iota(i)}]$, $i=1,\dots,\rf$. So to show that $\alpha=0$ it is enough to check that $\Co(\Pin E_{\iota(i)}, \Pin F_{\iota(i)}) = 0$ for $i=1,\dots,\rf$. And indeed this is manifestly true: for any given $i$, we have the flag $\psis(\sym_{\{i\}}(\Pin F_{\iota(i)}v_0)) = \psis( \Pin f_iv_0) = (\C^1 \supset \{0\})$, and we are to compute the inner product of this flag with itself with respect to the arrangement in $\C^1$ consisting of only the ``diagonal'' hyperplanes -- but in $\C^1$ this is the arrangement with no hyperplanes at all. 
\end{proof}

\begin{cor}\label{sscor} If $\gts$ is semisimple and of finite dimension, then $\Co = 0$. 
\end{cor}
\begin{proof}
Recall that when $\gts$ is semisimple and of finite dimension, Whitehead's lemma states that $H^2(\gts, \C) = 0$ (e.g. \cite[Lemma 6, Ch. 3]{Jacobson}). That is, every (2-)cocycle on $\gts$ is a coboundary. \end{proof}

The Lie algebra $\gts$ is finite-dimensional and semisimple in particular in the following situations:
\begin{enumerate}[(i)]
\item $\gts$ is the folding of a simply-laced finite-dimensional simple Lie algebra by an admissible diagram automorphism \cite[\S7.9]{KacBook}. That is, the pair $(B,T)$ is one of
\be (\mathrm{D}_{\ell+1},2), \quad (\mathrm{A}_{2\ell-1},2), \quad (\mathrm{E}_6,2), \quad\text{or}\quad (\mathrm{D}_4,3) \nn\ee
and $\gts$ is the finite-dimensional simple Lie algebra of type $\mathrm B_\ell, \mathrm C_\ell, \mathrm F_4$ or $\mathrm G_2$ respectively.    
\item $A$ is the direct sum of $T$ copies of a Cartan matrix $A'$ of finite type. Then $\gts$ is the finite-dimensional semisimple Lie algebra of type $A'$, $\gt$ is the direct sum of $T$ copies of $\gts$, and $\sigma$ is the automorphism which cyclically permutes these copies. 
\end{enumerate}

Now define a matrix $\Bs = (\bs_{ij})_{i,j\in \Is}$ by
\be \bs_{ij} := \sum_{k\in \ZT} b_{\iota(i), \siod^k\iota(j)}.\nn\ee 
Observe that $\Bs$ is symmetric.
\begin{prop} There is an embedding of Lie algebras $\g(\Bs)\into \gs$. 
\end{prop}
\begin{proof} 
We have 
\begin{align} [\Pin E_{\iota(i)},\Pin F_{\iota(j)}] &= \sum_{k,l=0}^{T-1} [E_{\siod^k\iota(i)},F_{\siod^l\iota(j)}] 
= \delta_{i,j}  \frac{T^2}{T_i^2} \sum_{k,l=0}^{T_i-1} [E_{\siod^k\iota(i)},F_{\siod^l\iota(i)}]\nn\\
&= \delta_{i,j} \frac{T^2}{T_i^2}  \sum_{k=0}^{T_i-1} H_{\siod^k\iota(i)} 
= \delta_{i,j} \frac{T^2}{T_i^2}  \frac{T_i}{T} \sum_{k=0}^{T-1} H_{\siod^k\iota(i)} 
= \delta_{i,j} \frac{T}{T_i} \Pin H_{\iota(i)}\end{align}
and 
\begin{align}
[\Pin H_{\iota(i)}, \Pin E_{\iota(j)}] &= \sum_{k,l=0}^{T-1} [ H_{\siod^k\iota(i)}, E_{\siod^l\iota(j)}] = \sum_{k,l=0}^{T-1} b_{\siod^k\iota(i),\siod^l\iota(j)} E_{\siod^l\iota(j)}\nn\\ 
&=  \sum_{k,l=0}^{T-1} b_{\iota(i),\siod^{l-k}\iota(j)} E_{\siod^l\iota(j)} = \sum_{k,l=0}^{T-1} b_{\iota(i),\siod^k\iota(j)} E_{\siod^l\iota(j)} = \bs_{i,j} \Pin E_{\iota(j)}.\end{align}    
and similarly for $[\Pin H_{\iota(i)},\Pin F_{\iota(j)}]$. 
Therefore $\Pin E_{\iota(i)}$, $\frac{T_i}T\Pin F_{\iota(i)}$ and $\Pin H_{\iota(i)}$, $i\in \Is$, generate a copy of  $\g(\Bs)$ inside $\gs$.
\end{proof}
\begin{prop} Suppose $x\in\g(\Bs)\into \gs$. Then $\Co(y,x) = 0$ for all $y\in \gs$. 
\end{prop}
\begin{proof} It is enough to consider $x\in \n(\Bs)$. (The case $x\in \n_+(\Bs)$ follows, while $x\in \h(\Bs)$ is immediate.) Consider the flag $\psis(\sym(xv_0))$. To say that $x\in \n(\Bs)$ is to say that  $x$ is a linear combination of commutators in the $\Pin F_{\iota(i)}$, $i=1,\dots,r$. By expanding commutators we can express $xv_0$ as a linear combination of monomials of the form $\Pin F_{\iota(i_1)} \dots \Pin F_{\iota(i_k)}v_0$, $i_1,\dots,i_k\in \Is$. Now $\psis(\sym(\Pin F_{\iota(i_1)} \dots \Pin F_{\iota(i_k)}v_0))$ is a linear combination of flags each of which has a first step of the form
\be \C^m \supset (t_p = 0) \supset \dots \ee
for some $p$. No such flag is adjacent to any tuple of hyperplanes from the arrangement $\cco I$ of ``diagonal'' hyperplanes $\Ht ijk$. Therefore by definition $\Co(y,x)$ is zero, cf. \eqref{G0def}.
\end{proof} 
\begin{cor}\label{eqcor} If $\g(\Bs) = \gs$ then $\Omega = 0$. \qed
\end{cor}
The next example shows that the embedding $\g(\Bs)\into \gs$ is not always an isomorphism, and that $\Co$ is not always zero.
\begin{exmp}\label{conzex} Suppose $\gt =\gt(B)$ is the untwisted affine Kac-Moody Lie algebra of type $D_4^{(1)}$. Let us  label the nodes of the Dynkin diagram as follows:
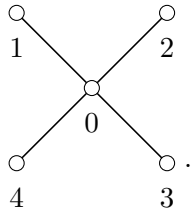
\be
\begin{tikzpicture}[baseline =-30pt,scale=1]
\draw[semithick] (0,0) -- (1,1); \draw[semithick] (0,0) -- (1,-1);
\draw[semithick] (0,0) -- (-1,1); \draw[semithick] (0,0) -- (-1,-1);
\filldraw[fill=white] (0,0) circle (1mm) node [below=2mm] {$0$};    
\filldraw[fill=white] (1,1) circle (1mm) node [below=2mm] {$2$};    
\filldraw[fill=white] (1,-1) circle (1mm)  node [below=2mm] {$3$};    
\filldraw[fill=white] (-1,1) circle (1mm)  node [below=2mm] {$1$};
\filldraw[fill=white] (-1,-1) circle (1mm)  node [below=2mm] {$4$};
\end{tikzpicture}.
\nn\ee
Let $\siod$ be the diagram automorphism $(1234)$. 
Set $\iota(0) = 0$, $\iota(1) = 1$. Then
\be \Bs = \bmx 8 & -4 \\ -4 & 2 \emx= \bmx 4 & \\ & 1 \emx \bmx 2 & -1 \\ -4 & 2 \emx.\nn\ee 
Thus $\gt(\Bs)$ is the twisted affine Kac-Moody Lie algebra of type $A_2^{(2)}$, whose Dynkin diagram is  \be\begin{tikzpicture}[baseline =-6,scale=1]
\draw[semithick] (0,-.09) -- (1,-.09); \draw[semithick] (0,0.09) -- (1,0.09); 
\draw[semithick] (0,-.03) -- (1,-.03); \draw[semithick] (0,0.03) -- (1,0.03);
\draw[semithick] (.4,.2) -- (.6,0) -- (.4,-.2); 
\filldraw[fill=white] (0,0) circle (1mm) node [below=2mm] {$0$};    
\filldraw[fill=white] (1,0) circle (1mm) node [below=2mm] {$1$};    
\end{tikzpicture}.\nn\ee
In $\g(\Bs) \into  \gs$ the subspace $\n(\Bs)_{[1,2]}$ has dimension one: it is spanned by
$\big[ [\Pin F_0, \Pin F_1],\Pin F_1 \big]$.
But $\ns_{[1,2]}$ has dimension strictly greater than one. Indeed, already in the quotient $\nts \subset \gts$, the subspace $\nts_{[1,2]}$ has dimension two: it is spanned by the elements $\big[ [\Pin F_0, \Pin F_1],\Pin F_1 \big]$ and\footnote{Note that in $\nts$, $[F_1,F_2]=0$ and hence $\big[[ F_0, F_2], F_1 \big] = -\big[[ F_0, F_1], F_2 \big]$ by the Jacobi identity.}
\be \Pin{\big[[ F_0, F_1], F_2 \big]} = \big[ [ F_0 , F_1], F_2\big] + \big[ [ F_0 , F_2], F_3\big] + \big[ [ F_0 , F_3], F_4\big] + \big[ [ F_0 , F_4], F_1\big]. \nn\ee
Hence $\g(\Bs) \not\cong  \gs$.

Now we shall show that 
\be \Co( \Pin{\big[[ E_0, E_1], E_2 \big]},\Pin{\big[[ F_0, F_1], F_2 \big]}) = -16.\nn\ee
Let $\omega = \sqrt{-1}$. We have coordinates $t_0, t_{\clc 11}, t_{\clc 12}$ on $\C^3$. Let\footnote{Here we choose the sign convention in which $\psis( \Pin f_{\clc 12} \wx \Pin f_{\clc 11} \wx \Pin f_{0} \ox v_0 ) = + (\C^3)$.}  
\begin{align} X&:=  \psis(\sym (\Pin{\big[[ F_0, F_1], F_2 \big]} v_0))  \nn\\
&= \sum_{s\in \Sigma_2} \sum_{k\in \Z_4} \psis( \Pin{\big[[ \f{-k}{0}, \f 0 {\clc 1{s(1)}} ], \f 1{\clc 1{s(2)}}\big] } v_0)\nn\\
&= \sum_{s\in \Sigma_2} \sum_{k\in \Z_4} (-1)^{|s|} \Big( \C^3 \supset (t_0 = \omega^kt_{\clc 1{s(1)}}) \nn\\
  & \qquad\qquad\qquad\qquad\qquad\qquad\supset ( t_0 =\omega^k t_{\clc 1{s(1)}} = \omega^{k+1} t_{\clc 1{s(2)}}) \supset (t_0 = t_{\clc 11} = t_{\clc 12} = 0)\Big).\nn
\end{align}
We are to evaluate the inner product $\Geom_{\cco {0,\clc 11,\clc 12}}(X,X)$. The hyperplanes of the arrangement $\cco {0,\clc 1 1, \clc 12}$  are 
\begin{itemize}
\item $\Ht 0 {\clc 1p} k  = (t_0 = \omega^k t_{\clc 1p})$, $k\in \Z_4$, $p\in \{1,2\}$, of weight $-1$, 
\item $\Ht 12 0 = ( t_{\clc 11} = t_{\clc 12})$, of weight 2, and
\item $\Ht 12 k = ( t_{\clc 11} = \omega^k t_{\clc 12})$, $k\in \Z_4 \setminus \{0\}$, of weight 0.
\end{itemize}
We may ignore the last of these since they have weight 0. Therefore the only unordered tuples of hyperplanes we need consider are those of the form $(\Ht 0{\clc 1{s(1)}} k, \Ht 0{\clc 1{s(2)}} {k+1}, x$) for some hyperplane $x$.
Note that for any $k\in \Z_4$, $s\in \Sigma_2$,
\be \bra X, \Ht 0{\clc 1{s(1)}} k \wx \Ht 0{\clc 1{s(2)}} {k+1} \wx \Ht 0{\clc 1{s(1)}} {k+2} \ket =1-1=0.\nn\ee
Thus we get, as claimed,
\begin{align} \Geom_{\cco {0,\clc 11,\clc 12}}(X,X) &= 
\sum_{s\in \Sigma_2}\sum_{k\in \Z_4}\Big( \bra X, \Ht 0{\clc 1{s(1)}} k \wx \Ht 0{\clc 1{s(2)}} {k+1} \wx \Ht {\clc 1{s(1)}}{\clc 1{s(2)}} 0 \ket^2 \times (-1)(-1)(2)  \nn\\
& \qquad\qquad +   \bra X, \Ht 0{\clc 1{s(1)}} k \wx \Ht 0{\clc 1{s(2)}} {k+1} \wx \Ht 0{\clc 1{s(1)}} {k+1} \ket^2 \times (-1)(-1)(-1)\nn\\
& \qquad\qquad +   \bra X, \Ht 0{\clc 1{s(1)}} k \wx \Ht 0{\clc 1{s(2)}} {k+1} \wx \Ht 0{\clc 1{s(1)}} {k-1} \ket^2 \times (-1)(-1)(-1)\nn\\
& \qquad\qquad +   \bra X, \Ht 0{\clc 1{s(1)}} k \wx \Ht 0{\clc 1{s(2)}} {k+1} \wx \Ht 0{\clc 1{s(2)}} {k+2} \ket^2 \times (-1)(-1)(-1)\nn\\
& \qquad\qquad +   \bra X, \Ht 0{\clc 1{s(1)}} k \wx \Ht 0{\clc 1{s(2)}} {k+1} \wx \Ht 0{\clc 1{s(2)}} {k-2} \ket^2 \times (-1)(-1)(-1) \Big)\nn\\
&= \sum_{s\in \Sigma_2}\sum_{k\in \Z_4}  (2 - 1 -1 -1 -1) = 2 \times 4 \times (-2) = -16. \nn
\end{align}
\end{exmp}

\begin{rem}\label{borrem}
Suppose $B$ is a symmetrized generalized Cartan matrix, as in Remark \ref{KMrem}, so that $\gt$ is a Kac-Moody Lie algebra. Then the fixed-point subalgebra $\gts$ is not a Kac-Moody algebra in general. However it does belong to the larger class of \emph{Borcherds} algebras \cite{Borcherds88}. 
More generally, in fact, $\gts$ is a Borcherds algebra whenever $\gt$ is a Borcherds algebra \cite[Theorem 5.1]{Borcherds88}.   
Note also that every central extension of a Borcherds algebra is isomorphic to some quotient of the corresponding \emph{universal Borcherds algebra}; see \cite{Borcherds91}. 
\end{rem}

\section{Canonical element, dimension formulas, and weight function}\label{cosec}

\subsection{Bases of flags spaces}
We have the Orlik-Solomon algebra $\A^\bl(\cca) \cong_\C \Fl^\bl(\cca)^*$, cf. \S\ref{os}, for the cyclotomic discriminantal arrangement of \S\ref{sec: cda}.  
We will now define a dual pair of bases of $\Fl^m(\cca)$ and $\A^m(\cca)$, and hence an expression for the canonical element in $\Fl^m(\cca)\otimes \A^m(\cca)$. 

For each $p=1,\dots,m$, let us construct a set of preferred flags $\Flag^p(\cca,\oo)$ in $\Flag^p(\cca)$. 
It will be a disjoint union
\be \Flag^p(\cca,\oo) = \bigsqcup_I \Flag^p(\cca,\oo)_I\nn\ee
over the set of $p$-element subsets $I\subset \{1,\dots,m\}$. 
First, we must set 
\be \Flag^0(\cca,\oo) = \Flag^0(\cca,\oo)_{\{\}} := \Flag^0(\cca) = \{ (L^0 = \C^m) \}.\nn\ee 
Then we proceed recursively: for $p=0,1,\dots,m-1$ and for $1\leq i<i_1<i_2<\dots< i_p\leq m$,
\begin{multline} \Flag^{p+1}(\cca,\oo)_{\{i,i_1,\dots,i_p\}} \\
= \bigg\{ 
\left(L^0 \supset \dots \supset L^p \supset L^p \cap H\right)  : 
(L^0 \supset \dots \supset L^p)\in \Flag^{p}(\cca,\oo)_{\{i_1,\dots,i_p\}} , \\ 
H\in \big\{
\Ht j{i}k\big\}_{1\leq j<i;k\in \ZT}\cup \big\{\Hz {i}jk\big\}_{1\leq j\leq N;k\in \ZT} \cup \big\{\Ho {i}\big\} \bigg\}.\label{Frec}\end{multline}
\begin{prop}\label{flIprop} $\left| \Flag^p(\cca,\oo)_{I} \right| = \prod_{i\in I} \left( 1 + (i-1+N)T\right)$.
\end{prop}
\begin{proof}
According to the recursion \eqref{Frec}, the set of preferred flags $\Flag^{p+1}(\cca,\oo)_{\{i,i_1,\dots,i_p\}}$ is obtained by taking each of the preferred flags $F = (L^0 \supset\dots\supset L^p) \in\Flag^{p}(\cca,\oo)_{\{i_1,\dots,i_p\}}$ in turn, and extending it in every one of the following $1+ NT+ (i-1)T$ different ways:
\begin{enumerate}[(1)]
\item ``setting $t_{i}$ to zero'': that is, letting $L^{p+1} = L^p \cap \Ho {i}$;
\item ``setting $t_{i}=\omega^kz_j$'': that is, letting $L^{p+1} = L^p \cap \Hz {i}jk$, for $k\in \ZT$ and $j=1,\dots,N$;
\item ``setting $t_{i}=\omega^k t_j$'': that is, letting $L^{p+1} = L^p\cap \Ht j{i}{-k}$, for $k\in \ZT$ and $j=1,\dots,i-1$.
\end{enumerate}
The result follows by induction.
\end{proof}
\begin{exmp} Suppose $T=1$ and $N=1$, and take $m=2$. So we have coordinates $t_1,t_2$. We have
\be \Flag^1(\cc_{0,z;2},\oo)_{\{1\}} = \{ \C^2 \supset ( t_1= z) , \C^2\supset (t_1=0) \}.\nn\ee 
Note that this set does not include the flag $\C^2 \supset (t_1=t_2)$. This flag does appear in 
\be \Flag^1(\cc_{0,z;2},\oo)_{\{2\}} = \{ \C^2\supset ( t_2= t_1),  \C^2\supset  (t_2=0),  \C^2\supset(t_2=z) \}.\nn\ee 
From $ \Flag^1(\cc_{0,z;2},\oo)_{\{2\}}$ we compute 
\begin{align} \Flag^2(\cc_{0;2},\oo)_{\{1,2\}} = \{ &  \C^2\supset (t_2=t_1) \supset (t_2=t_1=0) ,   \C^2\supset (t_2=t_1) \supset (t_2=t_1=z) , \nn\\
&  \C^2\supset (t_2=0)\supset  (t_2 = 0, t_1=z) ,  \C^2\supset (t_2=0)\supset  (t_2 = 0 = t_1),\nn\\
&  \C^2\supset (t_2=z)\supset  (t_2 = z= t_1) ,  \C^2\supset (t_2=z)\supset  (t_2 = z, t_1=0)\}.\nn\end{align}   
Note that $\Flag^2(\cc_{0,z;2},\oo)$ does \emph{not} contain, for example, the flag $  \C^2\supset (t_1=z) \supset (t_2=t_1=z) $.  
\end{exmp}

As the notation indicates, $\Flag^p(\cca,\oo)$ corresponds to a framing $\oo$ of the arrangement $\cca$ in the sense of \S\ref{fb}, \ie to a choice of hyperplane $H(L)$ for every edge $L$. 
Indeed, this framing $\oo$ is as follows. Given any edge $L$ of $\cca$, let $j$ be smallest such that $L$ is contained in either: $\Ht ijk$ for some $i<j$ and $k\in \ZT$, or $\Hz jik$ for some $i\in\{1,\dots,N\}$ and $k\in \ZT$, or $\Ho j$. Then for that $j$,
\begin{enumerate}[(1)]
\item if $L\subset \Ho j$, let $H(L) = \Ho j$;
\item if $L\subset \Hz jik$, let $H(L) = \Hz jik$;
\item otherwise, let $H(L) = \Ht ijk$ for $i$ as small as possible.  
\end{enumerate}

\begin{exmp} With $m=3$ and $N=1$ we have for instance
\be H( \{t_2 = t_1, t_3 = \omega z_1 \}) = \{t_2=t_1\} =\Ht 12 0, \qquad H(\{t_3 = t_1, t_2= \omega z_1\}) = \{t_2 = \omega z_1\} = \Hz 211.\nn\ee
\end{exmp} 

Consequently, by the results of \S\ref{fb}, we have the following.
\begin{thm}$ $\label{bthm2}
\begin{enumerate}[(i)]
\item\label{part1} $\Flag^p(\cca,\oo)$ is a basis of $\Fl^p(\cca)$;
\item\label{part2} the dual basis of $\Fl^p(\cca)^* \cong \A^p(\cca)$ is 
\be \left\{H(L^1) \wx H(L^{2}) \wx \dots \wx H(L^p) \right\}_{(L^0 \supset L^1 \supset \dots \supset L^p) \in \Flag^p(\cca,\oo)}.\nn\ee
\end{enumerate}\qed
\end{thm}

Combining this with Proposition \ref{flIprop}, we obtain the dimensions of the flag spaces of the cyclotomic discriminantal arrangements, as follows.
\begin{thm}\label{Pointhm} The Poincare polynomial $P(x) :=  \sum_{p=0}^m  x^{p} \dim(\Fl^p(\cca))$ of the arrangement $\cca$ is given by
\be P(x) = \prod_{p=0}^{m-1} \Big(1+\big(1+(p+N)T\big)x\Big).\nn\ee 
\qed\end{thm}

Moreover it follows from our recursive definition of $\Flag^m(\cca)$ above that the canonical element in $\Fl^m(\cca)\otimes \A^m(\cca)$ itself has a simple recursive definition. For $p=0,\dots,m-1$ we define a linear map
\be \theta_{m-p} :    \Fl^p(\cca) \otimes \A^p(\cca) 
           \to \Fl^{p+1}(\cca) \otimes \A^{p+1}(\cca) \nn\ee
by
\begin{multline} \theta_{m-p}( (L^0 \supset \dots \supset L^p) \otimes \lambda ) \\
  = \sum_{j=1}^{m-p-1} \sum_{k\in\ZT} (L^0 \supset \dots \supset L^p \supset L^p \cap \Ht j{m-p}k) \otimes (\lambda \wx \Ht j{m-p}k)  \\
+ (L^0 \supset \dots \supset L^p \supset L^p \cap \Ho {m-p}) \otimes(\lambda\wx \Ho {m-p}) \\
 + \sum_{j=1}^{N} \sum_{k\in\ZT} (L^0 \supset \dots \supset L^p \supset L^p \cap \Hz {m-p}jk) \otimes (\lambda\wx \Hz {m-p}jk ) \end{multline}
Define $\Theta := \theta_{1}(\theta_{2}(\dots \theta_{m}( L^0 \otimes 1 ) \dots ))$. In view of Corollary \ref{cancor} we have the following. 
\begin{prop}\label{Thetaprop} $\Theta$ is the canonical element of $\Fl^m(\cca) \otimes \A^m(\cca)$.\qed
\end{prop}

By Theorem \ref{p1}, we have an isomorphism $\psis_m^{-1}$ from $\Fl^m(\cca)$ to the subspace $(\Mb)_{\mones}$ of $\Mb= U(\nbs) \ox \bigotimes_{i=1}^N U(\nb)$. Hence we get a canonical element 
\be (\psis_m^{-1} \ox \id)(\Theta)  \in (\Mb)_{\mones}\ox \A^m(\cca),\label{Psidef}\ee
where we can regard $\A^m(\cca)$ as the space of logarithmic differential $m$-forms of the arrangement $\cca$, as in \S\ref{osdf}. 

\begin{exmp}\label{fzex}  Consider the case $N=0$ -- \ie no non-zero marked points $z_i$ -- and $m=3$. We have logarithmic differential forms $\iota(\Ho i)= dt_i/t_i$ and  $\iota(\Ht ijk)= d(t_i-\omega^k t_j)/(t_i-\omega^k t_j)$. The element $(\psis_m^{-1} \ox \id)(\Theta)$ is given by
\begin{multline}
(\psis_m^{-1} \ox \id)(\Theta) =\Pin f_1 \Pin f_2 \Pin f_3
\ox \frac {dt_3}{t_3} \wx \frac {dt_2}{t_2} \wx \frac {dt_1}{t_1} \\
+\sum_{q\in \ZT} \Pin{[\f q 2 ,f_1]}\,\,  \Pin f_3 
\ox \frac {dt_3}{t_3} \wx \frac {d(\omega^qt_2-t_1)}{\omega^qt_2-t_1} \wx \frac {dt_1}{t_1} 
\\
+\sum_{p\in \ZT} \Pin{[\f p 3 ,f_1]}\,\, \Pin f_2 
\ox \frac {d(\omega^pt_3-t_1)}{\omega^pt_3-t_1} \wx \frac {dt_2}{t_2}  \wx \frac {dt_1}{t_1} \\
+\sum_{p\in \ZT} \Pin{f_1}\,\,\Pin{[\f p 3,  f_2]} 
\ox \frac {d(\omega^pt_3-t_2)}{\omega^pt_3-t_2} \wx \frac {dt_2}{t_2} \wx \frac {dt_1}{t_1} 
\\
+\sum_{p,q\in \ZT} \Pin{[\f q 2 ,[\f p 3,  f_1]]} 
\ox \frac {d(\omega^pt_3-t_1)}{\omega^pt_3-t_1} \wx \frac {d(\omega^qt_2-t_1)}{\omega^qt_2-t_1}\wx\frac{dt_1}{t_1}    \\
+
\sum_{p,q\in \ZT} \Pin{[[\f{p+q}3 ,\f q 2],  f_1]}
\ox \frac {d(\omega^pt_3-t_2)}{\omega^pt_3-t_2}\wx \frac {d(\omega^qt_2-t_1)}{\omega^qt_2-t_1}\wx  \frac {dt_1}{t_1}  .
\end{multline}
\end{exmp}

Combining Theorem \ref{p1} and Theorem \ref{bthm2}(\ref{part1}) we obtain a basis,
\be \psis_m^{-1} (\Flag^m(\cca,\oo)),\label{Mbasis}\ee
of the subspace  $(\Mb)_{\mones}$ of $\Mb= U(\nbs) \ox \bigotimes_{i=1}^N U(\nb)$. 

\begin{cor}\label{dimcor} We have $\dim (\Mb)_{\mones} = \dim \Fl^m(\cca) = \prod_{k=1}^{m-1} (1+ (k+N)T)$.
In particular, 
$\dim U(\nbs)_{\mones} = \dim \Fl^m(\cc_{0;m}) = \prod_{k=1}^{m-1} (1+ kT)$.
\end{cor}
\begin{proof}
Follows from Theorem \ref{Pointhm}.
\end{proof}

\subsection{The case $T=1$ and Lyndon words}\label{lyndon}
Let us consider in passing the special case $T=1$, \ie the non-cyclotomic case.
Then $\nbs = \nb$ and we have the weight $[1,\dots,1]$ component $U(\nb)_{\mones}$ of the universal envelope $U(\nb)$ of the free Lie algebra $\nb$ in generators $f_1,\dots, f_m$.
Corollary \ref{dimcor} reduces to the statement that
\be \dim U(\nb)_{\mones} = (1+1)(1+2) \dots (1+(m-1)) = m!.\nn\ee 
An obvious basis of $U(\nb)_{\mones}$ is given by monomials $f_{s(1)}...f_{s(m)}$ labelled by permutations $s\in \Sigma_m$.

Let us describe explicitly the basis of $U(\nb)_{\mones}$ constructed in \eqref{Mbasis}.

Recall that given any ordered alphabet there is the reverse\footnote{To match our conventions above it is convenient to use reverse lexicographical ordering.} lexicographical total ordering on the set of words in that alphabet: if $w_1=a_k \dots a_3a_2a_1$ and $w_2 = b_\ell  \dots b_3b_2 b_1$ are distinct words then $w_1<w_2$ if and only if $a_i<b_i$ for the smallest $i$ such that $a_i\neq b_i$.

A word $w$ is called a \emph{Lyndon word} if $w< vu$ for every splitting $w=uv$ of $w$ into (non-empty) words $u$ and $v$.
Note that if $a_k\dots a_3a_2a_1$ has no repeated letters, then it is a Lyndon word if and only if $a_1$ is  least among its letters.

Free Lie algebras are known to admit bases labelled by Lyndon words \cite{CFL,Shi,Hall} (see e.g. \cite{BCL} for recent work with extensive bibliography). Indeed, given an ordered alphabet $A$, let $W(A)$ denote the set of Lyndon words in $A$ and $L(A)$ the free complex Lie algebra in generators $(f_a)_{a\in A}$. Define a map $\gamma:W(A)\to L(A)$  recursively as follows. For each letter $a\in A$, set $\gamma(a) := f_a\in L(A)$. For any Lyndon word $w$ of length $>1$, there is is a unique way to write $w=uv$ for (non-empty) words Lyndon words $u$ and $v$ such that $u$ is as long\footnote{Again, note our convention here.}  as possible.
Set $\gamma(w) := [\gamma(u),\gamma(v)]$. Then the image $\gamma(W(A))$ is a basis of $L(A)$. The reverse lexicographical ordering among Lyndon words makes this into an ordered basis, and so we can apply the PBW theorem to get a basis of the envelope $U(L(A))$. 

In our case, let $\P$ denote the set of all $m!$ words in the ordered alphabet $1,2,\dots,m$ in which each letter is used exactly once.
Every word $w \in \P$ can be uniquely factored into a product of Lyndon words such that these Lyndon words are ordered (reverse lexicographically) when read from left to right. (The factors are, in an obvious sense, the maximal Lyndon words occurring in $w$.)  
Applying the map $\gamma$ factor by factor, we arrive an element of $U(\nb)_{\mones}$. Doing this for each word $w\in \P$ in turn, we produce a basis of $U(\nb)_{\mones}$.
\begin{prop} This basis of $U(\nb)_{\mones}$ coincides with the basis \eqref{Mbasis}, for $N=0$ and $T=1$.
\end{prop} 
\begin{proof} By inspection. \end{proof}
\begin{exmp} When $m=5$ we have for example the following factorizations of words in $\P$ into reverse-lexicographically ordered products of Lyndon words 
\be 54321 = (54321),\quad 41235 = (41)(2)(3)(5),\quad 24531 = (24531),\quad 12345 = (1)(2)(3)(4)(5).\nn\ee
Applying $\gamma$ we get the following elements of $U(\nb)_{\mones}$, which indeed belong to our basis above:
\be \bigg[\Big[\big[[f_5,f_4],f_3\big],f_2\Big],f_1\bigg], \quad [f_4,f_1]f_2f_3f_5, \quad \bigg[f_2,\Big[\big[f_4,[f_5,f_3]\big],f_1\Big]\bigg], \quad f_1f_2f_3f_4f_5. \nn\ee 
\end{exmp}

\begin{rem}
When $T>1$, the basis \eqref{Mbasis} of $U(\nbs)_{\mones}$  consists of ``decorated'' Lyndon words. For example, the element $[f_2,f_1]f_3$ of the Lyndon basis for $m=3$, $T=1$ corresponds to $T$ elements $\Pin{[\f q 2 ,f_1]}\,\,  \Pin f_3$ labelled by $q\in\Z_T$  if $T>1$.
\end{rem}

\subsection{Symmetrized canonical element and the weight function}
Let now $\lambda = (\lambda_1,\dots,\lambda_r) \in \Z_{\geq 0}^\rf$ with $m= \lambda_1+ \dots +\lambda_\rf$. As in \S\ref{symflags}, the finite group
\be \Sw_\lambda = \Sw_{\lambda_1} \times \dots \times \Sw_{\lambda_{\rf}}\nn\ee 
acts linearly on the flag spaces $\Fl^p(\cca)$ and their duals, $\A^p(\cca)\cong \Fl^p(\cca)^*$.

Let $\Fl^m(\cca) = \bigoplus_d W_d$ be the canonical decomposition of the space of full flags into its isotypical components with respect to this action.  Let $\A^m(\cca) = \bigoplus_d W^*_d$ be the corresponding decomposition of its dual. Then the canonical element $\Theta$ decomposes as
$\Theta = \bigoplus_d \Theta_d$ where $\Theta_d$ is the canonical element of $W_d \ox W_d^*$.
In particular, let 
\be \Theta^\Swl\in  \Fl^m(\cca)^\Swl \ox \A^m(\cca)^\Swl.\nn\ee
where $\Fl^{m}(\cca)^\Swl$ denotes the isotypical component of $\Fl^m(\cca)$ spanned by flags $F$  such that $(-1)^{|g|}g\on F =  F$, as in \S\ref{symflags}, and where  $\A^{m}(\cca)^\Swl$ denotes the isotypical component of $\A^m(\cca)$ spanned by forms $\omega$  such that $(-1)^{|g|}g\on \omega = \omega$. Observe that $\Theta^\Swl$ can be obtained from $\Theta$ by projecting either tensor factor. 


Let $\chi_i : \Sw_{\lambda_i} \to \Cx$ be the one-dimensional representation of $\Sw_{\lambda_i}$ given by $\chi( (s, (k_1,\dots,k_{\lambda_i})) ) = \omega^{(k_1+\dots+k_{\lambda_i})T_i}$. Then $\chi:= \chi_1\dots\chi_\rf$ is a one-dimensional representation of $\Swl$.
For later use, observe that the Jacobian of the linear map $\C^m \to \C^m; p \mapsto g^{-1}\on p$ is 
\be J_g 
     := \det\left(\frac{\del t(g^{-1} \on p)_i}{\del t_j(p)}\right)_{1\leq i,j\leq m}
      = \det\left(\frac{\del (g\on t)_i(p)}{\del t_j(p)}\right)_{1\leq i,j\leq m}
      = \chi(g) (-1)^{|g|} \label{Jac}
\ee
where $(-1)^{|g|}$ is as in \eqref{modgdef}.

An element $\eta \in \A^m(\cca)$ is a top-degree form, so it can be written uniquely as
\be \eta = u \,dt_1 \wx \dots \wx dt_m \nn\ee 
for some rational function $u$ regular on $\C^m \setminus\bigcup_{H\in \cca} H$. 
We have
\be (-1)^{|g|} g\on\eta =  (g\on u)  \chi(g) \, dt_1 \wx \dots \wx dt_m.\nn\ee

\begin{exmp} 
Suppose $\rf=1$, $\lambda_1=m=2$ and $T=1$. Then $\Swl = \Sigma_2$. Suppose 
$\eta = d\log(t_1-z) \wx d\log(t_2-z) = \frac{1}{t_1-z}\frac{1}{t_2-z} dt_1\wx dt_2$.
By definition
$(-1)^{|(12)|}(12).\eta 
= - d\log(t_2-z) \wx d\log(t_1-z) = \eta$
and indeed $\chi= 1$ and $u(t) = \frac{1}{t_1-z_1}\frac{1}{t_2-z_1}$ is invariant under $t_1\leftrightarrow t_2$. 
\end{exmp}
\begin{exmp} Suppose $\rf=1$, $\lambda_1=m=1$ and $T=3$, $T_1=1$. Then $\Swl = \Z_3$. Suppose $\eta = d\log(t_1-z) = \frac1{t_1-z} dt_1$. Then $k\on \eta = d\log(\omega^k t_1-z) = \frac1{\omega^k t_1- z} \omega^k dt_1$.
\end{exmp}

It follows that $\eta\in  \A^m(\cca)^\Swl$ if and only if $u= \chi(g) g\on u$ for all $g\in \Swl$. That is, if and only if  $u(p) = u(g^{-1}\on p)\chi(g)$ for all $g\in \Swl$. Hence the projector $\A^m(\cca)\to  \A^m(\cca)^\Swl$ can be written as
\be u \,dt_1 \wx \dots \wx dt_m  \mapsto \frac 1{|\Swl|}\sum_{g\in \Swl} u( g^{-1} \on p) \chi(g) dt_1 \wx \dots \wx dt_m .\nn\ee 

In particular, the canonical element $\Theta$ is of the form
\be \Theta = \theta\, dx_1\wx \dots dt_m, \nn\ee
for a unique rational map $\theta:\C^m \to  \Fl^m(\cca)$, regular on $\C^m \setminus\bigcup_{H\in \cca} H$. Recall that to obtain the component $\Theta^\Swl\in \Fl^m(\cca)^\Swl \ox \A^m(\cca)^\Swl$ from $\Theta\in \Fl^m(\cca)\ox\A^m(\cca)$ it is enough to apply the projector  $\A^m(\cca)\to  \A^m(\cca)^\Swl$ to the second factor. Thus $\Theta^\Swl$ is given by
\be \Theta^\Swl = \theta^\Swl\,  dx_1\wx \dots dt_m \nn\ee
where $\theta^\Swl:\C^m \to  \Fl^m(\cca)^\Swl$ is the rational map, regular on $\C^m \setminus\bigcup_{H\in \cca} H$, given by
\be\theta^\Swl = \frac 1{|\Swl|}\sum_{g\in \Swl} \theta( g^{-1} \on p) \chi(g) .\label{tswl}\ee

\begin{defn} The \emph{weight function} $\Psi$ is the rational map $\C^m \to L(\bm \Lambda)_\lambda$ given by 
\be \Psi : \C^m \xrightarrow{\theta^\Swl} \Fl^m(\cca)^\Swl \xrightarrow\sim M(\bm\Lambda)_\lambda  \to  L(\bm \Lambda)_\lambda \nn\ee
where the isomorphism $M(\bm\Lambda)_\lambda \cong \Fl^m(\cca)^\Swl$ is that of Theorem \ref{symthm}. 
\end{defn}
Equivalently (given Theorem \ref{mt} and Corollary \ref{irrepcor}) 
\be \Psi : \C^m \xrightarrow{\theta^\Swl} \Fl^m(\cca)^\Swl \xrightarrow\Ghom  \Ghom(\Fl^m(\cca)^\Swl) \xrightarrow\sim L(\bm \Lambda)_\lambda \nn\ee

\begin{exmp} Let $B$ be the Cartan matrix of type $\mathrm A_3$  and $\siod$ the diagram automorphism, $\siod = (13)$. Let $\lambda = (1,0)$. Let the number of non-zero marked points be $N=1$. The weight function is
\be \Psi = \frac{ v_0 \ox F_1 v_1}{t_1 - z} + \frac{ v_0 \ox F_3 v_1}{t_1 + z} + \frac{(F_1+F_3)v_0 \ox v_1}{ t_1 } \in L(\bm \Lambda).\nn\ee
\end{exmp}

The weight function is regular at least on $\C^m \setminus\bigcup_{H\in \cca} H$. The following proposition establishes that it is also regular on any hyperplane $H\in \cca$ whose weight is zero, $a(H)=0$. 
\begin{prop}$ $\label{wfrprop}
\begin{enumerate}
\item  Suppose $(\alpha_{\iota(i)}, \sigma^k \alpha_{\iota(j)})  =0$. Then the weight function $\Psi$ is regular on the hyperplanes $\Ht {\clc in}{\clc jp}k$, $n=1,\dots,\lambda_i$, $p=1,\dots,\lambda_j$. 
\item Let $v\in \{1,\dots,N\}$ (resp. $v=0$). If $(\alpha_{\iota(i)}, \sigma^k\Lambda_v) = 0$ then the weight function $\Psi$ is regular on the hyperplanes $\Hz {\clc in} v k$ (resp. the hyperplane $\Ho {\clc in}$) for $n=1,\dots,\lambda_i$.  
\end{enumerate}
\end{prop}
\begin{proof} The proof is given in Appendix \ref{wfr}. \end{proof}

Recall the definition of the master function $\Phi$ from \S\ref{sg}. In the case of the cyclotomic discriminantal arrangement $\cca$ with the weighting defined in \eqref{HdC} and \eqref{rhw}, $\Phi$ is given explicitly by
\begin{multline}\label{MF}
\cmfn =
-\sum_{i=1}^\rf\sum_{n=1}^{\lambda_i} \big( \alpha_{i}, \Lambda_0 \big)   \log t_{\clc in}
- \sum_{k\in \ZT} \sum_{i=1}^N\sum_{j=1}^\rf\sum_{n=1}^{\lambda_j}  
  \big( \alpha_{j},\lsigma^k\Lambda_i\big) \log  (t_{\clc jn}-\omega^kz_i) \\ 
+ \sum_{k\in \ZT}\sum_{1\leq i<j \leq \rf} \sum_{n=1}^{\lambda_i} \sum_{p=1}^{\lambda_j}\big( \alpha_{i},\lsigma^k\alpha_j\big)\log (t_{\clc in}-\omega^kt_{\clc jp})  
+ \sum_{k\in \ZT}\sum_{i=1}^\rf \sum_{1\leq n<p\leq \lambda_i} \big( \alpha_{i},\lsigma^k\alpha_i\big)\log (t_{\clc in}-\omega^kt_{\clc ip}) . 
\end{multline} 

Given a point  $p\in \C^m$, let $\Stab p\subset \Swl$ denote the stabilizer subgroup of $p$. Define a number $\eta(p)\in \C$ by
\be \eta(p) := \sum_{h\in \Stab p} \chi(h). \nn\ee
Whenever $p\in  \C^m \setminus\bigcup_{H\in \cca}$ then $\Stab p$ is trivial and hence $\eta(p)=1$.

\newcommand{\bba}{\mc B_{0,N;m}}
Let now $\bba\subset \cca$ be the subarrangement consisting of all the hyperplanes of $\cca$ of nonzero weight. 
Proposition \ref{wfrprop} shows that $\Psi$ is regular on the complement $\C^m \setminus\bigcup_{H\in \bba} H$ of this subarrangement. 
A point $p\in \C^m \setminus\bigcup_{H\in \bba} H$ in this complement is called a \emph{critical point} if $d\Phi|_p = 0$. 
The \emph{Bethe vector} corresponding to a critical point $p$ is the vector $\Psi(p) \in L(\bm\Lambda)_\lambda$. 

\begin{thm} \label{bvthm} 
Suppose $p\in \C^m \setminus\bigcup_{H\in \bba} H$ is an isolated critical point of the master function $\Phi$. The Bethe vector $\Psi(p)$ is nonzero if and only if $\eta(p)\neq 0$. 

In particular the Bethe vector $\Psi(p)$ is nonzero whenever the stabilizer $\Stab p\subset \Swl$ is trivial. 
\end{thm}
\begin{proof} By Lemma \ref{funglem}, we may work with the geometric form of the subarrangement $\bba$.

We use results from \cite{V11}.
Consider first the rational map 
\be \widetilde\Psi : \C^m \xrightarrow{\theta} \Fl^m(\cca) \twoheadrightarrow \Ghom(\Fl^m(\cca)) \subset \A^m(\cca).\nn\ee
Suppose $p$ is an isolated critical point of $\Phi$. It follows from Theorem 9.13 in \cite{V11} that $\widetilde\Psi(p)$ (denoted there $[E](p)$) is nonzero. 

Indeed, in \cite{V11} a linear map $\alpha_p: A_p \to \Fl^m(\cca)/\ker \Ghom \cong \Ghom(\Fl^m(\cca))$
is defined, where $A_p$ is the \emph{local algebra} of the critical point $p$. The local algebra is the quotient of the algebra of germs at $p$ of holomorphic functions by the ideal generated by the partial derivatives $(\del \Phi/\del t_i)_{1\leq i\leq m}$. It comes equipped with a canonically defined non-degenerate bilinear form, the Grothendieck form. Theorem 9.13 of \cite{V11} says that $\alpha_p$ is an injection and that under this injection the Grothendieck form on $A_p$ is identified with the geometric form $\Geom$ on $\Fl^m(\cca)/\ker \Ghom$. There is a preferred element of $A_p$, the Hessian determinant 
\be \Hess := \det(\del^2\Phi/ \del t_i\del t_j)_{1\leq i,j\leq m}.\nn\ee 
It is known that $\Hess \neq 0$ in $A_p$. (Let us stress that this is true even if $p$ is a degenerate critical point, \ie even if $\Hess|_p=0$.)
 It is shown in \cite{V11} that the image of $\Hess$ under the injection $\alpha_p$ is $\widetilde\Psi(p)$, up to a nonzero multiplicative constant.  

Now consider the $\Swl$-orbit of $p$. 

Let us first consider the case that $p\in \C^m \setminus\bigcup_{H\in \cca} H$. Then the orbit consists of $|\Swl|$ isolated critical points, and $\Swl$ acts freely on it. To each point $g\on p$ is associated its local algebra $A_{g\on p}$. Theorem 9.14 from \cite{V11} says that the images of the corresponding injections $\alpha_{g\on p} : A_{g\on p} \to \Fl^m(\cca)/\ker \Ghom$ are orthogonal with respect to the (nondegenerate, on this quotient) bilinear form $\Geom$. Hence a linear combination of the form $\sum_{g\in \Swl} c(g) \widetilde\Psi(g^{-1}\on p)$ is non-zero whenever any of the coefficients $c(g)$ are nonzero. In particular the sum
\be \frac{1}{|\Swl|} \sum_{g\in \Swl} \chi(g) \widetilde\Psi( g^{-1} \on p) \label{swlchi}\ee
is nonzero. In view of \eqref{tswl}, it follows that the Bethe vector $\Psi(p)$ is non-zero.

In the case when $p$ lies on a hyperplane $H\in \cca$ of zero weight, the stabilizer subgroup $\Stab p\subset \Swl$ may be non-trivial. The sum \eqref{swlchi} is nonzero if and only if $\sum_{g\in \Stab p} \chi(g) = \eta(p)$ is nonzero. 
\end{proof}

For critical points that are not only isolated but \emph{non-degenerate}, \ie such that $\Hess|_p \neq 0$, we have also the following.

\begin{prop}
Suppose $p\in \C^m \setminus\bigcup_{H\in \bba} H$ is a critical point of the master function $\Phi$. Then the norm of the Bethe vector is given by $S(\Psi(p), \Psi(p)) = \frac{\eta(p)}{|\Swl|}\Hess|_p$. 

In particular, if $p$ is a non-degenerate critical point with $\eta(p)\neq 0$ then the Bethe vector $\Psi(p)$ has nonzero norm.
\end{prop}
\begin{proof}
First observe that, since $(g\on t_i)(g\on p) = t_i(p)$ for all $g\in \Swl$, 
\begin{align} \Hess|_{g^{-1}\on p} &=  \left.\det\left(\frac{\del^2\Phi}{ \del t_i\del t_j}\right)_{1\leq i,j\leq m}\right|_{g^{-1}\on p} \nn\\
 &=  \left.\det\left(\frac{\del^2\Phi}{ \del(g^{-1}\on t_i)\del(g^{-1}\on t_j)}\right)_{1\leq i,j\leq m}\right|_{g^{-1}\on p} \det\left(\frac{\del(g^{-1}\on t_i)}{\del t_j}\right)_{1\leq i,j\leq m}^2\nn\\
 &= \left.\det\left(\frac{\del^2\Phi}{ \del t_i\del t_j}\right)_{1\leq i,j\leq m}\right|_{p} J_{g^{-1}}^2
 = \Hess|_{p} \chi(g)^{-2}, \label{Hge}\end{align}
cf. \eqref{Jac}. 
Let $\widetilde \Psi$ be as in the previous proof. It is shown in \cite[Theorem 5.2]{V06} that for any non-degenerate critical point $p$, $S(\widetilde\Psi(p), \widetilde\Psi(p)) = \Hess|_p$, and moreover that if $p,q$ are distinct non-degenerate critical points then $S(\widetilde\Psi(p), \widetilde\Psi(q)) = 0$. 
Thus, given \eqref{tswl} and \eqref{Hge}, we have
\begin{align}
S(\Psi(p),\Psi(p)) &=  \frac{1}{|\Swl|^2} \sum_{a,b\in \Swl} \chi(a)\chi(b) S(\widetilde\Psi( a^{-1} \on p), \widetilde\Psi( b^{-1} \on p) )\nn\\
 &=  \frac{1}{|\Swl|^2} \sum_{a\in\Swl}\sum_{b\in a\Stab p} \chi(a)\chi(b) S(\widetilde\Psi( a^{-1} \on p), \widetilde\Psi( a^{-1} \on p) )\nn\\
&=  \frac{1}{|\Swl|} \left(\sum_{h\in \Stab p}\chi(h)\right) \Hess|_{p} 
\end{align}
as required. 
\end{proof}

\subsection{On the Bethe vectors of the cyclotomic Gaudin model}\label{gaudin}
In \cite{VY1} a \emph{cyclotomic Gaudin model} was constructed and solved by Bethe ansatz. The data defining the model include a triple $(\g,\sigma,T)$ where $\g$ is a finite-dimensional semisimple Lie algebra and $\sigma$ is an automorphism of $\g$ whose order divides $T\in \Z_{\geq 0}$. 

Consider the special case when  $\sigma$ is a diagram automorphism of order $T$. Then we are in the setting of the present paper. Choose the matrix $B = (\alpha_i,\alpha_j)_{i,j=1,\dots,\ru}$ above to be the symmetrized Cartan matrix of $\g$. 

\begin{prop}  The image in $L(\bm \Lambda)$ of the cyclotomic weight function $\psi_\Gamma\in M(\bm \Lambda)$ of \cite{VY1} coincides with $\Psi$ as defined here. 
\qed\end{prop} 
This is so because the recursive definition of the canonical element $\Theta$ given after Theorem \ref{bthm2} coincides (by construction)  with the recursive definition of the weight function via ``swapping'' given in \cite{VY1} following \cite{FFR}. 

Consequently, the results of the preceding subsection apply to the weight function of \cite{VY1}. 

We should remark on the following subtlety. In the present paper the highest weight $\Lambda_0\in (\hs)^*$ of the module $M^\si(\Lambda_0)$  determines the weighting of the hyperplanes $\Ho i$ defined by the equations $(t_{i}=0)$ simply by 
\be a(\Ho {\clc in}) = \big(\alpha_i, \Lambda_0\big).\label{ad1}\ee
In \cite{VY1} the relationship is in general less trivial. Namely, 
\be a(\Ho {\clc in}) = \big( \alpha_i , \Lambda_0+ \lambda_0 \big) -   
  \sum_{k=1}^{T-1} \frac{\big( \alpha_i, \sigma^k \alpha_i \big) }{1-\omega^k} \label{ad2}\ee
where $\lambda_0\in (\hs)^*$ is a weight defined as follows.\footnote{Strictly speaking, \cite{VY1} considered only the case $\Lambda_0=0$, but the generalization is straightforward.} Let $\Delta^+$ denote the set of positive roots of $\g$. The Cartan decomposition can be chosen such that, for each $\alpha\in \Delta^+$, we have $\sigma E_{\alpha} = \tau_{\alpha} E_{\siod(\alpha)}$ and $\sigma F_\alpha = \tau_{\alpha}^{-1} F_{\siod(\alpha)}$ for some root of unity $\tau_\alpha$. Let
\be\lambda_0  :=  \sum_{k=1}^{T-1} \frac{1}{1 - \omega^k} \sum_{\substack{\alpha\in \Delta^+\\\siod^k(\alpha)=\alpha}} \left( \prod_{p=0}^{k-1} \tau_{\siod^p(\alpha)}^{-1} \right) \alpha.
\nn\ee
\begin{lem}
The definitions \eqref{ad1} and \eqref{ad2} do coincide whenever $\siod$ is a diagram automorphism.
\end{lem}
\begin{proof}
If $T=1$ there is nothing to prove. If $T>1$ but $\siod=\id$ is the trivial automorphism then $\tau_\alpha=1$ for all $\alpha\in \Delta^+$ and the result follows from the identity $\big(\alpha_i,\alpha_i\big) = 2 \big(\alpha_i, \rho\big)$ where $\rho = \half\sum_{\alpha\in \Delta^+}\alpha$ is the Weyl vector. So suppose $\siod$ is a non-trivial diagram automorphism and $T$ is a multiple of the order of $\siod$. 
To say that $\siod$ is a diagram automorphism is to say that $\tau_{\alpha_i}=1$ for each simple root $\alpha_i$. 
Then $\tau_\alpha\in \{+1,-1\}$ for all $\alpha\in \Delta^+$. Now we proceed by case-by-case inspection of the non-trivial diagram automorphisms of the finite-type Dynkin diagrams.  
Define a sign $\eps$ to be $-1$ in type $A_{2n}$ and $+1$ otherwise. 
One finds that $\tau_\alpha = \eps$ for all $\alpha\in\Delta^+$ such that $\siod(\alpha) = \alpha$. Moreover for each simple root $\alpha_i$,
\be \big(\alpha_i, \siod \alpha_i\big) =  \eps \sum_{\substack{\alpha\in \Delta^+\\\siod(\alpha)=\alpha}}
\big( \alpha_i,  \alpha\big).\nn\ee
This implies the result.
\end{proof}

\appendix

\section{Proof that $\Psi$  is regular on hyperplanes of weight zero.}\label{wfr}
In this appendix we give the proof of Proposition \ref{wfrprop}.
Consider part (i). 
Recall our basis $\Flag^m(\cca,\oo)$ and its dual from Theorem \ref{bthm2}. The canonical element is the sum 
\be \Theta = \sum_{F\in \Flag^m(\cca,\oo)} F \otimes \omega_F,\nn\ee where $\omega_F\in \A^m(\cca)$ is the basis vector dual to the flag $F$, for each flag $F\in \Flag^m(\cca,\oo)$.  
Let $u_F$ be the rational function defined by $\omega_F = u_F dt_1 \wx \dots \wx dt_m$. 

Fix $i,j,k$ such that $(\alpha_{\iota(i)}, \sigma^k \alpha_{\iota(j)})  =0$. Fix $n\in \{1,\dots,\lambda_i\}$ and $p\in \{1,\dots,\lambda_j\}$. 
We must show that
\be \sum_{F\in \Flag^m(\cca,\oo)} \Ghom(F) u_{F} \nn\ee
is regular on $\HH$. 

There is an equivalence relation on the set $\Flag^m(\cca, \oo)$ of basis flags in which two flags are equivalent if and only if ``they are the same when we equate $t_{\clc in}$ and $\omega^k t_{\clc jp}$''. That is,  $F = (L^0 \supset \dots \supset L^m)$ and $\tilde F = (\tilde L^0 \supset \dots \supset \tilde L^m)$ are equivalent if $L^r \cap \HH = \tilde L^r \cap \HH$ for each $r$. 
Let $[F]\subset \Flag^m(\cca,\oo)$ denote the equivalence class of a flag $F$.
We shall now show that, for every basis flag $F$,
\be \sum_{\tilde F\in [F]} \Ghom(\tilde F) u_{\tilde F} \label{tfsum}\ee
is regular on $\HH$. Since the classes partition $\Flag^m(\cca,\oo)$, this suffices to prove the proposition.


Suppose that $F$ is any basis flag such that the function $u_F$ is singular on $\HH$. This happens if and only if $u_F$ contains a factor $(t_{\clc in} - \omega^k t_{\clc jp})^{-1}$. In view of Theorem \ref{bthm2} part (\ref{part2}), that means $H(L^r) = \HH$ for some edge $L^r$, $r\in \{1,\dots,m\}$, of the flag $F$. 

Observe that the difference $u_F-u_{\tilde F}$ is regular on $\HH$ for all $\tilde F\in [F]$. (Indeed, $u_F = (t_{\clc in} - \omega^k t_{\clc jp})^{-1} f$ and $u_{\tilde F} = (t_{\clc in} - \omega^k t_{\clc jp})^{-1} \tilde f$ for rational functions $f,\tilde f$ regular \emph{and equal} on $\HH$.) 
So for the purposes of computing the singular part of \eqref{tfsum} we can replace $u_{\tilde F}$ by $u_F$ in the sum. Therefore it is enough to show that
\be \sum_{\tilde F\in [F]} \tilde F  \in \ker \Ghom. \nn\ee

Recall that in our recursive definition of the basis flags, we first ``set $t_m$ equal to something'', then $t_{m-1}$, then $t_{m-2}$, and so on, until at the $r$th step we get to $t_{\max( \clc in, \clc jp)}$. Thus $r= m+1-\max(\clc in, \clc jp)$. There is a unique nonnegative integer $s<r$ and unique integers $h_1,\dots,h_s$ with  $m\geq h_1 > \dots > h_s > \max(\clc in,\clc jp)$ such that: for every flag $\tilde F = (\tilde L^0 \supset \dots \supset \tilde L^m)$ in the class $[F]$, each of the hyperplanes $H(\tilde L^{m+1-h_u})$, $u=1,\dots,s$, involves one or other of $\clc in$ or $\clc jp$, while the remainder of the first $r-1$ hyperplanes $H(\tilde L^1), \dots, H(\tilde L^{r-1})$ do not. (Intuitively speaking, $t_{h_1},\dots,t_{h_s}$ are the coordinates which  got equated to a multiple of $t_{\clc in}$ or $t_{\clc jp}$ before either of the latter got set equal to something.)

Note that for all $\tilde F =(\tilde L^0 \supset \dots \supset \tilde L^m)\in [F]$, $L^{s} = \tilde L^{s}$ for $r\leq s \leq m$. Consider the edge $L^{r}$ of $F$. It is an archipelago of islands in the sense of \S\ref{efs}. One of these islands is a swimming island involving precisely the elements of $J= \{h_1,\dots,h_s\} \cup \{\clc in, \clc jp\}$. 
Recall the notion of shortened flags from \S\ref{sandl}. By inspection, it is enough to consider the shortened flags $\tilde F|_J$ and to show that
\be \sum_{\tilde F\in [F]} \tilde F|_J \in \ker \Ghom .\label{rfs}\ee

Observe that $a(\HH) = (\alpha_{\iota(i)}, \sigma^k \alpha_{\iota(j)})  =0$ implies $\psis(\Pin{[\f k {\clc jp}, f_{\clc in}]}) \in \ker\Ghom$ and hence
\be C := \psis(\Pin{[ \f{q_s}{h_s},[\f{q_{s-1}}{h_{s-1}},\dots,[\f{q_1}{h_1},  [\f{k}{\clc jp}, f_{\clc in}] \dots ]]]}) \in \ker \Ghom.\nn\ee
Now we claim that (for suitable $q_1,\dots,q_s\in \ZT$), $C$  \emph{is} the sum of flags in \eqref{rfs}. Indeed, consider the following procedure. Here we simplify notation by letting $u:=\f{q_u}{h_u}$ and $i:= f_{\clc in}$, $j:= \f{k}{\clc jp}$.
First distribute $1$ over $[j,i]$ using the Jacobi identity in the form
\be [1,[j,i]] = [[1,j],i] + [j,[1,i]].\nn\ee
Then distribute $2$ over the result:
\begin{align} \big[2,[1,[j,i]\big] & = \Big[2, [[1,j],i] + [j,[1,i]] \Big] \nn\\
 & = \big[[2,[1,j]],i\big] + \big[[1,j],[2,i]\big]+ \big[[2,j],[1,i]\big]  + \big[j,[2,[1,i]]\big] ,\nn\end{align}
and so on.
In general, at the $u$th step we have a sum of terms of the form $[s,\dots,[u+1,[u,[X,Y]]\dots ]]$ for some commutators $X,Y$ and we re-write such terms as follows:
\be \Big[s,\dots,\big[u+1,[u,[X,Y]]\big]\dots \Big] = \Big[s,\dots,\big[u+1, [[u,X],Y]+[X,[u,Y]] \big]\dots \Big].\nn\ee 
With the isomorphisms of Theorem \ref{p1} in mind, one sees that this indeed coincides with the recursive procedure given above to construct the basis flags. Thus we have the equality \eqref{rfs}, as required.

\newcommand{\HHo}{\Ho{\clc in}}
We turn now to part (ii).
The argument is similar to the one above. We shall consider for definiteness the case of $(\alpha_{\iota(i)},\Lambda_0) = 0$. There is an equivalence relation among basis flags in which $F = (L^0 \supset \dots \supset L^m)$ and $\tilde F = (\tilde L^0 \supset \dots \supset \tilde L^m)$ are equivalent if $L^r \cap \HHo = \tilde L^r \cap \HHo$ for every $r$. As before one checks that $u_F - u_{\tilde F}$ is regular on $\HHo$ for all $\tilde F\in [F]$ so that it is enough to show that for each equivalence class $[F]$, $\sum_{\tilde F\in [F]} \tilde F  \in \ker \Ghom$. 

Pick a basis flag $F=(L^0 \supset \dots \supset L^m)$ such that $u_F$ is singular. Let $r$ be the unique number such that $H(L^r) = \HHo$. There is a unique $s<r$ and unique integers $h_1,\dots,h_s$ with  $m\geq h_1 > \dots > h_s > \clc in$ such that: for every flag $\tilde F = (\tilde L^0 \supset \dots \supset \tilde L^m)$ in the class $[F]$, each of the hyperplanes $H(\tilde L^{m+1-h_u})$, $u=1,\dots,s$, either involves $\clc in$ or else is $\Ho{h_u}$, while the remainder of the first $r-1$ hyperplanes $H(\tilde L^1), \dots, H(\tilde L^{r-1})$ are not of either of these types. Let $J= \{h_1,\dots,h_s,\clc in\}$. The edge $L^r$ of $F$ includes the fixed island $L_{J}^0$. 

Now the fact that $a(\HHo) = (\alpha_{\iota(i)}, \Lambda_0) = 0$ implies that $\psis(\Pin f_{\clc in} v_0) \in \ker \Ghom$ and hence that
\be A := \psis(\Pin{f_{h_s}} \,\Pin{f_{h_{s-1}}}\,\dots\, \Pin{f_{h_1}} ) \Pin{f_{\clc in}} v_0 \in \ker \Ghom. \nn\ee 
Once more we conclude by observing that $A = \sum_{\tilde F\in [F]} \tilde F|_J$. Indeed, this follows from repeatedly rewriting using the identity 
\be \Pin f_{h_t} \Pin X y = \Pin X\, \Pin f_{h_t} y + \sum_{k\in \ZT} \Pin{[ \f k{h_t}, X]} y \ee
for $t=1,\dots,s$ in turn, where $X$ is a commutator including $f_{\clc in}$ and $y$ is a monomial of the form $\Pin f_{h_*} \dots \Pin f_{h_*} v_0$ (compare Example \ref{fzex}).

\def\cprime{$'$}
\providecommand{\bysame}{\leavevmode\hbox to3em{\hrulefill}\thinspace}
\providecommand{\MR}{\relax\ifhmode\unskip\space\fi MR }
\providecommand{\MRhref}[2]{%
  \href{http://www.ams.org/mathscinet-getitem?mr=#1}{#2}
}
\providecommand{\href}[2]{#2}

\end{document}